\documentclass[reqno,a4paper]{amsart}
\usepackage{amsmath,amssymb}
\usepackage{slashed}
\usepackage{latexsym}
\usepackage[all]{xy}
\usepackage[latin1]{inputenc}
\usepackage{graphicx}
\usepackage{array}
\usepackage{url}
\usepackage{enumerate}
\usepackage{tikz}
\usetikzlibrary{positioning,shadows,backgrounds}

\usepackage[colorlinks=true]{hyperref}

\newlength{\arrowsize}
\newcommand{\bigarrowtipdef}{%
\pgfarrowsdeclare{biggertip}{biggertip}{
  \setlength{\arrowsize}{.4pt}
  \addtolength{\arrowsize}{.5\pgflinewidth}
  \pgfarrowsrightextend{0}
  \pgfarrowsleftextend{-5\arrowsize}
}{
  \setlength{\arrowsize}{0.4pt}
  \addtolength{\arrowsize}{.3\pgflinewidth}
  \pgfpathmoveto{\pgfpoint{-6\arrowsize}{5\arrowsize}}
  \pgfpathlineto{\pgfpointorigin}
  \pgfpathlineto{\pgfpoint{-6\arrowsize}{-5\arrowsize}}
  \pgfusepathqstroke
}%
}

\newcommand{\eg}{{e.g.}}
\newcommand{\ie}{{i.e.}\ }
\newcommand{\loccit}{{loc. cit.}}

\newcommand{\cO}{{\mathcal O}}
\newcommand{\cP}{{\mathcal P}}
\newcommand{\cV}{{\mathcal V}}

\newcommand{\bbA}{{\mathbb{A}}}
\newcommand{\bbN}{{\mathbb{N}}}
\newcommand{\bbP}{{\mathbb{P}}}
\newcommand{\bbZ}{{\mathbb{Z}}}

\newcommand{\half}{\frac{1}{2}}
\newcommand{\smallbullet}{{\scriptscriptstyle\bullet}}

\newcommand{\restr}[1]{_{|_{\scriptstyle #1}}}
\newcommand{\smallrestr}[1]{_{|_{#1}}}

\newcommand{\too}{\longrightarrow}
\newcommand{\tooby}[1]{\buildrel #1\over\longrightarrow}
\newcommand{\tooo}[1]{\mathop{\vcenter{\hbox to #1em{\hrulefill}}\kern-5pt\to}}
\newcommand{\ourfrac}[2]{\genfrac{}{}{0pt}{1}{#1}{#2}}
\newcommand{\Ourfrac}[2]{\genfrac{}{}{0pt}{0}{#1}{#2}}
\newcommand{\hook}{\hookrightarrow}
\newcommand{\oto}[1]{\overset{#1}\to}
\newcommand{\otoo}[1]{\overset{#1}\too}
\newcommand{\isoto}{\oto{\sim}}
\newcommand{\isotoo}{\otoo{\sim}}
\newcommand{\onto}{\mathop{\twoheadrightarrow}}
\newcommand{\into}{\mathop{\rightarrowtail}}

\DeclareMathOperator{\im}{\mathrm{im}} 
\DeclareMathOperator{\id}{\mathrm{id}} 

\DeclareMathOperator{\Wcoh}{W}
\DeclareMathOperator{\Wper}{W}
\newcommand{\Wtot}{\Wcoh^{\text{\rm tot}}}
\newcommand{\bord}{\partial} 
\DeclareMathOperator{\ext}{e} 
\newcommand{\Gm}{\mathbb{G}_{\text{\rm m}}}

\DeclareMathOperator{\Spec}{\mathrm{Spec}}    
\DeclareMathOperator{\Grass}{\mathrm{Gr}}    
\DeclareMathOperator{\Gal}{\mathbb{G}}    
\DeclareMathOperator{\rank}{\mathrm{rk}} 
\DeclareMathOperator{\Pic}{\mathrm{Pic}} 

\newcommand{\can}{\omega} 

\newcommand{\Bl}{B} 
\newcommand{\Exc}{E} 

\newcommand{\deframe}{$(d\!\times\!e)$-frame}
\newcommand{\multi}[1]{\underline{#1}} 
\newcommand{\row}{\rho}
\newcommand{\zero}{\zeta}

\newcommand{\inv}{^{-1}}
\newcommand{\rect}[2]{[#1\times #2]}
\newcommand{\genpt}{\delta} 
\newcommand{\vide}{{\slashed{\square}}} 

\newcommand{\twist}{t} 
\newcommand{\Twist}{T} 

\newcommand{\barres}{\bar\upsilon}
\newcommand{\barpush}{\bar\iota}
\newcommand{\barbord}{\bar\bord}

\newcommand{\Taut}{{\mathcal T}} 
\newcommand{\Flag}{{\mathcal Fl}} 
\newcommand{\Det}{\Delta} 
\newcommand{\gen}{\phi} 
\newcommand{\ff}{f} 
\newcommand{\iso}{\phi}
\newcommand{\subb}{\lhd}
\newcommand{\subbeq}{\subb}
\newcommand{\nsubb}{\dot{\ntriangleleft}}

\newcommand{\BS}{X} 
\newcommand{\CS}{Y} 

\newcommand{\lax}{\textrm{lax}}
\newcommand{\weq}{\mathop{\,\leftrightsquigarrow\,}} 
\newcommand{\SmPic}{{\mathcal S}} 
\newcommand{\dt}[1]{\mathchoice{\mathop{\cdot}\limits_{#1}}{\cdot_{#1}}{}{}} 

\newcommand{\alis}[1]{{#1}{}^{{\circlearrowleft}}}
\newcommand{\alto}{\leadsto}
\newcommand{\alKto}[1]{\,\ensuremath\raisebox{.5ex}{$\underset{#1}{\leadsto}$}\,}

\newcommand{\Set}{\mathcal{I}} 
\newcommand{\varS}{i} 
\newcommand{\FSet}{\mathcal{J}} 
\newcommand{\varF}{i} 
\newcommand{\SetZ}{\mathcal{I}} 
\newcommand{\varSZ}{i} 
\newcommand{\SetX}{\mathcal{J}} 
\newcommand{\varSX}{j} 
\newcommand{\SetU}{\mathcal{K}} 
\newcommand{\varSU}{k} 

\swapnumbers 
\theoremstyle{plain}
\newtheorem{theo}{Theorem}[section]
\newtheorem*{maintheo*}{Main Theorem}
\newtheorem{lemm}[theo]{Lemma}
\newtheorem{coro}[theo]{Corollary}
\newtheorem{prop}[theo]{Proposition}

\theoremstyle{definition}

\newtheorem{defi}[theo]{Definition}
\newtheorem{rema}[theo]{Remark}

\newtheorem{exam}[theo]{Example}
\newtheorem{nota}[theo]{Notation}
\newtheorem{conv}[theo]{Convention}
\newtheorem*{conv*}{Convention}

\title{Witt groups of Grassmann varieties}
\author{Paul Balmer and Baptiste Calm{\`e}s}

\address{Paul Balmer, Department of Mathematics, UCLA, Los Angeles, CA 90095-1555, USA}
\urladdr{http://www.math.ucla.edu/\raise-.5ex\hbox{~}balmer}

\address{Baptiste Calm\`es, Universit\'e d'Artois, Laboratoire de Math\'ematiques de Lens, France}
\urladdr{http://www.math.uni-bielefeld.de/\raise-.5ex\hbox{~}bcalmes}

\thanks{The first author is supported by NSF grant DMS-0969644.}

\date{\today}



\begin{document}
\bibliographystyle{amsplain}

\begin{abstract}
We compute the Witt groups of split Grassmann varieties, over any
regular base~$X$. The answer is that the total Witt group of the
Grassmannian is a free module over the total Witt ring of~$X$. We
provide an explicit basis for this free module, which is indexed by
a special class of Young diagrams, that we call \emph{even} Young
diagrams.
\end{abstract}

\maketitle

\tableofcontents

\section*{Introduction}
At first glance, it might be surprising for the non-specialist that
more than thirty years after the definition of the Witt group of a
scheme, by Knebusch~\cite{Knebusch77b}, the Witt group of such a
well-known variety as a Grassmannian has not been computed yet. This
is especially striking since analogous results for ordinary
cohomologies, for $K$-theory and for Chow groups, have been settled
for even longer. The goal of this article is to solve this problem
and explain what made it so hard in the first place.

\begin{maintheo*}[See Thm.\,\ref{main_thm}]
Let~$X$ be a regular noetherian and separated scheme over
$\bbZ[\half]$, of finite Krull dimension.
Let $0 < d < n$ be integers and let $\Grass_X(d,n)$
be the Grassmannian of $d$-dimensional subbundles of the trivial
$n$-dimensional vector bundle $\cV=\cO_X^n$ over~$X$. (More generally, we treat
any vector bundle $\cV$ admitting a complete flag of subbundles.)

Then the
total Witt group of~$\Grass_X(d,n)$ is a free graded module over the
total Witt group of~$X$ with an explicit basis indexed by so-called
``even'' Young diagrams. The basis element corresponding to an even
Young diagram is essentially the push-forward of the unit along the
inclusion of the corresponding Schubert variety. The cardinal of
this basis equals $2\cdot\frac{\displaystyle(d'+e')!}{\displaystyle d'!\cdot e'!}$
where $d'=\big[\frac{\displaystyle d}{\displaystyle 2}\big]$ and $e'=\big[\frac{\displaystyle n-d}{\displaystyle 2}\big]$.
\end{maintheo*}

Before explaining the statement in more detail, recall that the
Grothendieck group, or the Chow group, of $\Grass_X(d,n)$ would also
be free over that of~$X$ but with a basis indexed by \emph{all}
Young diagrams. We shall explain below why only some Young diagrams
``make it to the Witt group''.

\smallbreak

The \emph{total} Witt group refers to the sum of the Witt groups
$\Wper^i(X,L)$
$$\Wtot(X) = \bigoplus_{\stackrel{\scriptstyle i \,\in\, \bbZ/4}{\scriptstyle [L] \,\in\, \Pic(X)/2}} \Wcoh^i(X,L)$$
for all possible shifts $i\in\bbZ/4$ and all possible twists
$[L]\in\Pic(X)/2$ in the duality. Details about this total Witt
group, including the dependency on choosing $L$ in its class
$[L]\in\Pic(X)/2$, can be found in Section~\ref{tot-form_sec}. For
this introduction, let us keep things simple\,: The total Witt group
of $X$ wraps up all Witt groups of~$X$, for all possible shifts~$i$
and all twists~$L$.

For $X=\Spec(F)$, the spectrum of a field, the total Witt group
boils down to the classical Witt group~$\Wper(F)$ but even in that
case the above Theorem is new and the total Witt group of
$\Grass_{F}(d,n)$ involves non-trivial shifted and twisted Witt
groups. The result has a very round form when stated for total Witt
groups but Knebusch's classical unshifted Witt groups
$\Wper^0(\Grass_{X}(d,n),L)$ can be isolated, as well as the
unshifted and untwisted Witt group
$\Wper(\Grass_{X}(d,n))=\Wper^0(\Grass_{X}(d,n),\cO)$. Indeed, the
announced basis consists of homogeneous elements and we describe
below how to read their explicit shifts~$i$ and twists~$L$ directly
on the corresponding Young diagram. For instance, it is worth noting
that there are no new interesting antisymmetric forms in the Witt
groups of $\Grass_X(d,n)$, that is, except for those extended
from~$X$, see Corollary~\ref{ranks_cor}.

To describe our basis explicitly, we need to introduce \emph{even}
Young diagrams. We first consider ordinary Young diagrams sitting in the
upper left corner of a rectangle with $d$ rows and $e$ columns,
which we call the \emph{frame} of the diagram. See
Figure~\ref{framed_fig}.

\begin{figure}[!ht]
\begin{center}
\begin{tikzpicture}[y=-1cm]

\path[draw=black,thick,fill=black!50] (5.6,3.4) -- (6.4,3.4) -- (6.4,2.6) -- (7.6,2.6) -- (7.6,1.4) -- (8.4,1.4) -- (8.4,1) -- (5.6,1);

\draw[thick,black] (5.6,1) rectangle (8.8,3.8);
\draw[thick,arrows=stealth-stealth,black] (5.4,1) -- (5.4,3.8);
\draw[thick,arrows=stealth-stealth,black] (8.8,0.8) -- (5.6,0.8);
\path (5,2.5) node[text=black,anchor=base west] {$d$};
\path (7.1,0.7) node[text=black,anchor=base west] {$e$};
\path (6.4,1.9) node[text=black,anchor=base west] {$\Lambda$};

\end{tikzpicture}
\end{center}
\caption{Young diagram $\Lambda$ in \deframe}
\label{framed_fig}%
\end{figure}
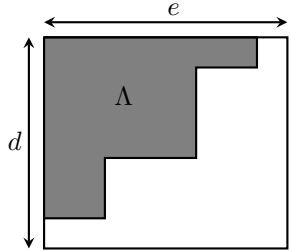

We say that such a framed Young diagram~$\Lambda$ is \emph{even} if
all the segments of the boundary of~$\Lambda$ which are strictly
inside the frame have even length. That is, we allow $\Lambda$ to
have odd-length segments on its boundary \emph{only where it touches
the outside frame}. See Figure~\ref{even_fig} for examples. (In
Figures~\ref{G44_fig}, \ref{G45_fig} and~\ref{G55_fig} we further
give all even diagram in \deframe\ for $(d,e)=(4,4)$, $(4,5)$ and
$(5,5)$, respectively.)

\begin{figure}[!ht]
\begin{center}
\begin{tikzpicture}[y=-1cm]

\path[draw=black,thick,fill=black!50] (5.4,3.4) -- (6.2,3.4) -- (6.2,2.6) -- (7,2.6) -- (7,1) -- (5.4,1);
\path[draw=black,thick,fill=black!50] (8.4,3) -- (10,3) -- (10,2.2) -- (10.8,2.2) -- (10.8,1) -- (8.4,1);
\path[draw=black,thick,fill=black!50] (3.2,3.4) -- (3.6,3.4) -- (3.6,2.6) -- (4.4,2.6) -- (4.4,1) -- (3.2,1);

\draw[thick,arrows=stealth-stealth,black] (7,2.5) -- (6.2,2.5);
\draw[thick,arrows=stealth-stealth,black] (6.2,3.3) -- (5.4,3.3);
\draw[thick,arrows=stealth-stealth,black] (6.3,2.6) -- (6.3,3.4);
\draw[thick,arrows=stealth-stealth,black] (7.1,1) -- (7.1,2.6);
\draw[thick,black] (5.4,1) rectangle (7.8,3.8);
\path (7.1,2) node[text=black,anchor=base west] {$4$};
\path (6.3,3.1) node[text=black,anchor=base west] {$2$};
\path (5.6,3.2) node[text=black,anchor=base west] {$2$};
\path (6.4,2.4) node[text=black,anchor=base west] {$2$};
\path (6.3,3.1) node[text=black,anchor=base west] {$2$};
\draw[thick,black] (8.4,1) rectangle (10.8,3.4);
\draw[thick,arrows=stealth-stealth,black] (10.8,2.1) -- (10,2.1);
\draw[thick,arrows=stealth-stealth,black] (10.1,2.2) -- (10.1,3);
\draw[thick,arrows=stealth-stealth,black] (10,2.9) -- (8.4,2.9);
\path (10.2,2) node[text=black,anchor=base west] {$2$};
\path (10.1,2.7) node[text=black,anchor=base west] {$2$};
\path (9.1,2.8) node[text=black,anchor=base west] {$4$};
\path[draw=black,thick,fill=black!50] (11.4,1) rectangle (13,3.8);
\draw[thick,black] (1,1) rectangle (2.6,3.8);
\draw[thick,arrows=stealth-stealth,black] (4.5,1) -- (4.5,2.6);
\draw[thick,arrows=stealth-stealth,black] (3.7,2.6) -- (3.7,3.4);
\draw[thick,arrows=stealth-stealth,black] (4.4,2.5) -- (3.6,2.5);
\draw[thick,black] (3.2,1) rectangle (4.8,3.4);
\path (3.7,3.1) node[text=black,anchor=base west] {$2$};
\path (3.8,2.4) node[text=black,anchor=base west] {$2$};
\path (4.45,2) node[text=black,anchor=base west] {$4$};

\end{tikzpicture}%
\end{center}
\caption{Five examples of even Young diagrams}
\label{even_fig}%
\end{figure}
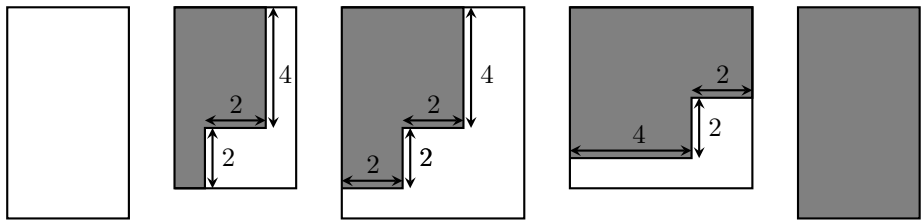

We shall see that basis elements of $\Wtot\big(\Grass_X(d,n)\big)$
are in bijective correspondence with even Young diagrams in
\deframe, for $e:=n-d$. Moreover, as explained in
Section~\ref{gen-Witt_sec}, the Witt class $\gen_{d,e}(\Lambda)$
corresponding to such a diagram $\Lambda$ lives in the Witt group
$\Wper^i\big(\Grass_X(d,n),L\big)$ for the shift
$i=|\Lambda|\in\bbZ/4$ equal to the \emph{area} of the diagram and
for the twist $[L]\in\Pic(\Grass_X(d,n))/2$ equal to the class of
$\twist(\Lambda)\cdot\Det$ where $\Det$ is the determinant of the
tautological bundle and where the integer $\twist(\Lambda)$ is half
the \emph{perimeter} of the diagram~$\Lambda$ (see
Figure~\ref{twist_fig} in Section~\ref{gen-Witt_sec}). More
generally, when $\cV$ is not free but admits a complete flag of
subbundles, the twist of $\gen_{d,e}(\Lambda)$ also involves a
multiple of the determinant of~$\cV$, in the direct summand of
$\Pic(\Grass_X(d,n))/2$ coming from $\Pic(X)/2$.

\smallbreak

Let us put our result in perspective. In its modern form, see for
instance Laksov~\cite{Laksov72}, the computation of the cohomology
(or $K$-theory, or Chow groups) of a cellular variety uses
essentially only localization long exact sequences and homotopy
invariance, applied to the classical cellular decomposition of the
Grassmannian. It took some time for Witt groups to reach the
necessary cohomological maturity. Indeed, the localization long
exact sequence could only be established by defining first ``higher"
or ``shifted" Witt groups, as was done in~\cite{Balmer00}
and~\cite{Balmer01} by the first author, using the framework
of triangulated categories. Then, homotopy invariance was proved by
Gille~\cite{Gille03b}. To be more precise, one actually also needs
some form of \emph{d\'evissage}, that allows us to compare the
theory with supports on a closed subset~$Z$ with the theory of~$Z$
itself. This piece of theory was a hard nut to crack because
dualities came in the way but is now well understood. (We return to
these dualities below.)

This construction of the cohomological machines accounts for most of
the delay that Witt groups accumulated in comparison to sister
theories. However, this is not the end of the story. What makes the
computation of the Witt groups of $\Grass_X(d,n)$ harder
than that of classical theories is the following very interesting
phenomenon. The classical computation proceeds by induction, using
the closed embedding of a smaller Grassmannian $\Grass_X(d,n-1)$
inside $\Grass_X(d,n)$, whose open complement
$U=\Grass_X(d,n)\smallsetminus \Grass_X(d,n-1)$ is an affine bundle
over another smaller Grassmannian, $\Grass_X(d-1,n)$. (See more in
Section~\ref{cell-dec_sec}.) Now the true miracle in the classical
proof is that the restriction homomorphism, from the big
Grassmannian $\Grass_X(d,n)$ to the open~$U$, is \emph{split
surjective}. This holds more precisely for any \emph{oriented}
cohomology theory in the sense of Levine-Morel~\cite{Levine07} or
Panin~\cite{Panin04}. In other words, there are no real localization
\emph{long} exact sequences involved in the classical proof, no real
connecting homomorphisms, just split short exact sequences.

The interesting point is that this miracle ceases to happen for Witt
groups\,: For a general
$(i,[L])\in\bbZ/4\times\Pic(\Grass_X(d,n))/2$, the restriction
homomorphism
$$
\Wper^i(\Grass_X(d,n),L)\too\Wper^i(U,L)
$$
does not admit a section. Worse, it is not even surjective\,! That
is, the connecting homomorphism in the localization long exact
sequence is not zero in general. This makes the Witt group
computation not a mere technical adaptation of the classical methods
but a completely different story\,: One needs to \emph{compute} a
connecting homomorphism in geometric enough terms, in order to
follow what happens on the given varieties. This very last pocket of
Witt resistance has been cleared in the companion
article~\cite{Balmer09} and Grassmannians are the first known
examples where this phenomenon can be described explicitly.

\smallbreak

Let us now comment on the organization of the paper.
Sections~\ref{Grass_sec} and~\ref{Young_sec} contain preparatory
material on Grassmann varieties, desingularizations of Schubert varieties and even Young diagrams.

Section~\ref{tot-form_sec} formalizes the use of total Witt groups,
since the group $\Wcoh^*(X,L)$ does not truly depend only on the
class $[L]\in\Pic(X)/2$ but really on the representative~$L$ in this
class. This is an old problem, rooting back to
Knebusch~\cite{Knebusch77b}, to which we proposed a solution
in~\cite{Balmer11}. We recall this formalism in
Section~\ref{tot-form_sec} with emphasis on the case of
Grassmannians. Until then, this Introduction should therefore be
read \textsl{cum grano salis}.

Our generators of $\Wtot\big(\Grass_X(d,n)\big)$ are defined in
Section~\ref{gen-Witt_sec} as push-forwards of the unit forms of
certain desingularized Schubert varieties. The reader should observe
that pushing the unit form is not always possible, due to the
presence of line bundles in the definition of the push-forward.
Indeed, for a proper morphism $f:\bar Y \to Y$ of constant relative
dimension $\dim(f)$, between regular noetherian schemes $\bar Y$ and
$Y$ (think $\dim f=\dim \bar Y-\dim Y$), the push-forward
along~$f$ is defined between the following Witt groups\,:
\begin{equation}
\label{push-forward_eq}%
\Wcoh^{i+\dim f}(\bar Y,\can_{f} \otimes f^*L) \overset{f_*}\too
\Wcoh^{i}(Y,L)\,,
\end{equation}
where the special line bundle $\can_f$ on the left is the
\emph{relative line bundle}. So, if unfortunately $\cO_{\bar Y}$ is
not (up to squares) of the form $\can_{f} \otimes f^*L$ for any line
bundle $L$ over~$Y$, then one simply \emph{cannot} push-forward the
unit form of~$\bar Y$, which lives in~$\Wcoh^0(\bar Y,\cO_{\bar
Y})$. This is why we start Section~\ref{gen-Witt_sec} by discussing
the ``parity'' of the relevant canonical bundles~$\can_{f}$.
Although somewhat heavy, these computations are elementary and are
all based on a repeated application of the computation of the
relative canonical bundle of a Grassmann bundle
(Prop.\,\ref{rel-can-Grass_prop}). The condition for a Young diagram
to be even implies the existence of such a push-forward for the unit
of the desingularized Schubert cell into the Grassmannian. Actually,
we could push-forward the unit form for more Schubert cells but
these additional generators would be redundant. The even Young
diagrams are chosen so that the corresponding forms are also
linearly independent.

Then, in Section~\ref{cell-dec_sec}, we recall the classical
relative cellular structure of the Grassmann varieties. In
Section~\ref{main_sec}, we compute how our candidate-generators
behave under the morphisms in the long exact sequence, especially
under the connecting homomorphism, which is most of the time not
zero (Cor.\,\ref{bord-non-zero_coro}). The proof of the main theorem
(Thm.\,\ref{main_thm}) then follows by induction on the rank of the
vector bundle~$\cV$. The last section contains corollaries and
examples.

\smallbreak

This article is written in the language of
``functors of points", which means that we describe schemes in terms
of their points (which are here flags) and morphisms of schemes as
how they act on those points. This method is completely rigorous in this
case. The original source is~\cite{SGA3} and
we also refer the reader to~\cite[\S\,I.1]{DemazureGabriel70} and \cite[Part 2]{Karpenko00b} for general
considerations on this subject. This language is customary when
dealing with flag varieties, see for instance~\cite{Laksov72} in
which it is used for the computation of Chow groups of Grassmann
varieties.

\section{Combinatorics of Grassmann and flag varieties}
\label{Grass_sec}%


We recall elementary facts about Grassmann varieties and desingularizations of Schubert cells. We also provide the necessary
material about canonical bundles to treat the push-forward
homomorphisms for Witt groups in Section~\ref{gen-Witt_sec}.

\begin{defi}\label{subb_def}%
A \emph{subbundle} $\cP\subb \cV$ of a vector
bundle $\cV$ over a scheme~$X$ is an $\cO_X$-submodule which is
locally a direct summand, \ie $\cP$ and $\cV/\cP$ are vector bundles.
\end{defi}

\begin{defi}
\label{Grass_def}%
Let $\cV$ be a vector bundle of rank $n>0$ over a
scheme~$X$ and let $d$ be an integer $0\leqslant d\leqslant n$. We
denote by $\Grass_X (d,\cV)$ the Grassmann bundle over~$X$
parameterizing the subbundles of rank~$d$ of~$\cV$. In the language
of functors of points, it means that for any morphism $f:\Spec(R)\to
X$, the set $\Grass_X(d,\cV)(R)$ consists of the $R$-submodules
$P\subb \cV(R)=f^*(\cV)$ which are direct summands of rank~$d$.

The scheme $\Grass_X(d,\cV)$ comes equipped with a smooth structural
morphism $\pi:\Grass_X(d,\cV) \to X$ and a tautological bundle
$\Taut_d=\Taut_d^{\Grass_X(d,\cV)}$ of rank~$d$.
\end{defi}

\begin{prop}
\label{Pic-Grass_prop}%
The scheme $\Grass_X(d,\cV)$ is smooth over~$X$ of relative
dimension $d(n-d)$. For $0<d<n$, the Picard group of
$\Grass_X(d,\cV)$ is given by
$$
\begin{array}{ccc}
\Pic(X) \oplus \bbZ & \cong & \Pic(\Grass_X(d,\cV))
\\
(\ell,m) & \mapsto & \pi^*(\ell)\cdot [\det(\Taut_d)]^m.
\end{array}
$$
In case $d=0$ or $d=n$, the morphism $\pi:\Grass_X(d,\cV)\to X$ is
the identity.
\end{prop}

\begin{proof}
The Picard group of a regular scheme coincides with its Chow group
$CH^1$, which is computed in~\cite{Laksov72} for Grassmannians\,: see
Theorem~16 for the case where $X$ is a field, and §13 to work over a
regular base $X$. Using the Pl\"ucker embedding, one checks that the
generator in \loccit{} is indeed $[\det(\Taut_d)]$.
\end{proof}

Let $\Det_d$ denote the class of $\det(\Taut_d)$ in $\Pic\big(\Grass_X(d,\cV)\big)/2$.

\begin{coro}
\label{Pic-Grass_coro}%
If $0<d<n$, we have a natural identification
$$\Pic\big(\Grass_X(d,\cV)\big)/2\ \cong\ \Pic(X)/2\ \oplus\ \bbZ/2\cdot\Det_d.$$
\end{coro}

\begin{prop}
\label{rel-can-Grass_prop}%
The class of the relative canonical bundle
$\can_{\Grass_X(d,\cV)/X}$ of the projection $\pi:\Grass_X (d,\cV)
\to X$ is $[\can_{\Grass_X(d,\cV)/X}]=[\det \cV]^{-d}\cdot \Det_d^n$
in $\Pic(\Grass_X (d,\cV))$. In particular, if $\cV=\cO_X^n$ is
trivial, $[\can_{\Grass_X(d,\cV)/X}]=\Det_d^n$.
\end{prop}
\begin{proof}
The morphism $\pi$ is smooth, so $\can_{\Grass_X(d,\cV)/X}$ is the determinant (highest exterior power) of the relative cotangent bundle of $\pi$. This cotangent bundle is the tautological bundle tensored by the dual of the tautological
quotient bundle (see~\cite[Appendix B.5.8]{Fulton98}). Taking the determinant, we get the result.
\end{proof}

\bigbreak \centerline{*\ *\ *}\goodbreak\bigbreak

We now extend the previous results from Grassmannians to some flag varieties.

\begin{defi}
\label{Flag_def}%
Let $k\geqslant1$ and $(\multi{d},\multi{e})$
be a pair of $k$-tuples of non-negative integers $\multi{d}=(d_1,\ldots,d_k)$ and
$\multi{e}=(e_1,\ldots,e_k)$ satisfying
\begin{equation}\label{pre-de_eq}%
0< d_1 < \cdots <d_k \qquad\text{and}\qquad
e_1+d_1\leqslant \cdots\leqslant e_k+d_k\,.
\end{equation}
(The second condition holds in particular if we have $e_1\leqslant\cdots\leqslant e_k$.) Consider a flag
\begin{equation}
\label{partial_eq}%
\cV_{d_1+e_1} \subb \cdots \subb \cV_{d_i+e_i} \subb \cdots \subb \cV_{d_k+e_k}
\end{equation}
of vector bundles over~$X$, where $\subb$ indicates subbundles in the strong sense of Definition~\ref{subb_def} and where the rank is given by the index: $\rank_X(\cV_r)=r$.

We associate to
this data the scheme $\Flag_{X}(\multi{d},\multi{e},\cV_{\smallbullet})$ over~$X$,
which parameterizes the flags of vector bundles $\cP_{d_1}\subb
\cP_{d_2}\subb\cdots\subb \cP_{d_k}$ such that $\rank \cP_{d_j}=d_j$
and $\cP_{d_j} \subb \cV_{d_j + e_j}$. As a functor of points, this gives for any morphism $f:Y\to X$
\begin{equation}\label{Flag_eq}%
\Flag_{X}(\multi{d},\multi{e},\cV_\smallbullet)(Y) := \left\{
\begin{array}{c}
\xymatrix@R=1.5ex@C=1.5ex{ 0 \ar@{}[r]|-{\subb} & \cP_{d_1}
\ar@{}[d]|{\rotatebox{270}{$\subbeq$}}^*!/^.5ex/{\labelstyle e_1}
\ar@{}[r]|{\subb} & \cP_{d_2}
\ar@{}[d]|{\rotatebox{270}{$\subbeq$}}^*!/^.5ex/{\labelstyle e_2}
\ar@{}[r]|{\subb} & \cdots \ar@{}[r]|{\subb} & \cP_{d_k}
\ar@{}[d]|{\rotatebox{270}{$\subbeq$}}^*!/^.5ex/{\labelstyle e_k}
\\
0 \ar@{}[r]|-{\subb} & f^*\cV_{d_1 + e_1} \ar@{}[r]|-{\subb}
& f^*\cV_{d_2 + e_2} \ar@{}[r]|-{\subb} & \cdots \ar@{}[r]|-{\subb}
& f^*\cV_{d_k + e_k} }
\end{array}
\right\}\,,
\end{equation}
where all $\cP_{d_i}$ are vector bundles over~$Y$ of rank~$d_i$ such that all inclusions are subbundles in the sense of Definition~\ref{subb_def}. The integers along inclusions indicate codimensions.
Following general practice, we shall drop the mention of $f^*$ in the sequel. Moreover, to avoid cumbersome notation, unless the original flag~\eqref{partial_eq} varies, we drop the mention of $\cV_\smallbullet$ from the notation\,: $\Flag_X(\multi{d},\multi{e})=\Flag_{X}(\multi{d},\multi{e},\cV_\smallbullet)$.
\end{defi}

\begin{exam}
\label{k=1_exa}%
For $k=1$, the scheme $\Flag_X(\multi{d},\multi{e})$ is simply $\Grass_X(d_1,\cV_{d_1+e_1})$.
\end{exam}

\begin{rema}
\label{morph_rem}%
For any choice $J$ of $k'$ indices among $\{1,\ldots,k\}$, one can consider the pair of $k'$-tuples $(\multi{d}',\multi{e}')$ obtained from $(\multi{d},\multi{e})$ by keeping $d_i$ and $e_i$ only for indices~$i\in J$. There is a natural morphism $\Flag_X(\multi{d},\multi{e})\to\Flag_X(\multi{d}',\multi{e}')$ over~$X$, obtained by dropping the $\cP_{d_j}$ for those indices~$j$ which are not in the chosen~$J$.

Furthermore, for any vector bundle $\cV$ such that $\cV_{d_k+e_k}\subb\cV$\,, there is a natural morphism $\ff_{\multi{d},\multi{e},\cV}$ of schemes over~$X$ as follows\,:
\begin{equation}
\label{pre-ff_eq}%
\begin{array}{rccccc}
\ff_{\multi{d},\multi{e},\cV}\,: &\Flag_X(\multi{d},\multi{e}) & \too &
\Grass_X(d,\cV_{d_k+e_k}) & \hookrightarrow & \Grass_X(d,\cV)
\\
&(\cP_{d_1},\ldots, \cP_{d_k}) & \longmapsto & \cP_{d_k} & \longmapsto & \cP_{d_k}\,,
\end{array}
\end{equation}
where the first morphism is as above and the second is a closed immersion.
\end{rema}

\begin{defi}
\label{tavtologic_def}%
The scheme $\Flag_{X}(\multi{d},\multi{e})$
is equipped with \emph{tautological bundles} $\Taut_{d_i}$, $1 \leqslant i
\leqslant k$, of rank $d_i$, whose determinant classes are denoted by
$\Det_{d_i}:=\det(\Taut_{d_i})$. The stalk of $\Taut_{d_i}$ at a point $(\cP_{d_1},
\ldots, \cP_{d_k})$ is $\cP_{d_i}$. In ambiguous cases, the full notation for
$\Taut_{d_i}$ would be $\Taut_{d_i}^{\Flag_X(\multi{d},\multi{e},\cV_\smallbullet)}$.
\end{defi}

\begin{rema}\label{zero-Det_rem}%
If $e_i=0$ then the vector bundles $\Taut_{d_i}=\cV_{d_i}$ and
$\Det_{d_i}=[\det\cV_{d_i}]$ are both extended from~$X$.
\end{rema}

\begin{lemm}
\label{rel_lem}%
Let $k\geqslant2$ and let $(\multi{d},\multi{e})$ be a pair of
$k$-tuples satisfying~\eqref{pre-de_eq}. Let $\cV_\smallbullet$ be a
flag as in~\eqref{partial_eq}. Define the $(k-1)$-tuples $\multi{d}\restr{k-1}$ and
$\multi{e}\restr{k-1}$ as the restrictions of $\multi{d}$ and
$\multi{e}$ to the first $k-1$ entries. Consider the scheme
$$
Y:=\Flag_X(\,\multi{d}\restr{k-1}\,,\,\multi{e}\restr{k-1}\,,\,\cV_\smallbullet\,)\,,
$$
which only ``uses'' the first $k-1$ bundles $\cV_{d_1+e_1}\subb\cdots\subb\cV_{d_{k-1}+e_{k-1}}$. Consider the pull-back to~$Y$ of the remaining bundle, still denoted $\cV_{d_{k}+e_{k}}$\,. Observe that $\Taut_{d_{k-1}}^Y\subb\cV_{d_{k}+e_{k}}$ and consider the quotient bundle
$$
\tilde\cV:=\cV_{d_k+e_k}/\Taut_{d_{k-1}}^Y
$$
over~$Y$. It has rank $d_k-d_{k-1}+e_k$. We then have a canonical isomorphism of schemes over~$Y$ (hence over~$X$)\,:
\begin{equation}
\label{rel-Flag_eq}%
\Flag_X(\multi{d},\multi{e},\cV_\smallbullet)\quad\cong\quad\Grass_Y\big(d_k-d_{k-1},\tilde\cV\big)\,.
\end{equation}
Under this identification, we have $\Taut^{\Flag(\multi{d},\multi{e},\cV_\smallbullet)}_{d_i}=\Taut^Y_{d_i}$ for all $1\leqslant i\leqslant k-1$ and
\begin{equation}
\label{conv-Taut_eq}%
\Taut_{d_{k}}^{\Flag_{X}(\multi{d},\multi{e},\cV_\smallbullet)}/\Taut_{d_{k-1}}^{\Flag_{X}(\multi{d},\multi{e},\cV_\smallbullet)}
=
\Taut_{d_k-d_{k-1}}^{\Grass_{Y}(d_k-d_{k-1},\tilde\cV)}\,.
\end{equation}
\end{lemm}

\begin{proof}
This simply amounts to the bijective correspondence between a flag $\cP_{d_1}\subb\cdots\subb\cP_{d_{k-1}}\subb\cP_{d_{k}}$ satisfying $\cP_{d_i}\subb\cV_{d_i+e_i}$ for all $1\leqslant i\leqslant k$ and the following data\,:
\begin{enumerate}
\item[(a)]
the beginning of this flag $\cP_{d_1}\subb\cdots\subb\cP_{d_{k-1}}$ satisfying $\cP_{d_i}\subb\cV_{d_i+e_i}$ for all $1\leqslant i\leqslant k-1$,
\item[(b)]
the bundle $\cP_{d_{k}}$ such that $\cP_{d_{k-1}}\subb\cP_{d_{k}}\subb\cV_{d_{k}+e_{k}}$
\end{enumerate}
and to observe that (b) is equivalent to a subbundle $\tilde\cP\subb\cV_{d_k+e_k}/\cP_{d_{k-1}}$ of rank $d_k-d_{k-1}$, where $\tilde\cP:=\cP_{d_{k}}/\cP_{d_{k-1}}$. Details are left to the reader.
\end{proof}

\begin{conv}
\label{zero_conv}%
When using $k$-tuples $\multi{d}=(d_1,\ldots,d_k)$, it will unify several formulas to simply define $d_0=0$.
\end{conv}

\begin{prop}
\label{Pic-Flag_prop}%
Let $\multi{d}$ and~$\multi{e}$ be two $k$-tuples as in~\eqref{pre-de_eq} and $\cV_\smallbullet$ be a flag as in~\eqref{partial_eq}.
Then $\Flag_X(\multi{d},\multi{e})$ is smooth over~$X$ of relative
dimension~$\sum_{i=1}^{k}(d_{i}-d_{i-1})\,e_i$.
The Picard group of $\Flag_X(\multi{d},\multi{e})$ is generated by
$\Pic(X)$ and the ``new'' classes $\Det_{d_i}$\,:
\begin{equation} \label{Pic-Flag_eq}%
\Pic(\Flag_X(\multi{d},\multi{e})) \cong {\Pic}(X)\ \oplus\!
\bigoplus_{\ourfrac{1\leqslant i\leqslant k}{\text{s.t.\ }e_i\neq0}} \bbZ \,\Det_{d_i}\,.
\end{equation}
The class of the relative canonical bundle
$\can_{\Flag_X(\multi{d},\multi{e})/X}$ is given by the
formula
\begin{equation}\label{can-abs-Flag_eq}%
[\can_{\Flag_X(\multi{d},\multi{e})/X}] =
\prod\limits_{i=1}^k \,[\det\cV_{d_i+e_i}]^{-d_i+d_{i-1}} \cdot
\prod\limits_{i=1}^{k-1}\,\Det_{d_{i}}^{d_{i}-d_{i-1}
+ e_{i}-e_{i+1}}
\cdot \,\Det_{d_k}^{d_k-d_{k-1} + e_k}\,,
\end{equation}
where $\Det_{d_i}=[\det\cV_{d_i}]$ if $e_i=0$ by Remark~\ref{zero-Det_rem} and where we use Convention~\ref{zero_conv}. In particular, for $k=1$, this reads $[\can_{\Flag_X(\multi{d},\multi{e})/X}]=[\det\cV_{d_1+e_1}]^{-d_1}\cdot\Det_{d_1}^{d_1+e_1}$.

\end{prop}

\begin{proof}
By induction on~$k$. The case $k=1$ is that of a Grassmannian over~$X$ (Example~\ref{k=1_exa}) so the result follows from Propositions~\ref{Pic-Grass_prop} and~\ref{rel-can-Grass_prop}.

Let now $k\geqslant 2$. Consider $Y=\Flag_{X}(\multi{d}|_{k-1},\multi{e}|_{k-1},\cV_\smallbullet)$ and the bundle $\tilde\cV=\cV_{d_k+e_k}/\Taut_{d_{k-1}}$ over~$Y$, as in Lemma~\ref{rel_lem}.
Recall that $\rank_Y(\tilde\cV)=d_k-d_{k-1}+e_k$, which is always strictly positive ($d_k>d_{k-1}$) and which is bigger than or
equal to $d_k-d_{k-1}$ with equality if and only if $e_k=0$.
Equation~\eqref{rel-Flag_eq} and Propositions~\ref{Pic-Grass_prop} and~\ref{rel-can-Grass_prop} immediately give smoothness, the formula for the relative dimension and that for the Picard group~\eqref{Pic-Flag_eq}.
Finally, to prove~\eqref{can-abs-Flag_eq}, observe that
\begin{align*}
[\can_{\Flag_{X}(\multi{d},\multi{e})/Y}]
& = [\can_{\Grass_{Y}(d_k-d_{k-1},\tilde\cV)\,/\,Y}]
\\
& =
[\det\tilde\cV]^{-d_k+d_{k-1}} \cdot \left(\Det_{d_k-d_{k-1}}^{\Grass_{Y}(d_k-d_{k-1},\tilde\cV)}\right)^{\rank(\tilde\cV)}
\\
& = [\det\tilde\cV]^{-d_k+d_{k-1}} \cdot
\left(\Det_{d_k}^{\Flag_X(\multi{d},\multi{e})}\right)^{\rank(\tilde\cV)}\cdot\left(\Det_{d_{k-1}}^{\Flag_X(\multi{d},\multi{e})}\right)^{-\rank(\tilde\cV)}
\\
&
=[\det\cV_{d_k+e_k}]^{-d_k+d_{k-1}}\cdot\,\Det_{d_{k-1}}^{-e_k}
\cdot\,\Det_{d_{k}}^{d_k-d_{k-1}+e_k}\,.
\end{align*}
The first equality uses~\eqref{rel-Flag_eq}, the second comes from
Proposition~\ref{rel-can-Grass_prop} and the third from~\eqref{conv-Taut_eq}. The last equality is a direct
computation (in which we drop the mention of $\Flag_X(\multi{d},\multi{e})$ for readability). By induction hypothesis, we get $[\can_{Y/X}]$ from Equation~\eqref{can-abs-Flag_eq} for $k-1$, that is, for the flag variety~$Y$. Since $[\can_{\Flag_{X}(\multi{d},\multi{e})/X}]=[\can_{\Flag_{X}(\multi{d},\multi{e})/Y}]\cdot[\can_{Y/X}]$ over~$\Flag_X(\multi{d},\multi{e})$, we get~\eqref{can-abs-Flag_eq} for~$k$ by adding the above.
\end{proof}

\begin{coro}
\label{can-Flag_cor}%
Let $\multi{d}$ and~$\multi{e}$ be two $k$-tuples as
in~\eqref{pre-de_eq} and $\cV_\smallbullet$ be a flag as
in~\eqref{partial_eq}. Let $\cV$ be a vector bundle of rank~$d+e$
such that $\cV_{d_k+e_k}\subb\cV$\,. The class in
$\Pic(\Flag_X(\multi{d},\multi{e}))$ of the relative
canonical bundle for the morphism
$\ff_{\multi{d},\multi{e},\cV}:\Flag_X(\multi{d},\multi{e})\to\Grass_X(d,\cV)$
of~\eqref{pre-ff_eq} is given by
\begin{equation}
\label{can-Flag_eq}%
\begin{array}{rl}
[\can_{\Flag_X(\multi{d},\multi{e})/\Grass_X(d,\cV)}] & =\
\prod\limits_{i=1}^{k}\,[\det\cV_{d_i+e_i}]^{-d_i+d_{i-1}} \cdot \,[\det\cV]^{d_k}\cdot
\\
\cdot \prod\limits_{i=1}^{k-1} & \,\Det_{d_{i}}^{d_{i}-d_{i-1}+e_{i}-e_{i+1}} \cdot \,\Det_{d_k}^{-d_{k-1} + e_k - e}\,,
\end{array}
\end{equation}
where $\Det_{d_i}=[\det\cV_{d_i}]$ if $e_i=0$ by Remark~\ref{zero-Det_rem} and where we use Convention~\ref{zero_conv}. For $k=1$, this reads $[\can_{\Flag_X(\multi{d},\multi{e})/\Grass_X(d,\cV)}] =\,[\det(\cV/\cV_{d_1+e_1})]^{d_1}\cdot\,\Det_{d_1}^{e_1-e}$.
\end{coro}

\begin{proof}
Subtract
$(\ff_{\multi{d},\multi{e},\cV})^*\,[\can_{\Grass_X(d,\cV)/X}]=\,[\det\cV]^{-d_k}\cdot \,\Det_{d_k}^{d_k+e}$
(Prop.~\ref{rel-can-Grass_prop}) from the bundle
$[\can_{\Flag_X(\multi{d},\multi{e})/X}]$ given
in~\eqref{can-abs-Flag_eq}.
\end{proof}

\begin{rema}
When $\cV_{\smallbullet}=\cO_X^\smallbullet$, all the formulas are simpler, since all the $[\det \cV_i]$ are trivial. This applies in particular when $X=\Spec(R)$ for a local ring $R$.
\end{rema}


\section{Even Young diagrams}
\label{Young_sec}%


We introduce \emph{even} Young diagrams that will parameterize the basis of the total Witt group of the Grassmann variety, to be constructed in Section~\ref{gen-Witt_sec}.

\begin{defi}
\label{Young_def}%
Let $d,e\geqslant1$. We shall call \emph{Young
diagram in \deframe}, or simply \emph{$(d,e)$-diagram},
any $d$-tuple $\Lambda=(\Lambda_1,\Lambda_2,\ldots,\Lambda_d)$ of
integers such that:
$$
e \geqslant \Lambda_1 \geqslant \Lambda_2 \geqslant \ldots \geqslant \Lambda_d \geqslant 0.
$$
See Figure~\ref{framed_fig} in the Introduction. The \emph{area} of $\Lambda$ is
$|\Lambda|=\Lambda_1+\Lambda_2+\ldots+\Lambda_d$. These
$(d,e)$-diagrams are just ordinary Young diagrams displayed in the
upper left corner of a rectangle with $d$ rows and $e$ columns,
possibly leaving empty rows below and empty columns to the right of
the Young diagram. So, an ordinary Young diagram with $\rho$
rows and $\gamma$ columns defines a $(d,e)$-diagram
for any $d\geqslant \rho$ and $e\geqslant \gamma$.
\end{defi}

\begin{nota}
\label{empty-full_nota}%
The empty diagram $(0,\ldots,0)\in\bbN^d$ is denoted by $\vide$ and
the full $(d\times e)$-rectangle $(e,\ldots,e)\in\bbN^d$ by
$\rect{d}{e}$.
\end{nota}

\begin{defi}
\label{de_def}%
Let $d,e\geqslant1$ and let $\Lambda$ be a Young diagram in
\deframe. The decreasing sequence
$\Lambda_1\geqslant\Lambda_2\geqslant\ldots\geqslant\Lambda_d$ can be written in a
unique way as a series of equalities and strict inequalities\,:
\begin{equation}
\label{dk_eq}%
\underbrace{\Lambda_1=\cdots=\Lambda_{d_1}}_{d_1\text{
terms}}\ >\
\underbrace{\Lambda_{d_1+1}=\ldots=\Lambda_{d_2}}_{d_2-d_1\text{
terms}}\ >\ \cdots\ >\
\underbrace{\Lambda_{d_{k-1}+1}=\cdots=\Lambda_{d_k}}_{d_k-d_{k-1}\text{
terms}}=\Lambda_d\,.
\end{equation}
Note that $d_k=d$. The integers $k\geqslant1$ and $0<d_1<\ldots<d_k$
depend on $\Lambda$. If we need to stress this we shall write
$k=k(\Lambda)$ and $d_i=d_i(\Lambda)$ for $1\leqslant i\leqslant
k(\Lambda)$.

For fixed $d$ and $e$, there is a bijection
(pictured on Figure~\ref{partition_fig}) between the Young diagrams
$\Lambda$ in
\deframe\ and pairs of $k$-tuples of integers
\begin{equation}
\label{de_eq}%
\begin{array}{lcl}
\multi{d}=(d_1,\ldots,d_k) & \text{such that} & 0< d_1 < \cdots <d_k
= d
\\
\multi{e}=(e_1,\ldots,e_k) & \text{such that} & 0\leqslant e_1 <
\cdots < e_k \leqslant e
\end{array}
\end{equation}
with $1\leqslant k\leqslant d$. The integers $k=k(\Lambda)$ and
$d_i=d_i(\Lambda)$ are the above ones and we set
$e_i=e_i(\Lambda):=e-\Lambda_{d_i}$ for all $i=1,\ldots,k$. The converse
construction is obvious.
\end{defi}

\begin{figure}[!ht]
\begin{center}
\begin{tikzpicture}[y=-1cm]

\draw[thick,black] (8.8,0.8) rectangle (11.6,4.4);
\draw[thick,arrows=stealth-stealth,black] (10,2.1) -- (11.6,2.1);
\draw[thick,arrows=stealth-stealth,black] (8.8,3.7) -- (11.6,3.7);
\draw[thick,arrows=stealth-stealth,black] (9.6,2.9) -- (11.6,2.9);
\path[draw=black,thick,arrows=stealth-stealth] (12,4.4) -- (12,0.8);
\path[draw=black,thick,arrows=stealth-stealth] (8.8,4.8) -- (11.6,4.8);
\path[draw=black,thick,arrows=stealth-stealth] (7.2,0.8) -- (7.2,2);
\path[draw=black,thick,arrows=stealth-stealth] (7.6,0.8) -- (7.6,2.8);
\path[draw=black,thick,arrows=stealth-stealth] (8,0.8) -- (8,3.6);
\path[draw=black,thick,arrows=stealth-stealth] (8.4,0.8) -- (8.4,4.4);
\draw[thick,black] (2,0.8) rectangle (4.8,4.4);
\draw[thick,arrows=stealth-stealth,black] (4,1) -- (4.8,1);
\draw[thick,arrows=stealth-stealth,black] (3.2,2.9) -- (4.8,2.9);
\path[draw=black,thick,arrows=stealth-stealth] (5.2,4.4) -- (5.2,0.8);
\path[draw=black,thick,arrows=stealth-stealth] (2,4.8) -- (4.8,4.8);
\path[draw=black,thick,arrows=stealth-stealth] (1.6,0.8) -- (1.6,4.4);
\path[draw=black,thick,arrows=stealth-stealth] (1.2,0.8) -- (1.2,2.8);
\path[draw=black,thick,arrows=stealth-stealth] (0.8,0.8) -- (0.8,2);
\draw[thick,arrows=stealth-stealth,black] (3.6,2.1) -- (4.8,2.1);
\path (12.2,2.6) node[text=black,anchor=base west] {$d$};
\path (9.9,0.5) node[text=black,anchor=base west] {$k=4$};
\path (7.9,4.1) node[text=black,anchor=base west] {$d_4$};
\path (7.5,3.3) node[text=black,anchor=base west] {$d_3$};
\path (7.1,2.5) node[text=black,anchor=base west] {$d_2$};
\path (6.7,1.6) node[text=black,anchor=base west] {$d_1$};
\path (10.1,5.2) node[text=black,anchor=base west] {$e$};
\path (5.4,2.6) node[text=black,anchor=base west] {$d$};
\path (3.9,3.2) node[text=black,anchor=base west] {$e_3$};
\path (0.3,1.5) node[text=black,anchor=base west] {$d_1$};
\path (4.3,1.3) node[text=black,anchor=base west] {$e_1$};
\path (4.1,2.4) node[text=black,anchor=base west] {$e_2$};
\path (3,0.5) node[text=black,anchor=base west] {$k=3$};
\path (0.7,2.5) node[text=black,anchor=base west] {$d_2$};
\path (1.1,3.7) node[text=black,anchor=base west] {$d_3$};
\path (3.3,5.2) node[text=black,anchor=base west] {$e$};
\path (2.6,2.1) node[text=black,anchor=base west] {$\Lambda$};
\path (9.3,1.8) node[text=black,anchor=base west] {$\Lambda$};
\path (10.6,1.5) node[text=black,anchor=base west] {$e_1\!=\!0$};
\path (10.3,3.2) node[text=black,anchor=base west] {$e_3$};
\path (10,4) node[text=black,anchor=base west] {$e_4$};
\path (10.5,2.4) node[text=black,anchor=base west] {$e_2$};

\path[draw=black,thick,fill=white!50!black] (8.8,0.8) -- (11.6,0.8) -- (11.6,2) -- (10,2) -- (10,2.8) -- (9.6,2.8) -- (9.6,3.6) -- (8.8,3.6) -- cycle;
\path[draw=black,thick,fill=white!50!black] (2,0.8) -- (4,0.8) -- (4,2) -- (3.6,2) -- (3.6,2.8) -- (3.2,2.8) -- (3.2,4.4) -- (2,4.4) -- cycle;

\end{tikzpicture}
\end{center}
\caption{Two examples of the two $k$-tuples $(d_1,\ldots,d_k)$ and
$(e_1,\ldots,e_k)$ corresponding to a Young diagram $\Lambda$ in \deframe.}
\label{partition_fig}%
\end{figure}
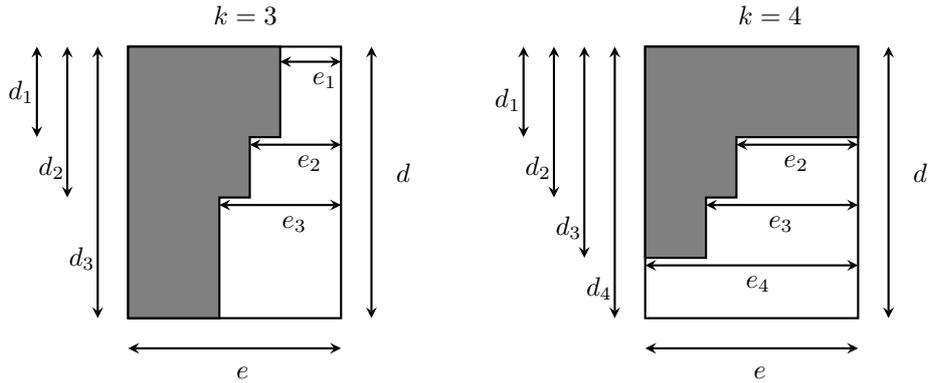

\begin{defi}\label{re-Flag_def}%
Let $d,e\geqslant1$. Fix a complete flag of vector bundles over~$X$
\begin{equation}
\label{complete_eq}%
0=\cV_0 \subb \cV_1 \subb \cdots \subb \cV_i \subb \cdots \subb \cV_{d+e}=:\cV\,.
\end{equation}
Note that we baptize $\cV$ the bundle of dimension $d+e$, to lighten notation.

Let $\Lambda$ be a Young diagram in
\deframe. By Definition~\ref{de_def}, this amounts to a pair $(\multi{d},\multi{e})$ of $k$-tuples of integers satisfying~\eqref{de_eq}, and hence satisfying~\eqref{pre-de_eq}. We can now apply Definition~\ref{Flag_def} to $\multi{d}$ and $\multi{e}$ and the flag~\eqref{partial_eq} taken from~\eqref{complete_eq} above\,:
\begin{equation}
\label{Flag-Lambda_eq}%
\Flag_X(d,e,\cV_\smallbullet;\Lambda):=\Flag_{X}(\multi{d},\multi{e},\cV_\smallbullet)
= \left\{
\begin{array}{c}
\xymatrix@R=1.5ex@C=1.5ex{ 0 \ar@{}[r]|-{\subb} & \cP_{d_1}
\ar@{}[d]|{\rotatebox{270}{$\subbeq$}}^*!/^.5ex/{\labelstyle e_1}
\ar@{}[r]|{\subb} & \cP_{d_2}
\ar@{}[d]|{\rotatebox{270}{$\subbeq$}}^*!/^.5ex/{\labelstyle e_2}
\ar@{}[r]|{\subb} & \cdots \ar@{}[r]|{\subb} & \cP_{d_k}
\ar@{}[d]|{\rotatebox{270}{$\subbeq$}}^*!/^.5ex/{\labelstyle e_k}
\\
0 \ar@{}[r]|-{\subb} & \cV_{d_1 + e_1} \ar@{}[r]|-{\subb} & \cV_{d_2
+ e_2} \ar@{}[r]|-{\subb} & \cdots \ar@{}[r]|-{\subb} & \cV_{d_k +
e_k} }
\end{array}
\right\}\,.
\end{equation}
As usual, instead of $\Flag_X(d,e,\cV_\smallbullet;\Lambda)$, we might simply write $\Flag(\Lambda)$ or anything ``in between'' depending on what is obvious from the context.

As in~\eqref{pre-ff_eq}, there is a natural morphism
$\ff_{\Lambda}$ from $\Flag_X(d,e;\Lambda)$ to $\Grass(d,\cV)$
\begin{equation}
\label{ff_eq}%
\begin{array}{rccc}
\ff_\Lambda=\ff_{d,e;\Lambda}:=\ff_{\multi{d},\multi{e},\cV}\,:
&\Flag_X(d,e;\Lambda) & \too & \Grass(d,\cV)
\\
&(\cP_{d_1},\ldots, \cP_{d_k}) & \longmapsto & \cP_{d_k}\,.
\end{array}
\end{equation}

When $X=\Spec(F)$ is a field, one can understand the image of
$\ff_\Lambda$ as the subset of those subspaces $\cP_d\subb \cV$ whose
intersection with each $\cV_{d_i+e_i}$ is of dimension at
least~$d_i$. This is the classical Schubert cell associated to the
diagram~$\Lambda$. It is pretty clear that $\ff_{\Lambda}$ is a
birational morphism. The advantage of $\Flag_X(d,e;\Lambda)$ over the Schubert
cell is that $\Flag_X(d,e;\Lambda)$ is not singular by Proposition~\ref{Pic-Flag_prop}.
\end{defi}

\begin{exam}
Following up on Example~\ref{k=1_exa}, when
$\Lambda=\vide$ is the empty diagram, that is for $k=1$ and
$e_1=e$, we have $\Flag_X(\vide)=\Grass_X(d,\cV)$ and $\ff_\vide$ is the identity. At the other
end, for $\Lambda=\rect{d}{e}$ the whole $(d\times e)$-rectangle,
that is for $k=1$ and $e_1=0$, we have
$\Flag_X(d,e;\Lambda)=\Grass_X(d,\cV_d)=X$ and $\ff_\Lambda$ is a closed immersion.
\end{exam}

\begin{defi}
\label{zerow_def}%
Let $\Lambda$ be a Young diagram in \deframe. We define $\row(\Lambda)\in\{0,\ldots,d\}$ to be the number of non-zero rows of $\Lambda$. Complementarily, we
define $\zero(\Lambda)=d-\row(\Lambda)$ to be the number of zero rows at the end of $\Lambda$, that is
$$
\begin{array}{lcll}
\row(\Lambda)=d & \text{and} & \zero(\Lambda)=0
 & \text{if }\Lambda_d>0\,,
\\
\row(\Lambda)=d_{k-1} & \text{and} & \zero(\Lambda)=d-d_{k-1}
& \text{if }\Lambda_d=0\,.
\end{array}
$$
For the empty diagram, we have $\row(\vide)=0$ and $\zero(\vide)=d$.
\end{defi}

We are going to use a certain class of $(d,e)$-diagrams, that we
call the \emph{even} $(d,e)$-diagrams. Defining them by a picture is
very easy. The condition to be even is that any segment of the
$(d,e)$-diagram which does not belong to the outer \deframe\ must
have even length. See Figure~\ref{even_fig}. The formal definition
is the following.

\begin{defi}
\label{even_def}%
Let $\Lambda$ be a Young diagram in \deframe\ and
let $\multi{d}$ and $\multi{e}$ be the associated $k$-tuples as in
Definition~\ref{de_def}. We say that $\Lambda$ is \emph{even} if all
the following conditions are satisfied:
\begin{enumerate}
\item[(i)] $d_{i+1}-d_i$ is even for all $i=1,\ldots,k-2$ (for
$k\geqslant 3$ otherwise no condition),
\item[(ii)] $e_{i+1}-e_i$ is even, for all
$i=1,\ldots,k-1$ (for $k\geqslant 2$),
\item[(iii)] when $0<e_1<e$ we also require $d_1$ to be even,
\item[(iv)] when $0<e_k<e$ we also require $d_k-d_{k-1}$ to be even.
\end{enumerate}
\end{defi}

\begin{exam}
For any $d,e\geqslant1$, both the empty diagram $\vide$ and
the full-rectangle $\rect{d}{e}$ are even $(d,e)$-diagrams (see Not.\,\ref{empty-full_nota}). Indeed,
in both cases, $k=1$ and $\multi{d}=(d)$, whereas
$\multi{e}(\vide)=(e)$ and $\multi{e}(\rect{d}{e})=(0)$; so there is
no condition to check.

When $d=1$ or $e=1$, these are the only even Young diagram in \deframe.

For more examples, the reader can find all even Young diagrams in
the cases $(d,e)=(4,4)$, $(4,5)$ and $(5,5)$ in
Figures~\ref{G44_fig}, \ref{G45_fig} and~\ref{G55_fig}, at the end
of the paper.
\end{exam}

\begin{rema}
Definition~\ref{even_def} depends on $d$ and $e$ as well as on
the Young diagram~$\Lambda$. For an even $(d,e)$-diagram $\Lambda$ to remain even in a bigger frame, we might have one or two more conditions to check, namely~(iii) or~(iv) in Definition~\ref{even_def}, in the case where $\Lambda$ was
touching the right border or the bottom border of its \deframe.
\end{rema}

\begin{rema}
For each even $(d,e)$-diagram we will construct an element in
one of the Witt groups of the Grassmannian $\Grass_X(d,\cV)$.
The proof that these Witt classes actually form a total
basis will proceed by induction on~$d+e=\rank(\cV)$, using the long
exact sequence of localization associated to a natural ``cellular''
decomposition of the Grassmannians. In that proof, we shall need the
description in terms of Young diagrams of the various Witt group
homomorphisms appearing in that long exact sequence. As we shall
see, these are the ones of the next proposition. This explains why
the following constructions are relevant here.
\end{rema}

\begin{defi} \label{bar_defi}

Let $\Lambda'$ be an even $(d,e-1)$-diagram with $d\geq 1$, $e \geq 2$ and such that $\zero(\Lambda')$ is even. We define the $(d,e)$-diagram $\barpush(\Lambda')$ as
\begin{equation}
\label{bar-iota_eq}%
\barpush(\Lambda')= (\Lambda'_1+1,\ldots,\Lambda'_d+1).
\end{equation}

Let $\Lambda$ be an even $(d,e)$-diagram with $d\geq 2$, $e\geq 1$ and such that $\Lambda_d=0$. We define the $(d-1,e)$-diagram $\barres(\Lambda)$ as
\begin{equation}
\label{bar-kappa_eq}%
\barres(\Lambda)=
\Lambda\restr{d-1,e}.
\end{equation}

Let $\Lambda''$ be an even $(d-1,e)$-diagram with $d,e\geq 2$ and such that $\Lambda''_{d-1}$ is odd. We define the $(d,e-1)$-diagram $\barbord(\Lambda'')$ as
\begin{equation}
\label{bar-bord_eq}%
\barbord(\Lambda'')=
(\Lambda_1''-1,\ldots,\Lambda_{d-1}''-1,0).
\end{equation}
See Figures~\ref{push_fig}, \ref{pull_fig} and~\ref{bord_fig}.
\end{defi}

\begin{prop}
\label{exact-Lambda_prop} %
\begin{enumerate}
\item When $d \geq 1$ and $e\geq 2$, the map $\barpush$ defines a bijection
$$
\xymatrix@R=.3em{ \left\{\Ourfrac{\text{even Young
}(d,e-1)\text{-diagrams}}{\Lambda'\text{ such that
}\zero(\Lambda')\text{ is even}}\right\} \ar@{<->}[r]^-{\simeq} &
\left\{\Ourfrac{\text{even Young
}(d,e)\text{-diagrams}}{\Lambda\text{ such that
}\Lambda_d>0}\right\}
\\
\qquad\Lambda'\qquad\ar@{|->}[r] & \quad
\barpush(\Lambda')=(\Lambda'_1+1,\ldots,\Lambda'_d+1)\quad
\\
(\Lambda_1-1,\ldots,\Lambda_d-1) & \qquad\Lambda \qquad\ar@{|->}[l]}
$$
\medbreak
\item When $d\geq 2$ and $e\geq 2$, the map $\barres$ defines a bijection
$$
\xymatrix@R=.3em{ \left\{\Ourfrac{\text{even Young
}(d,e)\text{-diagrams}}{\Lambda\text{ such that
}\Lambda_d=0}\right\} \ar@{<->}[r]^-{\simeq} &
\left\{\Ourfrac{\text{even Young
}(d-1,e)\text{-diagrams}}{\Lambda''\text{ such that
}\Lambda''_{d-1}\text{ is even}}\right\}
\\
\qquad\Lambda\qquad\ar@{|->}[r] & \qquad \barres(\Lambda)=\Lambda\restr{d-1,e}\quad
\\
(\Lambda_1'',\ldots,\Lambda_{d-1}'',0) & \qquad\Lambda''
\qquad\ar@{|->}[l]}
$$
\medbreak
\item When $d,e \geq 2$, the map $\barbord$ defines a bijection
$$
\xymatrix@R=.3em{ \left\{\Ourfrac{\text{even Young
}(d-1,e)\text{-diagrams}}{\Lambda''\text{ such that
}\Lambda''_{d-1}\text{ is odd}}\right\} \ar@{<->}[r]^-{\simeq} &
\left\{\Ourfrac{\text{even Young
}(d,e-1)\text{-diagrams}}{\Lambda'\text{ such that
}\zero(\Lambda')\text{ is odd}}\right\}
\\
\qquad\Lambda''\qquad\ar@{|->}[r] &
\barbord(\Lambda'')=(\Lambda_1''-1,\ldots,\Lambda_{d-1}''-1,0)
\\
(1+\Lambda_1',\ldots,1+\Lambda'_{d-1}) & \qquad\Lambda'
\qquad\ar@{|->}[l]}
$$
\end{enumerate}
\end{prop}

\begin{proof}
The proof essentially consists in checking that the announced
constructions are well-defined and that they preserve even diagrams.
Checking that they are mutually inverse constructions is
straightforward. The notation $\Lambda\restr{d-1,e}$ is the obvious
one\,: we view a diagram with empty last row in a smaller frame. All
this is most easily performed and followed on pictures. For
instance, the maps from left to right are pictured in the upper
parts of Figures~\ref{push_fig}, \ref{pull_fig} and~\ref{bord_fig}
below.
\end{proof}


%
\begin{figure}[!ht]
\begin{center}
\begin{tikzpicture}[y=-1cm]
\bigarrowtipdef

\path[draw=black,fill=black!0,arrows=|-biggertip] (4.4,2.8) -- (6,2.8);
\path[draw=black,fill=black!0,arrows=|-biggertip] (4.4,6) -- (6,6);
\path[draw=black,thick,arrows=stealth-stealth] (9.4,3.8) -- (9.4,1.8);
\path[draw=black,thick,arrows=stealth-stealth] (7,4.2) -- (9,4.2);
\path[draw=black,thick,fill=black!50] (7.4,3) -- (8.2,3) -- (8.2,2.2) -- (9,2.2) -- (9,1.8) -- (7,1.8) -- (7,3.8) -- (7.4,3.8) -- cycle;
\draw[thick,dashed,black] (7.4,1.8) -- (7.4,3);
\path[draw=black,thick,arrows=stealth-stealth] (1.6,4.2) -- (3.2,4.2);
\path[draw=black,thick,fill=black!50] (1.6,3) -- (2.4,3) -- (2.4,2.2) -- (3.2,2.2) -- (3.2,1.8) -- (1.6,1.8) -- cycle;
\path[draw=black,thick,arrows=stealth-stealth] (3.6,3.8) -- (3.6,1.8);
\path[draw=black,thick,fill=black!50] (1.6,5.8) -- (3.2,5.8) -- (3.2,5) -- (1.6,5);
\path[draw=black,thick,arrows=stealth-stealth] (1.5,7) -- (1.5,5.8);
\path[draw=black,thick,arrows=stealth-stealth] (1.5,3.8) -- (1.5,3);
\path (7.9,4.5) node[text=black,anchor=base west] {$e$};
\path (9.4,2.6) node[text=black,anchor=base west] {$d$};
\path (3.6,2.6) node[text=black,anchor=base west] {$d$};
\path (7.8,6.1) node[text=black,anchor=base west] {$0$};
\path (0.6,3.5) node[text=black,anchor=base west] {even};
\path (0.7,6.5) node[text=black,anchor=base west] {odd};
\path (1.9,4.5) node[text=black,anchor=base west] {$e-1$};

\draw[thick,black] (7,1.8) rectangle (9,3.8);
\draw[thick,black] (1.6,1.8) rectangle (3.2,3.8);
\draw[thick,black] (1.6,5) rectangle (3.2,7);

\end{tikzpicture}
\end{center}
\caption{Morphism $\barpush$ on various $(d,e-1)$-diagrams $\Lambda'$.} \label{push_fig}%
\end{figure}

\vfill

\begin{figure}[!ht]
\begin{center}
\begin{tikzpicture}[y=-1cm]
\bigarrowtipdef

\path[draw=black,thick,fill=black!50] (2,2.6) -- (2.8,2.6) -- (2.8,1.8) -- (3.6,1.8) -- (3.6,1.4) -- (2,1.4) -- cycle;
\path[draw=black,thick,fill=black!50] (7.2,2.6) -- (8,2.6) -- (8,1.8) -- (8.8,1.8) -- (8.8,1.4) -- (7.2,1.4) -- cycle;
\draw[thick,black] (7.2,1.4) rectangle (8.8,3);
\draw[thick,black] (2,1.4) rectangle (3.6,3.4);
\path[draw=black,thick,arrows=stealth-stealth] (4,3.4) -- (4,1.4);
\path[draw=black,thick,arrows=stealth-stealth] (2,3.8) -- (3.6,3.8);
\path[draw=black,thick,arrows=stealth-stealth] (9.2,3) -- (9.2,1.4);
\path[draw=black,thick,arrows=stealth-stealth] (7.2,3.8) -- (8.8,3.8);
\draw[thick,dashed,black] (7.2,3) -- (7.2,3.4) -- (8.8,3.4) -- (8.8,3);
\path[draw=black,fill=black!0,arrows=|-biggertip] (4.8,2.4) -- (6.4,2.4);
\path[draw=black,thick,fill=black!50] (2,6.6) -- (2.8,6.6) -- (2.8,5) -- (3.6,5) -- (3.6,4.6) -- (2,4.6) -- cycle;
\draw[thick,black] (2,4.6) rectangle (3.6,6.6);
\path[draw=black,fill=black!0,arrows=|-biggertip] (4.8,5.6) -- (6.4,5.6);
\path (9.2,2.3) node[text=black,anchor=base west] {$d-1$};
\path (4,2.3) node[text=black,anchor=base west] {$d$};
\path (8,5.8) node[text=black,anchor=base west] {$0$};
\path (2.7,4.1) node[text=black,anchor=base west] {$e$};
\path (7.9,4.1) node[text=black,anchor=base west] {$e$};

\end{tikzpicture}
\end{center}
\caption{Morphism $\barres$ on various $(d,e)$-diagrams $\Lambda$.} \label{pull_fig}%
\end{figure}

\vfill

\begin{figure}[!ht]
\begin{center}
\begin{tikzpicture}[y=-1cm]
\bigarrowtipdef

\path[draw=black,fill=black!0,arrows=|-biggertip] (4.8,2.8) -- (6.4,2.8);
\path[draw=black,fill=black!0,arrows=|-biggertip] (4.8,6) -- (6.4,6);
\path[draw=black,thick,arrows=stealth-stealth] (1.2,3.5) -- (2.4,3.5);
\path[draw=black,thick,arrows=stealth-stealth] (3.6,3.4) -- (3.6,1.8);
\path[draw=black,thick,arrows=stealth-stealth] (1.2,4.2) -- (3.2,4.2);
\path[draw=black,thick,fill=black!50] (2.4,2.6) -- (3.2,2.6) -- (3.2,1.8) -- (1.2,1.8) -- (1.2,3.4) -- (2.4,3.4) -- cycle;
\path[draw=black,thick,arrows=stealth-stealth] (7.2,4.2) -- (8.8,4.2);
\path[draw=black,thick,arrows=stealth-stealth] (9.2,3.8) -- (9.2,1.8);
\draw[thick,dashed,black] (8.8,3.4) -- (6.8,3.4) -- (6.8,1.8) -- (7.2,1.8);
\path[draw=black,thick,fill=black!50] (8,2.6) -- (8.8,2.6) -- (8.8,1.8) -- (7.2,1.8) -- (7.2,3.4) -- (8,3.4) -- cycle;
\path[draw=black,thick,fill=black!50] (1.2,6.6) -- (2,6.6) -- (2,5.8) -- (2.8,5.8) -- (2.8,5) -- (1.2,5);
\path[draw=black,thick,arrows=stealth-stealth] (1.2,6.7) -- (2,6.7);
\path (7.8,6.1) node[text=black,anchor=base west] {$0$};
\path (3.6,2.6) node[text=black,anchor=base west] {$d-1$};
\path (1.4,3.8) node[text=black,anchor=base west] {odd};
\path (1.2,7) node[text=black,anchor=base west] {even};
\path (9.2,2.9) node[text=black,anchor=base west] {$d$};
\path (7.6,4.5) node[text=black,anchor=base west] {$e-1$};
\path (2.1,4.5) node[text=black,anchor=base west] {$e$};

\draw[thick,black] (1.2,1.8) rectangle (3.2,3.4);
\draw[thick,black] (7.2,1.8) rectangle (8.8,3.8);
\draw[thick,black] (1.2,5) rectangle (3.2,6.6);

\end{tikzpicture}
\end{center}
\caption{Morphism $\barbord$ on various
$(d-1,e)$-diagrams $\Lambda''$.} \label{bord_fig}%
\end{figure}


\section{Total bases and lax-similitude}
\label{tot-form_sec}%

\begin{conv}
From now on, $\BS$ denotes a {\em regular} noetherian $\bbZ[\frac{1}{2}]$-scheme of finite Krull dimension.
\end{conv}

For precise statements and proofs of our results, it is convenient
to use the language of total bases and \lax\ similitude, as
developed in~\cite{Balmer11}. Here is a brief list of the
relevant facts.

We restrict to the subcategory $\SmPic_\BS$ of $\BS$-schemes
$\pi_\CS:\CS\to \BS$ that are smooth over $\BS$ and that
satisfy the following assumptions on Picard groups and global
sections~\cite[Definition 4.1]{Balmer11}:
\begin{enumerate}[\indent(I)]
\item \label{picinj_item} The map $\pi_\CS^*:\Pic(\BS) \to \Pic(\CS)$ is injective.
\smallbreak
\item \label{2torzero_item} The abelian group
$\Pic_\BS(\CS):=\Pic(\CS)/\pi_\CS^*\big(\Pic(\BS)\big)$ has no $2$-torsion.
\smallbreak
\item \label{Gmhyp_item} The map $\pi_\CS^*:\Gm(\BS) \to \Gm(\CS)/\Gm(\CS)^2$ is surjective.
\end{enumerate}
This ensures that the notions considered below are well-behaved.

\begin{rema}
All schemes considered in the computations of the remaining Sections~\ref{gen-Witt_sec} to~\ref{main_sec} are in the category $\SmPic_X$. Indeed, most are flag varieties constructed iteratively from $X$ as towers of Grassmann bundles (see Lemma~\ref{rel_lem}), so the Picard group assumptions follow from Proposition~\ref{Pic-Flag_prop}. Similarly, Property~\eqref{Gmhyp_item} follows from~\cite[Thm. 2.3.1]{EGA3-1}, \ie in all our examples, $\Gm(\BS)\to \Gm(\CS)$ is already surjective. The remaining schemes are vector bundles over these flag varieties, so each of them has the same Picard group and invertible global sections as its base. 
\end{rema}

\begin{defi}[{\cite[Definition 2.3]{Balmer11}}]
Let $L_1$ and $L_2$ be line bundles over a scheme~$\CS$. An
alignment $A:L_1 \alto L_2$ is a pair $A=(M,\psi)$ consisting of a
line bundle $M$ over $\CS$ together with an isomorphism
$\psi:M^{\otimes 2}\otimes L_1 \isoto L_2$. Of course, such an
alignment exists if and only if $[L_1]=[L_2]$ in $\Pic(\CS)/2$. It
induces an isomorphism on Witt groups
$$\alis{A}:\Wcoh^*(\CS,L_1) \isoto \Wcoh^*(\CS,L_2)$$
defined as the composition of multiplication by the form $M\isoto M^\vee \otimes M^{\otimes 2}$ (square periodicity) and of the identification of the dualities with values respectively in $M^{\otimes 2}\otimes L_1$ and in $L_2$ using $\psi$.

When $\CS$ is a scheme in $\SmPic_\BS$, we also use a relative
notion. An {\em $\BS$-alignment} from $L_1$ to $L_2$ is an alignment
$A:\pi_\CS^*K \otimes L_1 \alto L_2$ for some line bundle $K$
over~$\BS$. We denote this by $A:L_1 \alKto{K} L_2$.
\end{defi}

\begin{defi}[{\cite[Definition 2.5]{Balmer11}}] \label{laxsim_defi}
Two Witt classes $w_1 \in \Wcoh^j(\CS,L_1)$ and $w_2\in
\Wcoh^j(\CS,L_2)$ are {\em \lax-similar} if there exists an
alignment $A$ such that $w_2=\alis{A}(w_1)$. This is an equivalence
relation written $w_1\weq w_2$. Note also that $w_1\weq w_2$ forces $[L_1]=[L_2]$ in $\Pic(\CS)/2$.
\end{defi}

\begin{defi}[{\cite[Definition 3.4]{Balmer11}}] \label{laxpull_defi}
Let $f:\bar \CS \to \CS$, let $L$ be a line bundle on $\CS$ and let
$\bar L$ be a line bundle on~$\bar \CS$. If an alignment $\bar
A:f^*L \alto \bar L$ exists, we define a {\em \lax\ pull-back}
$$
\alis{\bar A}\circ
f^*\,:\quad\Wcoh^*_{\!Z}(\CS,L)\too\Wcoh^*_{f\inv Z}(\bar \CS,f^* L)
\isotoo\Wcoh^*_{f\inv Z}(\bar \CS,\bar L)\,.
$$
\end{defi}

\begin{rema} \label{laxpull_rema}
It is easy to see that two \lax\ pull-backs (along the same
morphism) of \lax-similar elements are \lax-similar.
\end{rema}

\begin{defi}[{\cite[Definition 3.5]{Balmer11}}] \label{laxpush_defi}
Similarly, if an alignment $\bar A:\bar L \alto \can_f \otimes f^*L$
exists, we define a {\em \lax\ push-forward}
$$
f_*\circ\alis{\bar A}\,:\quad
\Wcoh^{*+d}_{\!\bar Z}(\bar \CS,\bar{L})\isotoo
\Wcoh^{*+d}_{\!\bar Z}(\bar \CS,\can_f\otimes f^*
L)\too\Wcoh^*_{\!f\bar Z}(\CS,L)\,.
$$
\end{defi}

The freedom in the use of \lax\ push-forwards is summarized in the
following
\begin{theo} \label{laxpush_thm}
Let $f:\bar{\CS} \to \CS$ be a morphism of schemes in~$\SmPic_\BS$
and let $\bar{L}$ be a line bundle over~$\bar{\CS}$.
\begin{enumerate}[\indent(a)]
\smallbreak
\item\label{it:push-from}
A \lax\ push-forward starting from $\Wcoh^*(\bar{\CS},\bar{L})$
exists if and only if
\begin{equation} \label{genpushfromexists_eq} [\bar{L}] \in {\rm
Im}\big([\can_f] \otimes f^*: \Pic(\CS)/2 \to \Pic(\bar{\CS})/2\big)
\end{equation}
or equivalently replacing $\Pic(-)/2$ by $\Pic_\BS(-)/2$.
\smallbreak
\item Assuming~\eqref{genpushfromexists_eq} holds, the \lax\ push-forward can be chosen to land in $\Wcoh^*(\CS,L)$ for a line bundle $L$ on $\CS$ if and only if
\begin{equation} \label{genpushtoexists_eq}
[\can_f \otimes f^*L] = [\bar{L}] \in \Pic(\bar{\CS})/2.
\end{equation}
\smallbreak
\item\label{it:push-unique}
Given two line bundles $L_1$ and $L_2$ on $\CS$ both satisfying \eqref{genpushtoexists_eq}, \lax\ push-forwards from $\Wcoh^*(\bar{\CS},\bar{L})$ to $\Wcoh^*(\CS,L_1)$ and to $\Wcoh^*(\CS,L_2)$ are \lax-similar on~$\CS$ (\ie there exists an alignment $\alis{A}:\Wcoh^*(\CS,L_1)\to \Wcoh^*(\CS,L_2)$ turning one push-forward into the other) if and only if
\begin{equation}
[L_1]=[L_2] \in \Pic(\CS)/2
\end{equation}
or equivalently replacing $\Pic(-)/2$ by $\Pic_\BS(-)/2$. This condition is automatically satisfied if $f^*:\Pic_\BS(\CS)/2\to \Pic_\BS(\bar{\CS})/2$ is injective, so in that case, there is no need to specify the specific target of the \lax\ push-forward if one is only interested in \lax-similitude classes of the images.
\end{enumerate}
\end{theo}
\begin{proof}
This is detailed in~\S\,4 of~\cite{Balmer11}. The first two parts are straightforward from~\eqref{push-forward_eq}. Use~\cite[Lemma~4.3\,(d)]{Balmer11} to replace $\Pic(-)/2$ by $\Pic_\BS(-)/2$ in~\eqref{it:push-from}. To replace $\Pic(-)/2$ by $\Pic_\BS(-)/2$ in the last statement, use~\cite[Lemma~4.3\,(c)]{Balmer11}, in which case~\cite[Proposition~4.7]{Balmer11} gives that the images are \lax-similar.
\end{proof}

\begin{rema} \label{simpushpull_rema}
Remark~\ref{laxpull_rema} and Theorem~\ref{laxpush_thm} mean
that as long as one is only interested in elements up to
\lax-similarity, there is no need to be specific about where \lax\
pull-backs and push-forwards start and land, as long as they exist. One only needs to keep track of classes of line bundles in $\Pic_\BS(-)/2$. See also \cite[Remark~2.10]{Balmer11} for connecting homomorphisms.
\end{rema}

Let $\CS$ be a scheme in $\SmPic_\BS$ and let $Z$ be a closed subset
of $\CS$. Let $\Set$ be a set. Given a family of line bundles
$(L_\varS)_{\varS\in\Set}$ and a class $p\in \Pic_\BS(\CS)/2$, let
$\Set_p$ denote the subset of those $\varS\in\Set$ such that
$[L_\varS]=p$. Let $(w_\varS)_{\varS\in\Set}$ be a family of
Witt classes, for various shifts and twists\,: $w_\varS \in
\Wcoh^{j_\varS}(\CS,L_\varS)$.

\begin{defi} \label{lincomb}
Let $L$ be a line bundle on $\CS$ and let $k\in \bbZ$ be an integer.
Let $\FSet\subset\Set_{[L]}$ be a finite subset. Given a family of
$\BS$-alignments $(A_\varF:L_\varF \alKto{K_\varF} L)_{\varF\in\FSet}$ and a family of coefficients
$(\lambda_\varF)_{\varF\in\FSet}$ with $\lambda_\varF \in
\Wcoh^{k-j_\varF}(\BS,K_\varF)$, we can form the {\em linear
combination}
$$
\sum_{\varF\in\FSet}\lambda_\varF \dt{A_\varF} w_\varF\,
:=\,\sum_{\varF\in\FSet}
\alis{A_\varF}\big(\pi_\CS^*(\lambda_\varF)\cdot w_\varF\big),
$$
which is an element of $\Wcoh^k(\CS,L)$. See details in~\cite[\S\,6]{Balmer11}.
\end{defi}

\begin{defi}
\label{tot-basis_defi}%
Let $P \subseteq \Pic_\BS(\CS)/2$ -- typically $P$ is the whole $\Pic_{\BS}(\CS)/2$. The family $(w_\varS)_{\varS\in\Set}$ is called a \emph{total basis} of the $P$-part of the Witt groups of $\CS$ over~$\BS$ with support in~$Z$, if $[L_\varF]\in P$ for all $\varF$ and if for every line bundle $L$ with $[L]\in P$, we have the following two properties\,:
\begin{enumerate}
\item \emph{Total generation}\,: Any element in $\Wcoh^k(\CS,L)$ can be obtained as a linear combination of a finite subfamily $(w_\varF)_{\varF\in\FSet}$, \ie alignments and coefficients as in Definition~\ref{lincomb} can be found yielding the element as the linear combination.
\item \emph{Total independence}\,: Any linear combination of the $w_\varF$ yielding the zero element in $\Wcoh^k(\CS,L)$ has zero coefficients.
\end{enumerate}
\end{defi}

A total basis yields the following result, expressed in classical terms.

\begin{theo}[{\cite[Proposition
6.9]{Balmer11}}]\label{thm:classic-basis}%
For every line bundle $L$ with $[L]\in P$, every $k\in \bbZ$ and
\emph{for every choice}, for those $\varS\in \Set_{[L]}$, of a line
bundle $K_\varS$ over $\BS$ and a $K_\varS$-alignment
$C_\varS:L_\varS\alKto{K_\varS}L$, the following map is an
isomorphism
\begin{equation}
\label{eq:iso}%
\begin{array}{ccc}
\displaystyle\theta=\theta(C_\smallbullet)\,:\hspace{-1ex}\
\bigoplus_{\varS\in\Set_{[L]}}\hspace{-1ex}
\Wcoh^{k-j_\varS}(\BS,K_\varS) & \isotoo & \Wcoh^k_{\!Z}(\CS,L)
\\
(x_\varS)_{\varS\in\Set_{[L]}} \kern-10em & \longmapsto & \sum
x_\varS \dt{C_\varS}w_\varS\,.
\end{array}
\end{equation}
\end{theo}

\begin{rema}
\label{lax-basis_rema}%
In the same spirit, if one replaces elements of a total basis by
\lax-similar ones (Def.\,\ref{laxpull_defi}), they still form a
total basis. See~\cite[Cor.\,6.14]{Balmer11}.
\end{rema}

Finally, here is a way to keep track of total bases along
localization sequences. This will be our main tool to construct
inductively total bases for Grassmann varieties, together with
homotopy invariance and d\'evissage, under which total bases are
naturally preserved (see~\cite[Corollaries 6.13 and
6.16]{Balmer11} for precise statements).

Let $U$ be the open complement of a closed subset $Z\subset \CS$, and let $\upsilon: U \hook \CS$ be the
corresponding open embedding. Assume $U \in \SmPic_\BS$. Let $\ext: \Wcoh^*_{\!Z}(\CS,L) \to
\Wcoh^*(\CS,L)$ be the extension of support map. Recall from~\cite{Balmer00} that there is a long exact sequence of localization
\begin{equation}
\label{eq:LES}%
\cdots \too \Wcoh^i_{\!Z}(\CS,L) \tooby{\ext} \Wcoh^i(\CS,L)
\tooby{\upsilon^*} \Wcoh^i(U,\upsilon^*L)\tooby{\bord}
\Wcoh^{i+1}_{\!Z}(\CS,L) \to \cdots
\end{equation}
\begin{theo}[{\cite[Thm.\,7.1]{Balmer11}}]
\label{thm:localization}%
Let $P$ be a subset of $\Pic_\BS(\CS)/2$. Assume that the restriction
$\upsilon^*_{|P}:P \to \Pic_\BS(U)/2$ is injective and let
$P_U=\upsilon^*(P) \subset \Pic_\BS(U)/2$.

Let $\SetZ$, $\SetX$ and $\SetU$ be sets and let $(w'_i)_{\varSZ\in
\SetZ}$ and $(w_\varSX)_{\varSX\in \SetX}$ be Witt classes on~$\CS$, let $(v_\varSZ)_{\varSZ\in \SetZ}$ and
$(v'_\varSU)_{\varSU\in \SetU}$ be Witt classes
on\,~$\CS$ with support in $Z$ and let $(u_\varSU)_{\varSU\in
\SetU}$ and $(u'_\varSX)_{\varSX\in \SetX}$ be Witt
classes on\,~$U$, whose line bundles are restricted from $\CS$. Suppose the following conditions hold
(see Figure~\ref{fig:uvw})\,:
\begin{enumerate}[\indent(a)]
\item \label{eweq_item} for every $\varSZ\in \SetZ$, we have \lax-similitude $\ext(v_\varSZ)\weq w'_\varSZ$
\item \label{upsilonweq_item} for every $\varSX\in \SetX$, we have \lax-similitude $\upsilon^*(w_\varSX)\weq u'_\varSX$
\item \label{bordweq_item} for every $\varSU\in \SetU$, we have \lax-similitude $\bord(u_\varSU)\weq v'_\varSU$\,.
\end{enumerate}
Then, the following properties are satisfied\,:
\begin{enumerate}[\indent(1)]
\item \label{upsilonai_item} for every $\varSZ\in \SetZ$, we have
$\upsilon^*(w'_\varSZ)=0$;
\item \label{bordci_item} for every $\varSX\in \SetX$, we have $\bord(u'_\varSX)=0$;
\item \label{ebi_item} for every $\varSU\in \SetU$, we have $\ext(v'_\varSU)=0$.
\item \label{basisinduction_item} If, furthermore,
\begin{enumerate}[\indent(i)]
\item the $(v_\varSZ)_{\varSZ\in \SetZ}$ and $(v'_\varSU)_{\varSU\in \SetU}$ form together a total basis of the $P$-part of the Witt groups of\,~$\CS$ with support in $Z$, over~$\BS$,
\item the $(u_\varSU)_{\varSU\in \SetU}$ and $(u'_\varSX)_{\varSX\in \SetX}$ form together a total basis of the $P_U$-part of the Witt groups of\,~$U$, over~$\BS$,
\end{enumerate}
then the $(w'_\varSZ)_{\varSZ\in \SetZ}$ and $(w_\varSX)_{\varSX\in
\SetX}$ form together a total basis of the $P$-part of the Witt
groups of\,~$\CS$, over~$\BS$.
\end{enumerate}
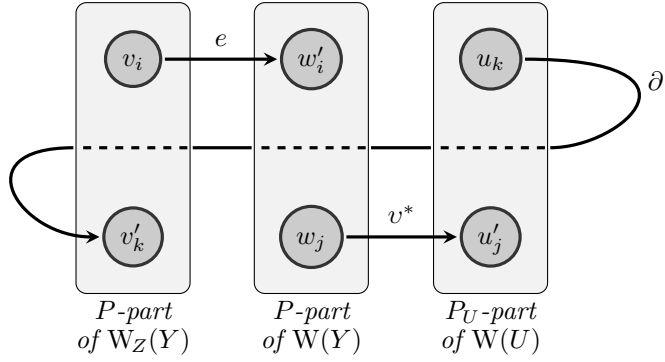
\begin{figure}[h!]
\begin{tikzpicture}[>=stealth,shorten >=1pt,shorten <=1pt,text height=1.5ex,text depth=.25ex,on grid,semithick,mynode/.style={circle,fill=black!20,very thick,draw=black!80,text=black}]
\matrix[row sep=5ex,column sep=10ex] {
\node (b)  [mynode] {$v_\varSZ$}; & \node (a)  [mynode] {$w'_\varSZ$}; & \node (c)  [mynode] {$u_\varSU$}; \\
\coordinate (x);  & \coordinate (y); & \coordinate (z); \\
\node (b') [mynode] {$v'_\varSU$}; & \node (w') [mynode] {$w_\varSX$}; & \node (c') [mynode] {$u'_\varSX$}; \\
};
\path (b)++(0.75,0.75) coordinate (b+);
\path (a)++(0.75,0.75) coordinate (a+);
\path (c)++(0.75,0.75) coordinate (c+);
\begin{pgfonlayer}{background}
\draw[fill=black!10,fill opacity=0.5,rounded corners] (b')+(-0.75,-0.75) rectangle (b+);
\draw[fill=black!10,fill opacity=0.5,rounded corners] (w')+(-0.75,-0.75) rectangle (a+);
\draw[fill=black!10,fill opacity=0.5,rounded corners] (c')+(-0.75,-0.75) rectangle (c+);
\end{pgfonlayer}
\draw[very thick,->] (b) -- node[above]{$e$} (a);
\draw[very thick,->] (w') -- node[above]{$\upsilon^*$}(c');
\path (x)++(-0.75,0) coordinate (x-);
\path (x)++(0.75,0) coordinate (x+);
\path (y)++(-0.75,0) coordinate (y-);
\path (y)++(0.75,0) coordinate (y+);
\path (z)++(-0.75,0) coordinate (z-);
\path (z)++(0.75,0) coordinate (z+);
\draw[very thick] (c) .. controls +(3,0) and +(1,0) .. node[above right] {$\bord$} (z+);
\draw[very thick,dashed] (z+) -- (z-);
\draw[very thick] (z-) -- (y+);
\draw[very thick,dashed] (y+) -- (y-);
\draw[very thick] (y-) -- (x+);
\draw[very thick,dashed] (x+) -- (x-);
\draw[very thick,->] (x-) .. controls +(-1.5,0) and +(-1.5,0) .. (b');
\draw (b')+(0,-1) node (pz) {$P$-part};
\draw (pz)+(0,-2.5ex) node {of $\Wcoh_{\!Z}(\CS)$};
\draw (w')+(0,-1) node (px) {$P$-part};
\draw (px)+(0,-2.5ex) node {of $\Wcoh(\CS)$};
\draw (c')+(0,-1) node (pu) {$P_U$-part};
\draw (pu)+(0,-2.5ex) node {of $\Wcoh(U)$};
\end{tikzpicture}
\caption{\label{fig:uvw}%
Families mapping to each other up to \lax-similitude in
Theorem~\ref{thm:localization}. No arrow means mapped to zero.}
\end{figure}
\end{theo}

Finally, we will need the following fact about push-forwards along blow-up\,:

\begin{prop}
\label{push-1_prop}%
Let $X$ be a quasi-compact and quasi-separated scheme (\eg\ an affine or a noetherian scheme) and let $Z\hookrightarrow X$ be a regular immersion of pure codimension~$d$. Let $\pi:\Bl\to X$ be the blow-up of $X$ along~$Z$. Then
\begin{enumerate}
\item There is a natural isomorphism $R\pi_*(\cO_\Bl)\cong \cO_X$ in
the derived category of~$X$.
\item Assume further that $X$ is regular and that $\can_{\Bl/X}$ is
a square, which happens exactly when $d$ is odd by~\cite[Proposition
A.11 (iii)]{Balmer09}. Then a \lax\ push-forward (in the sense of
Definition~\ref{laxpush_defi})
$\Wcoh^0(\Bl,\cO_{\Bl})\to\Wcoh^0(X,\cO_X)$ maps the unit class
$1_\Bl$ to an element \lax-similar to the unit class $1_X$.
\end{enumerate}
\end{prop}

\begin{proof}
Part~(a) can be found in SGA~6, see~\cite[Lemme~VII.3.5, p.\,441]{SGA6}
or the more recent account in~\cite[Lemme~2.3\,(a)]{Thomason93}.

Part (b) follows from~(a) and holds at the level of symmetric forms
already, before taking Witt classes. Indeed, when $d=1$, we have
$\Bl=X$ and there is nothing to prove. When $d\geqslant3$ then line
bundles over~$X$, and homomorphisms between them, are determined by
their restriction to the open complement~$U=X\smallsetminus Z$
of~$Z$ since $Z$ is of codimension at least~2. The result follows by
the base-change formula for push-forwards~\cite[Thm. 5.5]{Calmes11}
and by Theorem~\ref{laxpush_thm}, since
$\pi\restr{\pi^{-1}(U)}:\pi^{-1}(U)\to U$ is an isomorphism.
\end{proof}


\section{Construction of the total basis}
\label{gen-Witt_sec}%


For this section, let $\Lambda$ be a Young diagram in \deframe\ and recall the $k$-tuples $\multi{d}$ and $\multi{e}$ associated to $\Lambda$ in Definition~\ref{de_def}.

\begin{rema}
\label{cond-push_rem}%
Our goal is to construct classes in the total Witt group of $\Grass_X(d,\cV)$ by \lax-pushing-forward the unit form $1\in\Wper(\Flag_X(\Lambda))=\Wcoh^0(\Flag_X(\Lambda),\cO)$ along the morphism $\ff_\Lambda:\Flag_X(\Lambda)\to \Grass_X(d,\cV)$ of~\eqref{ff_eq}. As recalled in Theorem~\ref{laxpush_thm}\,\eqref{it:push-from}, this \lax\ push-forward only exists conditionally, namely only when the class of the relative canonical bundle $\can_{\Flag_X(\Lambda)/\Grass_X(d,\cV)}$ in $\Pic_X(\Flag_X(\Lambda))/2$ belongs to the image of
$$
(\ff_\Lambda)^*\,:\
\Pic_X(\Grass_X(d,\cV))/2\too\Pic_X(\Flag_X(\Lambda))/2\,.
$$
This is true if and only if the following conditions are satisfied:
\begin{enumerate}
\item $d_i-d_{i-1}+e_{i+1}-e_{i}$ is even for every $i=2,\ldots,k-1$
(for $k\geqslant3$)
\item when $0<e_1<e$ and $k\geq2$, require moreover $d_1+e_2-e_1$ even.
\end{enumerate}

We shall be more precise in Proposition~\ref{cond-even_prop} below but the reader can verify our claim using~\eqref{can-Flag_eq} in Corollary~\ref{can-Flag_cor}. For this, note that $\Det_{d_1}=[\det\cV_{d_1}]$ comes from $X$ when $e_1=0$ and that $\Det_d$ always comes from $\Grass_X(d,\cV)$ since $(\ff_\Lambda)^*\big(\Det_d^{\Grass_X(d,\cV)}\big)=\Det_d^{\Flag_X(\Lambda)}$, as can be checked on the tautological bundles already.

Conditions~(a) and (b) hold in particular when $\Lambda$ is \emph{even} in the sense of Definition~\ref{even_def}. Indeed, for such $\Lambda$ not only the sum
$d_i-d_{i-1}+e_{i+1}-e_{i}$ is even but actually both terms
$d_i-d_{i-1}$ and $e_{i+1}-e_{i}$ are. Compare Figure~\ref{conditions_fig}.

\begin{figure}[!ht]
\begin{center}
\begin{tikzpicture}[y=-1cm]

\path[draw=black,thick,fill=black!50] (5.6,3.8) -- (6,3.8) -- (6,2.6) -- (6.4,2.6) -- (6.4,2.2) -- (6.8,2.2) -- (6.8,1) -- (5.6,1);

\draw[thick,arrows=stealth-stealth,black] (6.3,2.2) -- (6.3,2.5) -- (6,2.5);
\draw[thick,arrows=stealth-stealth,black] (6.7,1) -- (6.7,2.1) -- (6.4,2.1);
\draw[thick,arrows=stealth-stealth,black] (5.9,2.6) -- (5.9,2.9) -- (5.9,3.8);
\draw[thick,black] (5.6,1) rectangle (7.6,3.8);
\path (6.3,2.8) node[text=black,anchor=base west] {$2$};
\path (6,3.3) node[text=black,anchor=base west] {$3$};
\path (6.7,2.4) node[text=black,anchor=base west] {$4$};

\end{tikzpicture}
\end{center}
\caption{Framed Young diagram satisfying Conditions~(a) and~(b) of
Remark~\ref{cond-push_rem} but which is not even (at all).}
\label{conditions_fig}%
\end{figure}
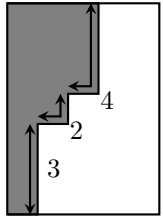

When (a) and (b) hold (\eg\ for $\Lambda$ even), there exists a line
bundle $L$ on $\Grass_X(d,\cV)$ such that $[\can_{\ff_\Lambda}]\cdot
\ff_\Lambda^*[L]=1$ in $\Pic(\Flag_X(\Lambda))\big/2$. Therefore,
there is a \lax\ push-forward~\ref{laxpush_defi} along
$f=\ff_\Lambda$ \ie a homomorphism\,:
$$
\Wcoh^{0}\big(\Flag(\Lambda),\cO\big)\simeq
\Wcoh^{0}\big(\Flag(\Lambda),\can_{\ff_\Lambda}\otimes \ff_\Lambda^*L\big)
\tooo{2}^{(\ff_\Lambda)_*}\Wcoh^{|\Lambda|}\big(\Grass_X(d,\cV),L\big)\,.
$$
(Use that
$\dim(\ff_\Lambda)=\dim\Flag_X(\Lambda)-\dim\Grass_X(d,\cV)=-|\Lambda|$
by Propositions~\ref{Pic-Grass_prop} and~\ref{Pic-Flag_prop}.)
Consequently, we can produce a Witt class over $\Grass_X(d,\cV)$, by
pushing  the unit form in the first group. This is what we are going
to do below for $\Lambda$ even, making the class of $L$ in
$\Pic(\Grass_X(d,\cV))/2$ more explicit in terms of the
diagram~$\Lambda$. Once this class of $L$ in $\Pic(\Grass_X(d,\cV))/2$ or in $\Pic_X(\Grass_X(d,\cV))/2$ is fixed, the choices involved in the \lax\ push-forward of
Definition~\ref{laxpush_defi} will be irrelevant. A nice fact is that this class in the Picard group can be read very easily on the diagram, as we now explain.
\end{rema}

\begin{rema}
\label{perim_rem}%
The perimeter of a Young diagram $\Lambda$ is an even integer.
Indeed, from the lower-left corner of $\Lambda$ to its upper-right
corner, there are two paths which follow the boundary (the upper
path and the lower path) and they have the same length, namely the
lattice distance between these two corners.
\end{rema}

\begin{defi}
\label{twist_def}%
Let $\Lambda$ be a Young diagram. We define $\twist(\Lambda)\in\bbZ/2$ to be the class of half the perimeter of~$\Lambda$. From the above remark, $\twist(\Lambda)$ is also the class of the (lattice) distance from the lower-left corner of $\Lambda$ to its upper-right corner. That is\,:
$$
\twist(\Lambda)=[\Lambda_1+\row(\Lambda)]\in\bbZ/2
$$
where $\row(\Lambda)$ is the number of non-zero rows of $\Lambda$ (Def.\,\ref{zerow_def}). Note that this Definition does not depend on an ambient frame.
\end{defi}

\begin{rema}
On an even Young diagram~$\Lambda$ in \deframe, there is another way to read $\twist(\Lambda)\in\bbZ/2$ on the diagram. Add the (parity of) the length of the segments where $\Lambda$ touches the right and the bottom of the frame. See Figure~\ref{twist_fig}. This is justified and generalized in Proposition~\ref{twist_prop}.
\end{rema}

\begin{figure}[!ht]
\begin{center}
\begin{tikzpicture}[y=-1cm]

\path[draw=black,thick,fill=black!50] (-0.8,2.6) -- (0,2.6) -- (0,1.8) -- (0.8,1.8) -- (0.8,1.4) -- (-0.8,1.4) -- cycle;
\draw[thick,arrows=stealth-stealth,black] (0.9,1.4) -- (0.9,1.8);
\draw[thick,black] (-0.8,1.4) rectangle (0.8,3);
\path (0.9,1.7) node[text=black,anchor=base west] {$1$};
\path[draw=black,thick,fill=black!50] (1.6,3) -- (2.4,3) -- (2.4,1.4) -- (1.6,1.4) -- (1.6,2.2);
\draw[thick,black] (1.6,1.4) rectangle (3.2,3);
\draw[thick,arrows=stealth-stealth,black] (1.6,3.1) -- (2.4,3.1);
\path (1.8,3.4) node[text=black,anchor=base west] {$2$};
\draw[thick,black] (4,1.4) rectangle (5.2,2.6);
\path[draw=black,thick,fill=black!50] (4.4,2.6) -- (4.4,1.8) -- (5.2,1.8) -- (5.2,1.4) -- (4,1.4) -- (4,2.6);
\draw[thick,arrows=stealth-stealth,black] (4,2.7) -- (4.4,2.7);
\draw[thick,arrows=stealth-stealth,black] (5.3,1.4) -- (5.3,1.8);
\path (4,3.05) node[text=black,anchor=base west] {$1$};
\path (5.3,1.7) node[text=black,anchor=base west] {$1$};
\path[draw=black,thick,fill=black!50] (6,1.4) rectangle (8,3);
\draw[thick,arrows=stealth-stealth,black] (6,3.1) -- (8,3.1);
\draw[thick,arrows=stealth-stealth,black] (8.1,1.4) -- (8.1,3);
\path (6.9,3.4) node[text=black,anchor=base west] {$5$};
\path (8.1,2.3) node[text=black,anchor=base west] {$4$};
\draw[thick,black] (-3.2,1.4) rectangle (-1.6,3);
\path (-3.1,4.1) node[text=black,anchor=base west] {$\twist(\Lambda)=0$};
\path (-0.7,4.1) node[text=black,anchor=base west] {$\twist(\Lambda)=1$};
\path (3.9,4.1) node[text=black,anchor=base west] {$\twist(\Lambda)=0$};
\path (1.7,4.1) node[text=black,anchor=base west] {$\twist(\Lambda)=0$};
\path (6.3,4.1) node[text=black,anchor=base west] {$\twist(\Lambda)=1$};
\end{tikzpicture}
\end{center}
\caption{Class $\twist(\Lambda)\in\bbZ/2$, for different $\Lambda$.}
\label{twist_fig}%
\end{figure}
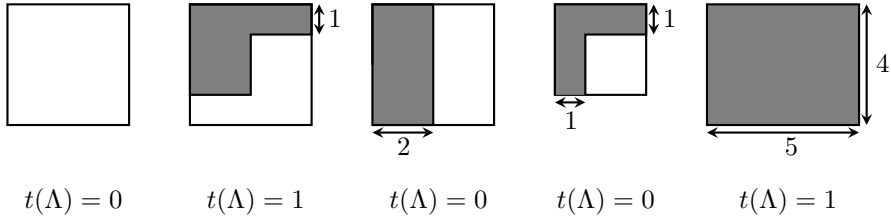

\begin{prop}
\label{twist_prop}%
Let $\Lambda$ be an even Young diagram in \deframe\ and let $\multi{d},\multi{e}$ be the associated $k$-tuples (Def.\,\ref{de_def}). Then $\twist(\Lambda)=[d_i+(e-e_j)]\in\bbZ/2$ for any $i,j\in\{1,\ldots,k\}$ such that $e_i<e$ (only $i=k$ should be avoided when $e_k=e$).
\end{prop}

\begin{proof}
Measure the half-perimeter of $\Lambda$ as the length of the lower boundary of~$\Lambda$, from the lower-left corner of $\Lambda$ to its upper-right corner (see Rem.\,\ref{perim_rem}).
Since $\Lambda$ is even, all segments on this lower half-perimeter which are not on the outside frame have even length.
So, the only two segments to contribute to the lower half-perimeter with possible odd length, are on the outside \deframe, \ie\,:
\begin{itemize}
\item the vertical segment most to the right, which has length $d_1$, and
\item the lowest horizontal segment, which has length $e-e_k$ when $\Lambda$ touches the lower part of the \deframe\ (otherwise $e_k=e$ and this length is even).
\end{itemize}
In any case, this shows that $\twist(\Lambda)=[d_1+(e-e_k)]\in\bbZ/2$, that is, the announced formula for $i=1$ and $j=k$.
The other formulas follow from this one since by Definition~\ref{even_def} the successive differences $d_i-d_{i-1}$ and $e_{j+1}-e_j$ are even for all $i=2,\ldots,k-1$, for all $j=1,\ldots,k-1$, and also for $i=k$ when $e_k<e$.
\end{proof}

\begin{defi}
\label{Twist_def}%
Let $\Lambda$ be an even Young diagram in
\deframe. We define \emph{the twist $\Twist(\Lambda)$ of
$\Lambda$} as the following class in $\Pic(\Grass_X(d,\cV))/2=\Pic(X)/2\ \oplus\ \bbZ/2\cdot \Det_d$ (see Corollary~\ref{Pic-Grass_coro} and recall that $\Det_d$ is the determinant of the tautological bundle)\,:
$$
\Twist(\Lambda)=\Twist(\Lambda,d,e):=[\det\cV]^{\row(\Lambda)}\cdot \Det_d^{\twist(\Lambda)}\,,
$$
where we recall that $\twist(\Lambda)$ is the half-perimeter
of~$\Lambda$ modulo~2 (Def.\,\ref{twist_def}) and that
$\row(\Lambda)$ is the number of non-zero rows of~$\Lambda$
(Def.\,\ref{zerow_def}).
\end{defi}

\begin{rema}
The important part of $\Twist(\Lambda)$ is of course
$\Det_d^{\twist(\Lambda)}$, which is not coming from the base~$X$.
Also, when $\cV$ is trivial, the other term disappears anyway; this
holds in particular with $X=\Spec(R)$ for a local ring~$R$ (\eg\ a field).
\end{rema}

\begin{prop}
\label{cond-even_prop}%
Let $\Lambda$ be an even Young diagram in \deframe. Then
$$
[\can_{\ff_\Lambda}]\cdot\ff_\Lambda^*\big(\Twist(\Lambda)\big)=1
\quad\text{in}\quad
\Pic\big(\Flag_X(\Lambda)\big)\big/2\,.
$$
\end{prop}

\begin{proof}
Suppose first that $k=k(\Lambda)$ is at least~$2$. Remove
from~\eqref{can-Flag_eq} all even exponents coming from the fact
that $\Lambda$ is even and use Proposition~\ref{twist_prop} for
$i=k-1$ and $j=k$. This gives in
$\Pic\big(\Flag_X(\Lambda)\big)\big/2$\,:
$$
[\can_{\ff_\Lambda}] =
[\det\cV_{d_1+e_1}]^{d_1}
 \cdot [\det\cV_{d+e_k}]^{d+d_{k-1}} \cdot [\det\cV]^d
 \cdot \Det_{d_1}^{d_1} \cdot \Det_d^{\twist(\Lambda)}\,.
$$
Now observe that $[\det\cV_{d_1+e_1}]^{d_1}\cdot \Det_{d_1}^{d_1}=1$ in $\Pic\big(\Flag_X(\Lambda)\big)\big/2$. Indeed, either $e_1>0$ hence $d_1$ is even, or $e_1=0$ hence $\Det_{d_1}=[\det\cV_{d_1}]$ by Remark~\ref{zero-Det_rem}. So, we can simplify the above equation in $\Pic\big(\Flag_X(\Lambda)\big)\big/2$\,:
$$
[\can_{\ff_\Lambda}] = [\det\cV_{d+e_k}]^{d+d_{k-1}} \cdot
[\det\cV]^{d} \cdot \Det_d^{\twist(\Lambda)}\,.
$$
Now, if $e_k<e$ then $d-d_{k-1}=d_k-d_{k-1}$ is even and
$\row(\Lambda)=d$\,; on the other hand, if $e_k=e$, then
$\row(\Lambda)=d_{k-1}$. In both cases, the above expression becomes
$[\det\cV]^{\row(\Lambda)}\cdot \Det_d^{\twist(\Lambda)}$, which is
$\Twist(\Lambda)$ by Definition~\ref{Twist_def}.

Similarly, the case $k=1$ is an easy consequence of
Corollary~\ref{can-Flag_cor}.
\end{proof}

Let $\Lambda$ be an even Young diagram in \deframe. Choose $L_\Lambda$ a line bundle of class $\Twist(\Lambda)$ in $\Pic\big(\Grass_X(d,\cV)\big)/2$. (For simplicity, choose $L_\vide=\cO$.) By the equality of Proposition~\ref{cond-even_prop}, there exists an isomorphism $\psi_\Lambda: M_\Lambda^{\otimes 2} \isoto \can_{\ff_\Lambda} \otimes \ff_\Lambda^*L_\Lambda$ over~$\Flag_X(\Lambda)$, \ie an alignment $A_\Lambda=(M_\Lambda,\psi_\Lambda)\,:\ \cO_{\Flag_X(\Lambda)}\alto \can_{\ff_\Lambda} \otimes \ff_\Lambda^*L_\Lambda$, which induces a \lax\ push-forward along the morphism $\ff_\Lambda:\Flag_X(\Lambda)\too\Grass_X(d,\cV)$, as in Definition~\ref{laxpush_defi}\,:
\begin{equation}\label{eq:gen_push_lambda}
(\ff_\Lambda)_*\circ \alis{A_\Lambda}: \Wcoh^{0}\big(\Flag_X(\Lambda),\cO\big) \too\Wcoh^{|\Lambda|}\big(\Grass_X(d,\cV),L_\Lambda\big)\,.
\end{equation}
\begin{defi}
\label{gen-Witt_def}%
For any even Young diagram $\Lambda$ in \deframe, we define%
\begin{equation}
\label{elements_eq}%
\gen_{d,e}(\Lambda)\in
\Wcoh^{|\Lambda|}\big(\Grass_X(d,\cV)\,,\,L_\Lambda\big)
\end{equation}
as the image of the unit form $1\in\Wcoh^0(\Flag_X(\Lambda),\cO)$ by the \lax\ push-forward \eqref{eq:gen_push_lambda}.
\end{defi}

\begin{rema}
\label{rem:choices}%
The definition of $\gen_{d,e}(\Lambda)$ involves the choice of $L_\Lambda$ in the class~$\Twist(\Lambda)$ and the choices of $M_\Lambda$ and of the isomorphism $\psi_\Lambda:M_\Lambda^{\otimes 2} \isoto \can_{\ff_\Lambda} \otimes \ff_\Lambda^*L_\Lambda$. By Theorem~\ref{laxpush_thm}~\eqref{it:push-unique}, the lax-similitude class of $\gen_{d,e}(\Lambda)$ is independent of these choices. By Remark~\ref{lax-basis_rema}, our Main Theorem~\ref{main_thm} (that they form a total basis) will hold regardless of these choices. For these reasons, and also to lighten notation, we do not incorporate these choices in the notation~$\gen_{d,e}(\Lambda)$.
\end{rema}

\begin{exam}
\label{empty_exa}%
Let $\Lambda=\vide$ be the empty Young diagram in \deframe, which is even, as we know. Then $T(\Lambda)$ is trivial and $\ff_\vide:\Flag_X(d,e;\vide)\to \Grass_X(d,\cV)$ is the identity, so $\can_{\ff_\Lambda}=\cO$ and $\gen_{d,e}(\vide)=1\in\Wcoh^0(\Grass_X(d,\cV),\cO)$ is the unit form on the Grassmannian if $L_\Lambda=M_\Lambda=\cO$ and the isomorphism $M_\Lambda^{\otimes 2} \simeq \can_{\ff_\Lambda} \otimes \ff_\Lambda^*L_\Lambda$ is the trivial one. (Otherwise, $\gen_{d,e}(\vide)$ is only \lax-similar to the unit form.)
\end{exam}

\begin{rema}
\label{desing_rem}%
By Proposition~\ref{push-1_prop}, the \lax-similitude class of our Witt classes $\gen_{d,e}(\Lambda)$ does not really depend on the chosen desingularization $\Flag(\Lambda)$ of the Schubert subvariety corresponding to the Young diagram~$\Lambda$. Indeed the unit will remain the unit when pushed-forward between two such desingularizations, if one is obtained from the other by blow-up. (The condition on the relative bundle being even in Proposition~\ref{push-1_prop} is automatically satisfied if both desingularizations have even relative bundle with respect to the Grassmannian.)
\end{rema}


\section{Cellular decomposition}
\label{cell-dec_sec}%


We describe the usual relative cellular decomposition of
Grassmannians. Fix $d,e\geqslant2$ for the whole section.

\begin{nota}
Fix a complete flag $\cV_\smallbullet$ of vector bundles on~$X$
$$
0=\cV_0\subb\cV_1\subb\cdots\subb \cV_{d+e}=\cV\,,
$$
as in~\eqref{complete_eq}. We set $\cV^1=\cV_{d+e-1}$ to be the chosen codimension one subbundle of~$\cV$. We have an obvious closed immersion $\Grass_X(d,\cV^1)\hookrightarrow\Grass_X(d,\cV)$, of codimension~$d$, whose open complement is denoted by~$U_X(d,\cV_\smallbullet)$.
\end{nota}

\begin{nota}
\label{nsubb_not}%
Let $\cP_d\subb\cV$ be a subbundle of rank~$d$. We write $\cP_d\nsubb\cV^1$ to express that $\cP_d$ is not a subbundle of $\cV^1$ but moreover satisfies the equivalent conditions\,:
\begin{enumerate}
\item The natural map from
$\cP_d/(\cP_d\cap\cV^1)=(\cP_d+\cV^1)/\cV^1$ into $\cV/\cV^1$ is an
isomorphism.
\item $\cP_d\cap\cV^1$ is a subbundle of $\cP_d$ (in the strong
sense of Definition~\ref{subb_def}).
\end{enumerate}
Over a field, this amounts to $\cP_d\not\subset\cV^1$ but this is not sufficient in general.
\end{nota}

\begin{defi}
\label{main'_def}%
Using notation of Section~\ref{Grass_sec}, we have a commutative diagram
\begin{equation}
\label{main'_eq}%
\vcenter{\xymatrix@C=2em{
 \Grass_X(d,\cV^1) \ar@{^{(}->}[r]^(.57){\iota}
 \ar@{}[r]|(.4){|}
& \Grass_X(d,\cV)
& \ \ U_X(d,\cV_\smallbullet)
 \ar@{_{(}->}[l]_-{\upsilon} \ar@{}[l]|-{\circ}
 \ar[d]^{\alpha}\ar@{_{(}->}[dl]_-{\tilde\upsilon}\ar@{}[dl]|-{\circ}
\\
\Flag_{X}((d-1,d),(e,e-1)) \ar[u]^{\tilde{\pi}} \ar@{^{(}->}[r]^(.5){\tilde{\iota}} \ar@{}[r]|(.53){|} &
\Flag_{X}((d-1,d),(e,e)) \ar[u]^{\pi} \ar[r]^-{\tilde{\alpha}} &
\Flag_{X}(d-1,\cV^1)\,.
}}\kern-.7em
\end{equation}
which looks as follows on points\,:
\begin{equation}
\label{main-pts_eq}%
\vcenter{\xymatrix{
\{\cP_d \subb \cV^1 \} \ar@{^{(}->}[r]^{\iota} \ar@{}[r]|(.4){|} &
\{\cP_d \subb \cV \} &
\ \{\cP_d \nsubb \cV^1 \}
 \ar@{_{(}->}[l]_{\upsilon}\ar@{}[l]|-{\circ}
 \ar@{_{(}->}[dl]_-{\tilde\upsilon}\ar@{}[dl]|-{\circ}
 \ar[d]^{\alpha}
\\
\{\cP_{d-1} \subb \cP_{d} \subb \cV^1\} \ar[u]^-{\tilde{\pi}}
 \ar@{^{(}->}[r]^(.37){\tilde{\iota}} \ar@{}[r]|(.4){|} &
\{\cP_{d-1} \subb \cP_d \subb \cV \;|\; \cP_{d-1} \subb \cV^1 \}
 \ar[u]^{\pi} \ar[r]^-{\tilde{\alpha}} &
\{ \cP_{d-1} \subb \cV^1 \}\,.}}\kern-1em
\end{equation}
Here $\iota$, $\tilde\iota$, $\pi$, $\tilde\pi$ and $\tilde\alpha$ are the obvious morphisms. The morphism $\tilde\upsilon$ maps $\cP_d$ to the flag $\cP_{d-1}\subb\cP_d$ with $\cP_{d-1}:=\cP_d \cap \cV^1$ (see Not.\,\ref{nsubb_not}). Finally $\alpha$ is defined as $\tilde\alpha\circ\tilde\upsilon$.
\end{defi}

\begin{prop}
\label{blowup_prop}%
In Diagram~\eqref{main'_eq}, the scheme $\Flag_X((d-1,d),(e,e))$ is the blow-up of $\Grass_X(d,\cV)$ along $\Grass_X(d,\cV^1)$ with exceptional fiber $\Flag_X((d-1,d),(e,e-1))$.
\end{prop}

\begin{proof}
This is probably folklore to algebraic geometers. By compatibility of blow-ups with pull-backs, we can reduce to the case where $X$ is affine (even $X=\Spec(\bbZ)$) and suppose that $\cV$ is free and that $\cV^1=\ker(\cV\onto\cO)$ is the kernel of a (split) epimorphism to~$\cO$. We omit $X$ in the notation for the rest of the proof.

Let us check that $\Bl:=\Flag((d-1,d),(e,e))$ has the universal property of the blow-up (see \cite[§ 8.1.2, Corollary~1.16]{Liu02}), \ie it is final among schemes over $\Grass(d,\cV)$ in which the preimage of $Z:=\Grass(d,\cV^1)$ is an effective Cartier divisor (\ie a codimension one closed subscheme locally given by a principal ideal). Let us first check that $B$ indeed has this property. Note that the left-hand square of~\eqref{main'_eq} is cartesian. Moreover, we have an identification $B=\Flag((d-1,d),(e,e))=\bbP_Y (\cV/\Taut_{d-1})$ where $Y=\Grass(d-1,\cV^1)$ as in Lemma~\ref{rel_lem}. Under this identification, the inverse image $\Flag((d-1,d),(e,e-1))$ of~$Z$ becomes $\bbP_Y(\cV^1/\Taut_{d-1})$. So this inverse image is locally $\bbP^{e-1}_Y \subset \bbP^e_Y$ hence an effective Cartier divisor, as wanted.

Suppose now that $f:W\to \Grass_X(d,\cV)$ is a morphism for which $f\inv(\Grass_X(d,\cV^1))$ is an effective Cartier divisor. Let us consider $W$ as a functor of points and show that there exists a unique morphism $g:W\to B$ such that $\tilde\pi\circ g=f$. Consider on $\Grass(d,\cV)$, the morphism $s:\Taut_d\to \cO$ obtained by composing the inclusion $\Taut_d\into\cV$ and the projection $\cV\onto \cV/\cV^1=\cO$. By definition, $Z$ is the zero locus of this morphism, \ie it is defined by the ideal $\im(s)\subset\cO_{\Grass(d,\cV)}$. By assumption on $f:W\to \Grass_X(d,\cV)$, the ideal $\im(f^*(s))=f\inv(\im(s))\subset\cO_{W}$ is invertible. Hence, over each point $\Spec(R)\to W$ of~$W$, $\ker(s\restr{R})$ is a codimension one subbundle of~$\Taut_d$, which is contained in~$\cV^1$ by construction. This defines the wanted morphism $g:W\to B$ sending each $\Spec(R)\to W$ to $\ker(s\restr{R})\subb \cP_d$. Uniqueness is easy. If $g:W\to B$ satisfies $g\circ\tilde\pi=f$ then $g$ maps a point $w:\Spec(R)\to W$ to $\cP_{d-1}\subb\cP_d$ with $\cP_d=f(w)$ forced. On the other hand $\cP_{d-1}\subset \cV^1$ forces $\cP_{d-1}$ to be in $\ker(s\restr{R})$, hence to be equal to it by dimension counting.
\end{proof}

\begin{defi}
\label{cell_def}%
Let $\Bl_X(d,\cV_\smallbullet)=\Flag_X((d-1,d),(e,e))$ be the
blow-up of $\Grass_X(d,\cV)$ along $\Grass_X(d, \cV^1)$ and let
$\Exc_X(d,\cV_\smallbullet)=\Flag_X((d-1,d),(e,e-1))$ be the
exceptional fiber. By~\eqref{main'_eq}, $\Grass_X(d,\cV)$ now has a
decomposition like in~\cite[Hypothesis 1.2]{Balmer09}, namely there exists an auxiliary morphism $\tilde\alpha:\Bl_X(d,\cV)\to Y:=\Grass_X(d-1,\cV^1)$ from the blow-up to another scheme~$Y$, such that $\alpha:=\tilde\alpha\circ\tilde\upsilon$ is an $\bbA^*$-bundle\,:
\begin{equation}
\label{main_eq}%
\vcenter{\xymatrix@C=5em{
\Grass_X(d,\cV^1)
 \ar@{^{(}->}[r]^(.55){\iota} \ar@{}[r]|(.4){|} &
\Grass_X(d,\cV) &
\ U_X(d,\cV_\smallbullet) \ar@{_{(}->}[l]_-{\upsilon}\ar@{}[l]|-{\circ}
\ar@{_{(}->}[dl]_-{\tilde\upsilon}\ar@{}[dl]|-{\circ}
\ar[d]^{\alpha}
\\
\Exc_X(d,\cV_\smallbullet) \ar[u]^{\tilde{\pi}}
\ar@{^{(}->}[r]^(.5){\tilde{\iota}} \ar@{}[r]|(.4){|} &
\Bl_X(d,\cV_\smallbullet) \ar[u]^{\pi} \ar[r]^-{\tilde{\alpha}} &
\Grass_X(d-1,\cV^1)\,.\! }}
\end{equation}
Indeed, $\alpha$ is an $\bbA^e$-bundle because of the canonical isomorphism between $U_X(d,\cV_{\smallbullet})$ and $\bbP_Y(\cV/\Taut^Y_{d-1})\smallsetminus \bbP_Y(\cV^1/\Taut^Y_{d-1})$, under which $\alpha$ corresponds to the structure morphism to~$Y$.
\end{defi}

\begin{rema}
\label{Pic-can_rem}%
We compute the relevant Picard groups and canonical bundles, via the methods of Section~\ref{Grass_sec}. Let us start with Picard groups, using~\eqref{Pic-Flag_eq}. Since $\Pic(X)$ is a direct summand of the Picard group of all schemes in~\eqref{main_eq}, we focus on the relative Picard groups $\Pic_X(-):=\Pic(-)/\Pic(X)$. Then ``$\Pic_X(-)$ of~\eqref{main_eq}" equals
\begin{equation}
\label{Picards_eq}%
\vcenter{\xymatrix@C=4em{
\bbZ \Det_d \ar[d]_-{\tilde\pi^*=\left(\begin{smallmatrix}0\\1\end{smallmatrix}\right)}
& \bbZ \Det_d \ar[l]_-{\iota^*=1} \ar[d]_-{\pi^*=\left(\begin{smallmatrix}0\\1\end{smallmatrix}\right)} \ar[r]^-{\upsilon^*=1}
& \bbZ\upsilon^*(\Det_d)=\bbZ \alpha^*(\Det_{d-1})\kern-1em
\\
\bbZ \Det_{d-1} \oplus \bbZ \Det_d
& \bbZ \Det_{d-1} \oplus \bbZ \Det_d
 \ar[l]^{\tilde\iota^*=\left(\begin{smallmatrix}1&0\\0&1\end{smallmatrix}\right)} \ar[ru]|-{\tilde\upsilon^*=\left(\begin{smallmatrix}1& 1\end{smallmatrix}\right)}
& \bbZ \Det_{d-1} \ar[l]^-{\tilde\alpha^*=\left(\begin{smallmatrix}1\\0\end{smallmatrix}\right)} \ar[u]_{\alpha^*=1}
}}
\end{equation}
(In the case $X=\Spec(R)$ for a local ring $R$, the Picard groups
are exactly as above.) Here we used that the closed subscheme
$\Grass_X(d,\cV^1)$ is of codimension $d\geqslant2$ in
$\Grass_X(d,\cV)$ to see that
$\upsilon^*:\Pic(\Grass_X(d,\cV))\cong\Pic(U_X(d,\cV_\smallbullet))$.
We also used that $e\geqslant2$, otherwise $e-1=0$ and
$\Det_{d-1}\in\Pic(X)$ by Remark~\ref{zero-Det_rem}. (When $d=1$,
resp.\ $e=1$, we loose all components $\bbZ \Det_{d-1}$, resp.\ all
components $\bbZ \Det_d$ in the left column, in the previous
diagram.) Alternatively, ``$\Pic_X(\eqref{main_eq})=\eqref{Picards_eq}$" follows from the
computation of the Picard groups provided
in~\cite[Proposition~A.6]{Balmer09}. Finally, the maps into the
upper right corner of Diagram~\eqref{Picards_eq} are deduced from
\begin{equation}
\label{alpha-upsilon-Pic_eq}%
\upsilon^*(\Det_d) \cdot \alpha^*(\Det_{d-1})^{-1} =  [\cV/\cV^1]\,,
\end{equation}
which itself follows from Condition~(a) in Notation~\ref{nsubb_not}.

We shall use push-forwards along some morphisms of~\eqref{main_eq}. The classes of the relevant relative canonical bundles in the respective (plain) Picard groups are\,:
\begin{eqnarray}
[\can_\iota] & = & [\cV/\cV^1]^d \cdot \Det_d^{-1} \label{rel-can-subGrass_eq}%
\\
\left[\can_\pi\right] & = & [\cV/\cV^1]^{d-1} \cdot \Det_{d-1}^{d-1} \cdot \Det_d^{1-d} \label{rel-can-Bl_eq}%
\\
\left[\can_{\tilde\iota}\right] & = & [\cV/\cV^1]\cdot \Det_{d-1} \cdot \Det_d^{-1} \label{rel-can-Exc-div_eq}%
\\
\left[\can_{\tilde{\pi}}\right] & = & \Det_{d-1}^d \cdot \Det_d^{1-d}\,.\!
\label{rel-can-Exc_eq}
\end{eqnarray}
Indeed, Corollary~\ref{can-Flag_cor} gives~\eqref{rel-can-subGrass_eq}, \eqref{rel-can-Bl_eq} and $[\can_{\iota\tilde\pi}]=[\can_{\pi\tilde\iota}]=[\cV/\cV^1]^{d}\cdot\Det_{d-1}^{d}\cdot\Det_{d}^{-d}$, out of which the other two follow by multiplicativity of~$\can_{-}$. Again, when the fixed complete flag $\cV_\smallbullet=\cO_X^\smallbullet$ is trivial, the ``noise'' $\cV/\cV^1$ vanishes.
\end{rema}

We end this section with two geometric lemmas which will be useful
in the proof of the main theorem.

\begin{lemm}
\label{double-PB_lem}%
Let $d,e\geqslant2$ and let $\Lambda$ be an even Young
$(d,e)$-diagram with empty last row (\ie $\Lambda_k=0$, \ie
$\zero(\Lambda)>0$). Hence $\Lambda\restr{d-1,e}$ is an even
$(d-1,e)$-diagram. Then the base-changes to
$U_X(d,\cV_\smallbullet)$ of the morphisms $\ff_{d,e;\Lambda}$ and
$\ff_{d-1,e;\Lambda\smallrestr{d-1,e}}$ coincide, that is, we have
two cartesian squares
\begin{equation}
\label{double-PB_eq}%
\vcenter{ \xymatrix{ \Grass_X(d,\cV) \ar@{}[dr]|{\Box}
& \ U_X(d,\cV_\smallbullet)
 \ar@{_{(}->}[l]_-{\upsilon} \ar[r]^-{\alpha} \ar@{}[dr]|{\Box}
& \Grass_X(d-1,\cV^1)
\\
\Flag_X(d,e;\Lambda) \ar[u]^{\ff_{d,e;\Lambda}}
& \ U' \ar@{_{(}->}[l] \ar[r]^-{\alpha'} \ar[u]
& \Flag_{X}(d-1,e;\Lambda\restr{d-1,e})\,.\! \ar[u]_{\ff_{d-1,e;\Lambda\smallrestr{d-1,e}}}
}}
\end{equation}
\end{lemm}

\begin{proof}
Let us check this on points. Let $\multi{d}$ and $\multi{e}$ be the
$k$-tuples associated to $\Lambda$ as usual (Def.\,\ref{de_def}). We
need to distinguish two cases, namely $d_{k}-d_{k-1}>1$ and $d_k-d_{k-1}=1$.

When $d_k>d_{k-1}+1$, that is, when there is more than one zero line
at the end of $\Lambda$ (\ie $\zero(\Lambda)>1$), we then have
$k(\Lambda\restr{d-1,e})=k(\Lambda)=k$ and the $k$-tuples
$\multi{d}(\Lambda\restr{d-1,e})$ and
$\multi{e}(\Lambda\restr{d-1,e})$ are almost the same as $\multi{d}$
and $\multi{e}$ except for the last entry of
$\multi{d}(\Lambda\restr{d-1,e})$ which becomes~$d-1$.
Diagram~\eqref{double-PB_eq} then looks as follows on points (as
usual the $\cP_i$ and $\cP'_j$ are ``variables'' whereas the $\cV_i$
belong to the fixed complete flag):
$$\xy 0;<1ex,0ex>:
(7,15)*\xybox{
\xymatrix@C=0.6ex{
\{ \cP_d \ar@{}[r]|-{\subb} & \cV \}
}
};
(34,15)*\xybox{
\xymatrix@C=0.6ex{
\{ \cP_d \ar@{}[r]|-{\nsubb} & \cV^1 \}
}
};
(63,15)*\xybox{
\xymatrix@C=0.6ex{
\{ \cP'_{d-1} \ar@{}[r]|-{\subb} & \cV^1 \}
}
};
(0,0)*!(0,-2.5)\xybox{
\xymatrix@R=1.5ex@C=0.5ex{
\cdots \ar@{}[r]|-{\subb} &
\cP_{d_{k-1}} \ar@{}[d]|{\rotatebox{270}{$\subbeq$}}_*!/_.0ex/{\labelstyle e_{k-1}} \ar@{}[r]|-{\subb} & \cP_{d} \ar@{}[d]|{\rotatebox{270}{$\subbeq$}}_*!/_.5ex/{\labelstyle e} \\
\cdots \ar@{}[r]|-{\subb} & \cV_{d_{k-1}+e_{k-1}} \ar@{}[r]|-{\subb} & \cV
}
}*\frm{\{}*\frm{\}};
(28,0)*!(0,-2.5)\xybox{
\xymatrix@R=1.5ex@C=0.5ex{
\cdots \ar@{}[r]|-{\subb} &
\cP_{d_{k-1}} \ar@{}[d]|{\rotatebox{270}{$\subbeq$}}_*!/_.0ex/{\labelstyle e_{k-1}} \ar@{}[r]|-{\subb} & \cP_{d} \ar@{}[d]|{\rotatebox{270}{$\subbeq$}}_*!/_.5ex/{\labelstyle e} \\
\cdots \ar@{}[r]|-{\subb} & \cV_{d_{k-1}+e_{k-1}} \ar@{}[r]|-{\subb} & \cV
}
}*\frm{\{}*\frm{\}};
(45.3,7)*\xybox{
\xymatrix@R=1.5ex{
\cV^1 \\
\ar@{}[u]|{\rotatebox{90}{$\nsubb$}}
}
};
(56,0)*!(0,-2.5)\xybox{
\xymatrix@R=1.5ex@C=0.5ex{
\cdots \ar@{}[r]|-{\subb} & \cP_{d_{k-1}} \ar@{}[d]|{\rotatebox{270}{$\subbeq$}}_*!/_.0ex/{\labelstyle e_{k-1}} \ar@{}[r]|-{\subb} & \cP'_{d-1}
\ar@{}[d]|{\rotatebox{270}{$\subbeq$}}_*!/_.5ex/{\labelstyle e} \\
\cdots \ar@{}[r]|-{\subb} &
\cV_{d_{k-1}+e_{k-1}} \ar@{}[r]|-{\subb} & \cV^1
}
}*\frm{\{}*\frm{\}};
{\ar@{_{(}->} (30,15)*{};(15,15)*{}};
{\ar (43,15)*{};(59,15)*{} ^-{\alpha}};
{\ar (10,5)*{};(10,13)*{}^{\ff_{d,e;\Lambda}}};
{\ar (38,5)*{};(38,13)*{}};
{\ar (65,5)*{};(65,13)*{}|{\ff_{d-1,e;\Lambda\smallrestr{d-1,e}}}};
{\ar@{_{(}->} (24,0)*{};(20.5,0)*{}};
{\ar (48.5,0)*{};(53,0)*{}^-{\alpha'}};
(23,9)*{\labelstyle \Box};
(51,9)*{\labelstyle \Box};
\endxy$$
where the morphisms $\alpha$ sends $\cP_d$ to $\cP'_{d-1}:=\cP_d \cap
\cV^1$ and similarly for~$\alpha'$ .

On the other hand, when $d_k=d_{k-1}+1$, that is, when $\Lambda$ has
only one zero line (\ie $\zero(\Lambda)=1$), then we have
$k(\Lambda\restr{d-1,e})=k(\Lambda)-1=k-1$ and the $(k-1)$-tuples
$\multi{d}(\Lambda\restr{d-1,e})$ and
$\multi{e}(\Lambda\restr{d-1,e})$ are respectively $\multi{d}$ and
$\multi{e}$ truncated from their last entry.
Diagram~\eqref{double-PB_eq} then looks as follows on points\,:
$$
\xy 0;<1ex,0ex>:
(6,15)*\xybox{
\xymatrix@C=0.6ex{
\{ \cP_d \ar@{}[r]|-{\subb} & \cV \}
}
};
(37,15)*\xybox{
\xymatrix@C=0.6ex{
\{ \cP_d \ar@{}[r]|-{\nsubb} & \cV^1 \}
}
};
(66,15)*\xybox{
\xymatrix@C=0.6ex{
\{ \cP'_{d-1} \ar@{}[r]|-{\subb} & \cV^1 \}
}
};
(0,0)*!(0,-2.5)\xybox{
\xymatrix@R=1.5ex@C=0.5ex{
\cdots \ar@{}[r]|-{\subb} &
\cP_{d_{k-1}} \ar@{}[d]|{\rotatebox{270}{$\subbeq$}}_*!/_.0ex/{\labelstyle e_{k-1}} \ar@{}[r]|-{\subb} & \cP_{d} \ar@{}[d]|{\rotatebox{270}{$\subbeq$}}_*!/_.5ex/{\labelstyle e} \\
\cdots \ar@{}[r]|-{\subb} & \cV_{d_{k-1}+e_{k-1}} \ar@{}[r]|-{\subb} & \cV
}
}*\frm{\{}*\frm{\}};
(31,0)*!(0,-2.5)\xybox{
\xymatrix@R=1.5ex@C=0.5ex{
\cdots \ar@{}[r]|-{\subb} &
\cP_{d_{k-1}} \ar@{}[d]|{\rotatebox{270}{$\subbeq$}}_*!/_.0ex/{\labelstyle e_{k-1}} \ar@{}[r]|-{\subb} & \cP_{d} \ar@{}[d]|{\rotatebox{270}{$\subbeq$}}_*!/_.5ex/{\labelstyle e} \\
\cdots \ar@{}[r]|-{\subb} & \cV_{d_{k-1}+e_{k-1}} \ar@{}[r]|-{\subb} & \cV
}
}*\frm{\{}*\frm{\}};
(48.3,7)*\xybox{
\xymatrix@R=1.5ex{
\cV^1 \\
\ar@{}[u]|{\rotatebox{90}{$\nsubb$}}
}
};
(62,0)*!(0,-2.5)\xybox{
\xymatrix@R=1.5ex@C=0.5ex{
\cdots \ar@{}[r]|-{\subb} & \cP_{d_{k-1}}
\ar@{}[d]|{\rotatebox{270}{$\subbeq$}}^*!/^.5ex/{\labelstyle
e_{k-1}} \\
\cdots \ar@{}[r]|-{\subb} & \cV_{d_{k-1}+e_{k-1}}
}
}*\frm{\{}*\frm{\}};
{\ar@{_{(}->} (33,15)*{};(14,15)*{}};
{\ar (45,15)*{};(62,15)*{} ^-{\alpha}};
{\ar (9,5)*{};(9,13)*{}^{\ff_{d,e;\Lambda}}};
{\ar (40,5)*{};(40,13)*{}};
{\ar (69,5)*{};(69,13)*{}|{\ff_{d-1,e;\Lambda\smallrestr{d-1,e}}}};
{\ar@{_{(}->} (27,0)*{};(21,0)*{}};
{\ar (52,0)*{};(58.5,0)*{}^-{\alpha'}};
(24,9)*{\labelstyle \Box};
(54,9)*{\labelstyle \Box};
\endxy
$$
where $\alpha$ still sends $\cP_d$ to $\cP'_{d-1}:=\cP_d \cap \cV^1$
and where $\ff_{d-1,e;\Lambda\smallrestr{d-1,e}}$ sends a flag
to~$\cP_{d_{k-1}}$. Note that in this case, $\alpha'$ drops the last
subspace~$\cP_d$ in the flag.

\smallbreak

In both cases, it is easy to check that the two squares are
cartesian.
\end{proof}

\begin{lemm}
\label{PB_lem}%
Let $d,e\geqslant2$ and let $\Lambda''$ be an even $(d-1,e)$-diagram
such that $\Lambda''_{d-1}$ is odd. Hence we can consider the even
$(d,e-1)$-diagram
$\Lambda'=(\Lambda''_1-1,\ldots,\Lambda''_{d-1}-1,0)$. Then, there
exists a commutative diagram
\begin{equation}
\label{PB_eq}%
\vcenter{\xymatrix{ \Grass_X(d,\cV^1)
& \Exc_X(d,\cV_\smallbullet) \ar[r]^-{\tilde\alpha\,\tilde\iota} \ar[l]_-{\tilde\pi}
& \Grass_X(d-1,\cV^1)
\\
\Flag_X(d,e-1;\Lambda') \ar[u]^{\ff_{d,e-1;\Lambda'}}
& F' \ar[u]^-{f'} \ar[r] \ar[l]^-{\pi'} \ar@{}[ru]|{\Box}
& \Flag_X(d-1,e;\Lambda'') \ar[u]_{\ff_{d-1,e;\Lambda''}}
}}
\end{equation}
where $\Exc_X(d,\cV_\smallbullet)$ is the exceptional fiber of
Diagram~\eqref{main_eq} and where the right-hand square is
cartesian. Moreover, either $\pi'$ is an isomorphism or the scheme
$F'$ (with the morphism~$\pi'$) identifies with the blow-up of
$\Flag_X(d,e-1;\Lambda')$ along a closed regular subscheme of odd
codimension.
\end{lemm}

\begin{proof}
Let $k=k(\Lambda'')$,
$\multi{d}=\multi{d}(\Lambda'')$ and
$\multi{e}=\multi{e}(\Lambda'')$ as usual (Def.\,\ref{de_def}). We
need to distinguish two cases, namely $\Lambda''_{d-1}>1$ and $\Lambda''_{d-1}=1$.

Suppose first that $\Lambda''_{d-1}>1$. Then $k(\Lambda')=k+1$ and
$\multi{d}(\Lambda')$ and $\multi{e}(\Lambda')$ are just
$\multi{d}(\Lambda'')$ and $\multi{e}(\Lambda'')$ with one more
entry at the end, namely $d$ and $e-1$ respectively. We can
describe the pull-back in Diagram~\eqref{PB_eq} as follows (on
points):
$$\xy 0;<1ex,0ex>:
(4,15)*\xybox{
\xymatrix@C=0.6ex{
\{ \cP_{d-1} \ar@{}[r]|-{\subb} & \cP_d \ar@{}[r]|-{\subb} & \cV^1 \}
}
};
(45,15)*\xybox{
\xymatrix@C=0.6ex{
\{ \cP_{d-1} \ar@{}[r]|-{\subb} & \cV^1 \}
}
};
(0,0)*!(0,-2.5)\xybox{
\xymatrix@R=1.5ex@C=0.5ex{
\cP_{d_{1}} \ar@{}[d]|{\rotatebox{270}{$\subbeq$}} \ar@{}[r]|-{\subb} & \cdots \ar@{}[r]|-{\subb} & \cP_{d-1} \ar@{}[d]|{\rotatebox{270}{$\subbeq$}} \ar@{}[r]|-{\subb} & \cP_{d} \ar@{}[d]|{\rotatebox{270}{$\subbeq$}} \\
\cV_{d_1+e_1} \ar@{}[r]|-{\subb} & \cdots \ar@{}[r]|-{\subb} & \cV_{d-1+e_{k}} \ar@{}[r]|-{\subb} & \cV^1
}
}*\frm{\{}*\frm{\}};
(40,0)*!(0,-2.5)\xybox{
\xymatrix@R=1.5ex@C=0.5ex{
\cP_{d_{1}} \ar@{}[d]|{\rotatebox{270}{$\subbeq$}} \ar@{}[r]|-{\subb} & \cdots \ar@{}[r]|-{\subb} & \cP_{d-1} \ar@{}[d]|{\rotatebox{270}{$\subbeq$}} \\
\cV_{d_1+e_1} \ar@{}[r]|-{\subb} & \cdots \ar@{}[r]|-{\subb} & \cV_{d-1+e_{k}}
}
}*\frm{\{}*\frm{\}};
{\ar (9,5)*{};(9,13)*{}^{f'}};
{\ar (48,5)*{};(48,13)*{}};
{\ar (18,15)*{};(41.5,15)*{}};
{\ar (24.5,0)*{};(34.5,0)*{}};
(30,9)*{\labelstyle \Box};
\endxy$$

This pull-back $F'$ is $\Flag_X(d,e-1;\Lambda')$. So, take
$\pi'=\id$. The morphism $f'$ composed with $\tilde\pi$ sends a flag
$\cP_{d_1}\subb\cdots\subb\cP_{d-1}\subb\cP_d$ to $\cP_d$ and so
does~$\ff_{d,e-1;\Lambda'}$.

\smallbreak

Suppose now that $\Lambda''_{d-1}=1$. Then $k(\Lambda')=k$ and
$\multi{e}(\Lambda')=\multi{e}(\Lambda'')$, whereas
$\multi{d}(\Lambda')$ is obtained from $\multi{d}(\Lambda'')$ by
replacing its last entry by $d$. We can describe the pull-back in
Diagram~\eqref{PB_eq} as follows\,:
$$\xy 0;<1ex,0ex>:
(11,15)*\xybox{
\xymatrix@C=0.6ex{
\{ \cP_{d-1} \ar@{}[r]|-{\subb} & \cP_d \ar@{}[r]|-{\subb} & \cV^1 \}
}
};
(54,15)*\xybox{
\xymatrix@C=0.6ex{
\{ \cP_{d-1} \ar@{}[r]|-{\subb} & \cV^1 \}
}
};
(0,0)*!(0,-2.5)\xybox{
\xymatrix@R=1.5ex@C=0.5ex{
\cP_{d_{1}} \ar@{}[d]|{\rotatebox{270}{$\subbeq$}} \ar@{}[r]|-{\subb} & \cdots \ar@{}[r]|-{\subb} & \cP_{d_{k-1}} \ar@{}[d]|{\rotatebox{270}{$\subbeq$}} \ar@{}[r]|-{\subb} & \cP_{d-1} \ar@{}[d]|{\rotatebox{270}{$\subbeq$}} \ar@{}[r]|-{\subb} & \cP_{d} \ar@{}[d]|{\rotatebox{270}{$\subbeq$}} \\
\cV_{d_1+e_1} \ar@{}[r]|-{\subb} & \cdots \ar@{}[r]|-{\subb} &
\cV_{d_{k-1}+e_{k-1}} \ar@{}[r]|-{\subb} & \cV_{d+e-2}
\ar@{}[r]|-{\subb} & \cV^1
}
}*\frm{\{}*\frm{\}};
(50,0)*!(0,-2.5)\xybox{
\xymatrix@R=1.5ex@C=0.5ex{
\cP_{d_{1}} \ar@{}[d]|{\rotatebox{270}{$\subbeq$}} \ar@{}[r]|-{\subb} & \cdots \ar@{}[r]|-{\subb} & \cP_{d-1} \ar@{}[d]|{\rotatebox{270}{$\subbeq$}} \\
\cV_{d_1+e_1} \ar@{}[r]|-{\subb} & \cdots \ar@{}[r]|-{\subb} & \cV_{d+e-2}
}
}*\frm{\{}*\frm{\}};
{\ar (17,5)*{};(17,13)*{}^{f'}};
{\ar (57,5)*{};(57,13)*{}};
{\ar (25,15)*{};(50,15)*{}};
{\ar (36.5,0)*{};(44.5,0)*{}};
(38,9)*{\labelstyle \Box};
\endxy$$%
where $f'$ is the obvious morphism. The left-hand square of~\eqref{PB_eq} is defined by\,:
$$\xy 0;<1ex,0ex>:
(8,15)*\xybox{
\xymatrix@C=0.6ex{
\{ \cP_d \ar@{}[r]|-{\subb} & \cV^1 \}
}
};
(51,15)*\xybox{
\xymatrix@C=0.6ex{
\{ \cP_{d-1} \ar@{}[r]|-{\subb} & \cP_d \ar@{}[r]|-{\subb} & \cV^1 \}
}
};
(0,0)*!(0,-2.5)\xybox{
\xymatrix@R=1.5ex@C=0.5ex{
\cP_{d_{1}} \ar@{}[d]|{\rotatebox{270}{$\subbeq$}} \ar@{}[r]|-{\subb} & \cdots \ar@{}[r]|-{\subb} & \cP_{d_{k-1}} \ar@{}[d]|{\rotatebox{270}{$\subbeq$}} \ar@{}[r]|-{\subb} & \cP_d \ar@{}[d]|{\rotatebox{270}{$\subbeq$}} \\
\cV_{d_1+e_1} \ar@{}[r]|-{\subb} & \cdots \ar@{}[r]|-{\subb} & \cV_{d_{k-1}+e_{k-1}} \ar@{}[r]|-{\subb} & \cV^1}
}*\frm{\{}*\frm{\}};
(40,0)*!(0,-2.5)\xybox{
\xymatrix@R=1.5ex@C=0.5ex{
\cP_{d_{1}} \ar@{}[d]|{\rotatebox{270}{$\subbeq$}} \ar@{}[r]|-{\subb} & \cdots \ar@{}[r]|-{\subb} & \cP_{d_{k-1}} \ar@{}[d]|{\rotatebox{270}{$\subbeq$}} \ar@{}[r]|-{\subb} & \cP_{d-1} \ar@{}[d]|{\rotatebox{270}{$\subbeq$}} \ar@{}[r]|-{\subb} & \cP_{d} \ar@{}[d]|{\rotatebox{270}{$\subbeq$}} \\
\cV_{d_1+e_1} \ar@{}[r]|-{\subb} & \cdots \ar@{}[r]|-{\subb} &
\cV_{d_{k-1}+e_{k-1}} \ar@{}[r]|-{\subb} & \cV_{d+e-2}
\ar@{}[r]|-{\subb} & \cV^1
}
}*\frm{\{}*\frm{\}};
{\ar (11,5)*{};(11,13)*{}};
{\ar (57,5)*{};(57,13)*{}^{f'}};
{\ar (47,15)*{};(17,15)*{}_-{\tilde\pi}};
{\ar (34.5,0)*{};(27.5,0)*{}_-{\pi'}};
\endxy$$
The morphism $\pi':F'\to\Flag_X(d,e-1;\Lambda')$ simply drops $\cP_{d-1}$ in this case.

Let $\cV^2:=\cV_{d+e-2}$ and
$Y:=\Flag_X((d_1,\ldots,d_{k-1}),(e_1,\ldots,e_{k-1}),\cV_\smallbullet)$.
We have $\Flag_X(d,e-1;\Lambda')=\Grass_Y(d-d_{k-1},\cV^1/\Taut_{d_{k-1}})$ by Lemma~\ref{rel_lem}.
As in Definition~\ref{cell_def}, we consider the blow-up
$\Bl_Y(d-d_{k-1}\,,\,\cV^2/\Taut_{d_k-1}\subb\cV^1/\Taut_{d_k-1})$
of $\Grass_Y(d-d_{k-1},\cV^1/\Taut_{d_{k-1}})$ along the closed
regular immersion of~$\Grass_Y(d-d_{k-1},\cV^2/\Taut_{d_{k-1}})$. By
Proposition~\ref{blowup_prop}, this blow-up coincides with~$F'$ and
the morphism~$\pi$ of~\eqref{main-pts_eq} here becomes the above
morphism~$\pi'$. In other words, the following diagram commutes\,:
$$
\xymatrix{
\Flag_X(d,e-1;\Lambda') \ar@{=}[d]_{}
& F' \ar@{=}[d] \ar[l]_-{\pi'}
\\
\Grass_Y(d-d_{k-1},\cV^1/\Taut_{d_{k-1}})
& \Bl_Y(d-d_{k-1}, \,\cV^2/\Taut_{d_k-1}\subb\cV^1/\Taut_{d_{k-1}})\,.\! \ar[l]_-{\text{``}\pi\text{''}}
}
$$
The closed immersion $\Grass_Y(d-d_{k-1},\cV^2/\Taut_{d_{k-1}})\hookrightarrow \Grass_Y(d-d_{k-1},\cV^1/\Taut_{d_{k-1}})$ is of odd codimension equal to~$d-d_{k-1}$.
\end{proof}


\section{Main result}
\label{main_sec}%


We are now ready to state and prove the main result of the paper.

\begin{theo} \label{main_thm}%
Let $d,e\geqslant1$. Let $X$ be a regular scheme over~$\bbZ[\half]$.
Let $\cV$ be a vector bundle of rank $d+e$ which admits a complete
flag~\eqref{complete_eq} of subbundles (for instance, $\cV$
free). Then, the elements $\gen_{d,e}(\Lambda)$ of
Definition~\ref{gen-Witt_def}, for $\Lambda$ even, form a total
basis of the Witt groups of $\Grass_X(d,\cV)$ over $X$.
\end{theo}

Theorem \ref{main_thm} is unchanged when the $\gen_{d,e}(\Lambda)$ are changed to \lax-similar elements, so it holds independently of all choices in the construction of~$\gen_{d,e}(\Lambda)$, as mentioned in Remark~\ref{rem:choices}.
Furthermore, \lax\ pull-backs, push-forwards and connecting homomorphisms are compatible with \lax-similarity by Remark~\ref{simpushpull_rema}. We therefore use the following convention\,:
\begin{conv} \label{lax_conv}
All pull-backs, push-forwards and connecting homomorphisms in the rest of Section \ref{main_sec} are {\em \lax}. 
\end{conv}
This precisely means that they are classical pull-backs, push-forwards and connecting homomorphisms, preceded or followed by alignment morphisms in order to start and end in Witt groups with specific line bundles, chosen for the various statements to make sense. These alignments exist and the freedom in their choices has no influence on the \lax-similarity classes considered, by the results of \cite{Balmer11}, as summarized in Section~\ref{tot-form_sec}.

\begin{rema}
\label{ind_rem}%
The case $d=1$ or $e=1$ is that of the projective bundle
$\bbP(\cV)$. In that case, Walter~\cite{Walter03} has proved such a
result (see~\cite[Proposition~8.1]{Walter03} when $\cV$ has a
complete flag). One can check that his generator $\xi$ is
\lax-similar in the sense of Definition~\ref{laxsim_defi} to our
$\gen(\Lambda)$ when $\Lambda=\rect{d}{e}$ is the only non-empty
even diagram in
\deframe. The same result was also proved in the case
$\cV=\cO_X^{d+e}$ in~\cite{Balmer05}, using methods closer to the
present geometric philosophy. We could therefore assume the starting
point of the induction (\ie the case $d=1$ or $e=1$). However our
proof of the induction step covers this ``fringe" case as well.
\end{rema}

\begin{proof}[Proof of Theorem~\ref{main_thm}]
The proof occupies the rest of Section~\ref{main_sec}. We proceed by induction
on~$\rank(\cV)=d+e$. The initial case $d=e=1$ will follow from the
same proof, except that we shall not use any induction hypothesis in
that case. In other words, from now on, \emph{when $d+e\geq 2$, we
can assume the result for vector bundles $\cV$ of rank less than
$d+e$}. We fix a complete flag
$$
0=\cV_0\subb\cV_1\subb\cdots\subb\cV_{d+e}=\cV\,.
$$
We set as before $\cV^1:=\cV_{d+e-1}$.

The idea of the proof is to use Theorem~\ref{thm:localization} in
the localization setting of Section~\ref{cell-dec_sec}, \ie for the
regular closed immersion $\iota:Z=\Grass_X(d,\cV^1) \hook
\Grass_X(d,\cV)$ with open complement $U=U_X(d,\cV_\smallbullet)$.
For every line bundle $L$ on $\Grass_X(d,\cV)$, the localization
long exact sequence~\eqref{eq:LES} reads here
\begin{equation}
\label{LES_eq}%
\end{equation}
\vglue-3em
$$
\vcenter{\xymatrix@C=2.7ex{ \cdot\!\cdot\!\cdot \ar[r] &
\Wcoh^{i}_Z(\Grass_X(d,\cV),L) \ar[r]^-{\ext} &
\Wcoh^{i}(\Grass_X(d,\cV),L) \ar[r]^-{\upsilon^*} &
\Wcoh^i(U_X(d,\cV_\smallbullet),\upsilon^*L)
\ar`d[l]`[dll]^{\bord}[dll]
& \\
&
\Wcoh^{i+1}_Z(\Grass_X(d,\cV),L) \ar[r]^-{\ext} &
\Wcoh^{i+1}(\Grass_X(d,\cV),L) \ar[r]^-{\upsilon^*} &
\Wcoh^{i+1}(U_X(d,\cV_\smallbullet),\upsilon^*L) \ar[r] &
\cdot\!\cdot\!\cdot
\\
}}
$$
where $\ext: \Wcoh^i_Z(\Grass_X(d,\cV),L) \to
\Wcoh^i(\Grass_X(d,\cV),L)$ is extension of support.

\begin{rema} \label{skipspecialcases_rema}
The reader will notice along the proof that when $d=1$ (resp.\ $e=1$), the situation is somewhat particular. This is because the map $\upsilon^*$ (resp.\ $\iota^*$) on relative Picard groups modulo~$2$ is no longer injective. So, some adjustments need to be made. In a first reading, it might be a good idea to skip these special cases to follow more easily the structure of the proof.
\end{rema}

Before applying Theorem~\ref{thm:localization}, we first need to use
d\'evissage to obtain a total basis of the Witt groups of
$\Grass_X(d,\cV)$ with support in $\Grass_X(d,\cV^1)$, from a total
basis of the Witt groups of $\Grass_X(d,\cV^1)$. Let
$(\iota_{\!Z})_*$ be the push-forward along $\iota$ from the Witt
groups of $Z$ to the Witt groups with support in $Z$. We therefore
have $\ext \circ (\iota_{\!Z})_* = \iota_*$, where $\ext$ is the
extension of support from $Z$ to the whole~$\Grass_X(d,\cV)$.

\begin{lemm} \label{devissagegen_lemm}
Assume $e\geq 2$ (and recall Convention \ref{lax_conv}). The elements
$$(\iota_{\!Z})_*\big(\gen_{d,e-1}(\Lambda'')\big) \qquad \text{with $\Lambda''$ even $(d,e-1)$-diagram}$$
form a total basis of the Witt groups of
$\Grass_X(d,\cV)$ with support in $Z=\Grass_X(d,\cV^1)$.
\end{lemm}

\begin{proof}
Since $e\geq 2$ the map $\iota^*:\Pic_X(\Grass_X(d,\cV))/2 \to
\Pic_X(\Grass_X(d,\cV^1))/2$ is an isomorphism (see
Remark~\ref{Pic-can_rem}). We can therefore
apply~\cite[Corollary~6.16]{Balmer11} with
$P=\Pic_X(\Grass_X(d,\cV))$, using that the
$\gen_{d,e-1}(\Lambda'')$ form a total basis of the Witt groups of
$\Grass_X(d,\cV^1)$ by the induction assumption.
\end{proof}

When $e=1$, the situation is slightly different. Since
$\Grass_X(d,\cV^1)=X$, the map $\iota^*$ is no longer injective on
Picard groups modulo $2$, and there are two essentially different
\lax\ push-forwards. Indeed, in that case
$\Pic_X(\Grass_X(d,\cV))/2=\bbZ/2\cdot \Det_d$ and $\Pic_X(X)/2=1$.
Observe the convenient
$\Pic_X(\Grass_X(d,\cV))/2=\{\Det_d^d\,,\,\Det_d^{d+1}\}$.

\begin{defi}
Assume $e=1$ and therefore $\Grass_X(d,\cV^1)=X$. Let $L_d$ (resp.\ $L_{d+1}$) be a line bundle on $\Grass_X(d,\cV)$ such that
$[L_d]=\Det_d^{d} \in \Pic_X(\Grass_X(d,\cV))/2$ (resp.\ $[L_{d+1}]=\Det_d^{d+1}$). Let $\genpt_{d}$ (resp.\ $\genpt_{d+1}$)
be a \lax\ push-forward of the unit form $1_X$ to
$\Wcoh^*_X(\Grass_X(d,\cV),L_{d})$ (resp.\ to
$\Wcoh^*_X(\Grass_X(d,\cV),L_{d+1})$) along $\iota:X\hook
\Grass_X(d,\cV)$.
\end{defi}

\begin{lemm}[Analogue of Lemma~\ref{devissagegen_lemm} in special case]
Assume $e=1$. Then
\begin{enumerate}[(a)]
\item \label{genptbasisd_item} The element $\genpt_d$ (resp.\ $\genpt_{d+1}$) forms a total basis of the $\{\Det_d^d\}$-part (resp.\ the $\{\Det_{d}^{d+1}\}$-part) of the Witt groups of $\Grass_X(d,\cV)$ with support in $X=\Grass_X(d,\cV^1)$.
\item \label{genptbasis_item} Together, $\genpt_d$ and $\genpt_{d+1}$ form a total basis of the Witt groups of $\Grass_X(d,\cV)$ with support in $X=\Grass_X(d,\cV^1)$.
\end{enumerate}
\end{lemm}
\begin{proof}
Part \eqref{genptbasisd_item} follows
from~\cite[Corollary~6.15]{Balmer11}, applied to the subset
$P=\{\Det_d^d\}$ (resp.\ $P=\{\Det_d^{d+1}\}$). Part
\eqref{genptbasis_item} follows from \eqref{genptbasisd_item} since
the union of total bases of disjoint parts $P_1$ and $P_2$ forms a
total basis of the part $P=P_1 \cup P_2$. See
\cite[Lemma~6.10]{Balmer11}. Here we use
$\Pic_X(\Grass_X(d,\cV))/2=\{\Det_d^d\}\sqcup \{\Det_d^{d+1}\}$.
\end{proof}

We need to track our classes $\gen_{d,e}(\Lambda)$ along the morphisms involved in the long exact sequence of localization. Recall the constructions $\barpush$, $\barres$ and $\barbord$ on even Young diagrams (Def.\,\ref{bar_defi}) and the \lax-similarity equivalence relation $\weq$ (Def.\,\ref{laxsim_defi}).

\begin{prop}
\label{push-pull-bord_prop}%
With the above notation (and with Convention \ref{lax_conv}), we have
\begin{enumerate}[(a)]
\item \label{iotapush_item}
Let $\Lambda'$ be an even Young diagram in
$(d\!\times\!(e-1))$-frame, with $d\geq 1$, $e\geq 2$. The
push-forward $\iota_*$ satisfies (see Figure~\ref{push_fig}):
$$
\iota_*(\gen_{d,e-1}(\Lambda')) \weq
\iso_{d,e}(\barpush(\Lambda')) \qquad \text{if $\zero(\Lambda')$ is
even.}
$$
\item
Let $\Lambda$ be an even Young diagram in \deframe, with $d\geq 2$
and $e\geq 1$. The restriction morphism $\upsilon^*$ satisfies (see
Figure~\ref{pull_fig}):
$$
\upsilon^*\big(\gen_{d,e}(\Lambda)\big) \weq
\alpha^*\big(\iso_{d-1,e}(\barres(\Lambda))\big) \qquad \text{if
$\Lambda_d=0$.}
$$
\item
Let $\Lambda''$ be an even Young diagram in $((d-1)\!\times\!e)$-frame,
with $d,e \geq 2$. The connecting homomorphism $\bord$ satisfies
(see Figure~\ref{bord_fig}):
$$
\bord\big(\alpha^*(\gen_{d-1,e}(\Lambda''))\big) \weq
(\iota_{\!Z})_*\big(\iso_{d,e-1}(\barbord(\Lambda''))\big) \qquad
\text{if $\Lambda''_{d-1}$ is odd.}
$$
\end{enumerate}
\end{prop}

\begin{proof}
\noindent\textit{(a).} Let $\Lambda'$ be an even $(d,e-1)$-diagram
such that $\zero(\Lambda')$ is even. Let $\multi{d}$ and $\multi{e}$
be the corresponding $k$-tuples (Def.\,\ref{de_def}). By assumption
we can consider the even $(d,e)$-diagram
$\barpush(\Lambda'):=(\Lambda'_1+1,\ldots,\Lambda'_d+1)$. Observe
that it has the same associated $k$-tuples (with respect to the
\deframe). From~\eqref{Flag-Lambda_eq}, it is then easy to see that
$\Flag_{X}(d,e-1;\Lambda')=\Flag_{X}(d,e;\barpush(\Lambda'))$ and
that the diagram
$$\xymatrix{
\Grass_X(d,\cV^1) \ar[r]^-{\iota} & \Grass_X(d,\cV)
\\
\Flag_{X}(d,e-1;\Lambda') \ar[u]^{\ff_{d,e-1;\Lambda'}} \ar@{=}[r] &
\Flag_{X}(d,e;\barpush(\Lambda'))
\ar[u]_{\ff_{d,e,\barpush(\Lambda')}} }$$
commutes. Thus, (a) follows by composition of push-forwards and by
Theorem~\ref{laxpush_thm}. 

\medbreak

\noindent\textit{(b).} Let $\Lambda$ be an even $(d,e)$-diagram such
that $\Lambda_d=0$. Let $\multi{d}$ and $\multi{e}$ be the
corresponding $k$-tuples (Def.\,\ref{de_def}). By assumption we can
consider the even $(d-1,e)$-diagram
$\barres(\Lambda):=\Lambda\restr{d-1,e}$. We then have two cartesian
squares\,:
$$
\xymatrix{ \Grass_X(d,\cV) \ar@{}[dr]|{\Box}
& \ U_X(d,\cV_\smallbullet)
\ar@{_{(}->}[l]_-{\upsilon} \ar[r]^-{\alpha} \ar@{}[dr]|{\Box} &
\Grass_X(d-1,\cV^1)
\\
\Flag_X(d,e;\Lambda) \ar[u]^{\ff_{d,e;\Lambda}}
& \ U' \ar@{_{(}->}[l] \ar[r] \ar[u]
& \Flag_{X}(d-1,e;\barres(\Lambda))
\ar[u]_{\ff_{d-1,e;\barres(\Lambda)}} }
$$
by Lemma~\ref{double-PB_lem}. The equality in (b) then follows by
base change on the above two cartesian squares
(see~\cite[Theorem~5.5]{Calmes11}, the horizontal morphisms are
smooth, so flat, and the vertical maps are proper, including $U'\to
U_X(d,\cV_\smallbullet)$). Using Remark~\ref{laxpull_rema} and
Theorem~\ref{laxpush_thm}, both sides of the desired equation
are \lax-similar to the push-forward of the same
$1_{U'}\in\Wcoh^0(U')$ along $U'\to U_X(d,\cV_\smallbullet)$.

\medbreak

\noindent\textit{(c).} Let $\Lambda''$ be an even $(d-1,e)$-diagram
such that $\Lambda''_{d-1}$ is odd. Let $\multi{d}$ and $\multi{e}$
be the corresponding $k$-tuples (Def.\,\ref{de_def}). By assumption,
we can consider the even $(d,e-1)$-diagram
$$
\barbord(\Lambda''):=(\Lambda''_1-1,\ldots,\Lambda''_{d-1}-1,0)\,.
$$

The key tool here is~\cite[Theorem~2.6]{Balmer09} which allows us to
compute $\partial$ geometrically, as the pull-back to the
exceptional fiber along
$\tilde\alpha\circ\tilde\iota:\Exc_X(d,\cV_\smallbullet)\to
\Grass_X(d-1,\cV^1)$ followed by the push-forward along
$\tilde\pi:\Exc_X(d,\cV_\smallbullet)\to \Grass_X(d,\cV^1)$, as long
as the twist of the Witt class under consideration satisfies the
assumptions of~\cite[Theorem~2.6]{Balmer09}.

Here, the twist $\Det_{d-1}^{\twist(\Lambda'')}$ of $\gen_{d-1,e}(\Lambda'')$ is
given by $\twist(\Lambda'')=[d+1]\in\bbZ/2$. Using
Diagram~\eqref{Picards_eq} and Equation~\eqref{rel-can-Bl_eq}, one
can easily check that this is the twist for which
of~\cite[Theorem~2.6]{Balmer09} applies. By Lemma~\ref{PB_lem}, we
have the right-hand cartesian square in the following commutative
diagram\,:
$$
\xymatrix{
\Grass_X(d,\cV^1)
& \Exc_X(d,\cV_\smallbullet) \ar[l]_-{\tilde\pi} \ar[r]^-{\tilde\alpha\tilde\iota} \ar@{}[dr]|{\Box} &
\Grass_X(d-1,\cV^1)
\\
\Flag_X(d,e-1;\barbord(\Lambda''))
\ar[u]^{\ff_{d,e-1;\barbord(\Lambda'')}}
& F' \ar[u]\ar[l]_-{\pi'}\ar[r]
& \Flag_X(d-1,e;\Lambda'')\,. \ar[u]_{\ff_{d-1,e;\Lambda''}} }$$
We can now compute $\bord(\alpha^*(\gen_{d-1,e}(\Lambda'')))$ by
starting with the unit form 1 in the lower right corner,
pushing-forward along $\ff_{d-1,e;\Lambda''}$, pulling back along
$\tilde\alpha\tilde\iota$ and then pushing forward along
$\tilde\pi$. By base-change on the right-hand square, this class
$\bord(\alpha^*(\gen_{d-1,e}(\Lambda'')))$ is also the (lax)
push-forward of the unit form on~$F'$ along $\pi'$ and then along
$\ff_{d,e-1;\barbord(\Lambda'')}$. The key point is to check that
$\pi'_*$ preserves the unit form, up to \lax-similitude. By
Lemma~\ref{PB_lem}, we know that $\pi'$ is either an isomorphism or
a blow-up along a closed regular immersion of odd codimension. In
both cases $\pi'_*(1)\weq 1$ by Proposition~\ref{push-1_prop} and we
get the result.
\end{proof}

\begin{prop}[Analogue of Proposition~\ref{push-pull-bord_prop} in special cases] \label{push-pull-bord-special_prop}
When $e=1$, the push-forward $\iota_*:\Wcoh^0(X,\cO_X) \to
\Wcoh^d(\Grass_X(d,\cV),L_{d+1})$ satisfies
\begin{equation}
\iota_*(1_X) =\ext(\genpt_{d+1}) \weq \gen_{d,1}(\rect{d}{1}).
\end{equation}
When $d=1$, the restriction morphism $\upsilon^*$ satisfies
\begin{equation} \label{upsilonspecial_eq}
\upsilon^*(\gen_{1,e}(\vide)) \weq \alpha^*(1_X).
\end{equation}
When $e=1$ and $d\geq 2$, the connecting homomorphism of the
localization sequence with twist $L_d$ satisfies
\begin{equation}
\bord\big(\alpha^*(\gen_{d-1,1}(\rect{(d-1)}{e}))\big) \weq
\genpt_{d}\,.
\end{equation}
When $e=d=1$, the connecting homomorphism of the localization
sequence with twist $L_d=L_1$ satisfies
\begin{equation}
\bord\big(\alpha^*(1_X)\big) \weq \genpt_{1}\,.
\end{equation}
\end{prop}
\begin{proof}
The proof is entirely analogous to the one of
Proposition~\ref{push-pull-bord_prop}, so we leave the details to
the reader.
\end{proof}

We now apply localization, using Theorem~\ref{thm:localization}, in
which a part $P$ of the relative Picard group modulo $2$ needs to be
specified. Again, there is a general case, when $d\geq 2$, and a
particular case where $d=1$.

\medskip
\noindent\textit{Case $d,e\geq 2$}. We choose $P$ to be the whole
$\Pic_X(\Grass_X(d,\cV))/2$. Note that we indeed have $\upsilon^*:
\Pic_X(\Grass_X(d,\cV)/2 \to \Pic_X(U_X(d,\cV_\smallbullet))/2$
injective by \eqref{Picards_eq}. By induction and by Lemma~\ref{devissagegen_lemm},
the $\gen_{d,e-1}(\Lambda')$, $\Lambda'$ even, form a total basis of
the Witt groups of $\Grass_X(d,\cV)$ with support in
$\Grass_X(d,\cV^1)$. By induction, we also know that the
$\gen_{d-1,e}(\Lambda'')$, $\Lambda''$ even, form a total basis of
the Witt groups of $\Grass_X(d-1,\cV^1)$. Homotopy invariance
and~\cite[Corollary~6.13]{Balmer11} ensure that the
$\alpha^*(\gen_{d-1,e}(\Lambda''))$ form a total basis of the Witt
groups of $U_X(d,\cV_\smallbullet)$. The correspondence of
Proposition~\ref{push-pull-bord_prop} and
Proposition~\ref{exact-Lambda_prop} then show that assumptions
\eqref{eweq_item}, \eqref{upsilonweq_item} and \eqref{bordweq_item}
of Theorem~\ref{thm:localization} are satisfied when choosing the
following collections of Witt classes\,:
$$\SetZ=
\xymatrix@R=.3em{ \left\{\Ourfrac{\text{even Young
}(d,e-1)\text{-diagrams}}{\Lambda'\text{ such that
}\zero(\Lambda')\text{ is even}}\right\} \ar@{->}[r]^-{\simeq}_-{\barpush} &
\left\{\Ourfrac{\text{even Young
}(d,e)\text{-diagrams}}{\Lambda\text{ such that
}\Lambda_d>0}\right\}
}
$$
$$
v_{\Lambda'}=(\iota_Z)_*(\gen_{d,e-1}(\Lambda'))\qquad
\text{and}\qquad w'_{\Lambda'}=\gen_{d,e}(\barpush(\Lambda'))\weq
\ext(v_{\Lambda'})\qquad\text{for all }\Lambda'\in\SetZ
$$
$$
\SetX=
\xymatrix@R=.3em{ \left\{\Ourfrac{\text{even Young
}(d,e)\text{-diagrams}}{\Lambda\text{ such that
}\Lambda_d=0}\right\} \ar@{->}[r]^-{\simeq}_-{\barres} &
\left\{\Ourfrac{\text{even Young
}(d-1,e)\text{-diagrams}}{\Lambda''\text{ such that
}\Lambda''_{d-1}\text{ is even}}\right\}
}
$$
$$
w_{\Lambda}=\gen_{d,e}(\Lambda)\qquad \text{and}\qquad
u'_{\Lambda}=\alpha^*(\gen_{d-1,e}(\barres(\Lambda)))\weq
\upsilon^*(w_{\Lambda})\qquad\text{for all }\Lambda\in\SetX
$$
$$
\SetU=
\xymatrix@C=4ex@R=.3em{ \left\{\Ourfrac{\text{even Young
}(d-1,e)\text{-diagrams}}{\Lambda''\text{ such that
}\Lambda''_{d-1}\text{ is odd}}\right\} \ar@{->}[r]^-{\simeq}_{\barbord} &
\left\{\Ourfrac{\text{even Young
}(d,e-1)\text{-diagrams}}{\Lambda'\text{ such that
}\zero(\Lambda')\text{ is odd}}\right\}
}
$$
$$
u_{\Lambda''}=\alpha^*(\gen_{d-1,e}(\Lambda''))\ \text{and}\ 
v'_{\Lambda''}=(\iota_Z)_*(\gen_{d,e-1}(\barbord(\Lambda'')))\weq
\bord(u_{\Lambda''})\quad \text{for all $\Lambda''\in\SetU$}
$$
We therefore conclude by part \eqref{basisinduction_item} of
Theorem~\ref{thm:localization} that we have a total basis
$\{\gen_{d,e}(\Lambda)\}_{\Lambda}$ indexed by $\Lambda$ in
$\{\Lambda\text{ even with
}\Lambda_d=0\}\sqcup \{\Lambda\text{ even with
}\Lambda_d>0\}$, \ie by all even Young diagrams in \deframe. This is the result.

\medskip
\noindent\textit{Case $d\geq 2$, $e=1$.} The proof is the same as in
the case $d,e\geq 2$, except that
Proposition~\ref{push-pull-bord-special_prop} is used instead of
Proposition~\ref{push-pull-bord_prop}. We take (see Figure
\ref{specialCases1_fig})
$$\begin{array}{lll}
\SetZ=\left\{\star\right\}, & v_{\star}=\genpt_{d+1}, & w'_{\star}=\gen_{d,1}(\rect{d}{1})\weq \ext(v_{\star}),\vspace{1ex} \\
\SetX=\left\{\vide\right\}, & w_\vide=\gen_{d,1}(\vide), & u'_\vide=\alpha^*(\gen_{d-1,1}(\vide))\weq \upsilon^*(w_\vide),\vspace{1ex} \\
\SetU=\left\{\diamond\right\}, &
u_{\diamond}=\alpha^*(\gen_{d-1,1}(\vide)), &
v'_{\diamond}=\genpt_{d}\weq \bord(u_{\vide}).
\end{array}$$
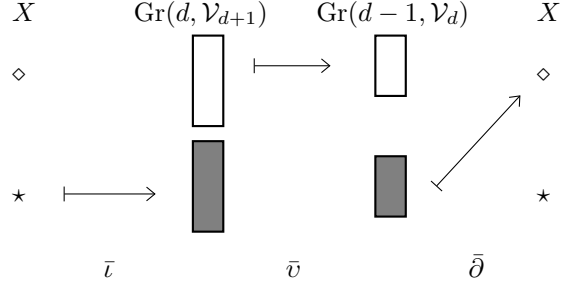
\begin{figure}[!ht]
\begin{center}
\begin{tikzpicture}[y=-1cm]
\bigarrowtipdef

\path[draw=black,fill=black!0,arrows=|-biggertip] (3.1,4.7) -- (4.3,4.7);
\path[draw=black,fill=black!0,arrows=|-biggertip] (5.6,3) -- (6.6,3);
\path[draw=black,fill=black!0,arrows=|-biggertip] (8.0,4.6) -- (9.1,3.4);
\path (2.3,3.2) node[text=black,anchor=base west] {$\diamond$};
\path (2.3,4.8) node[text=black,anchor=base west] {$\star$};
\path (9.2,3.2) node[text=black,anchor=base west] {$\diamond$};
\path (9.2,4.8) node[text=black,anchor=base west] {$\star$};
\path (3.9,2.4) node[text=black,anchor=base west] {$\Grass(d,\cV_{d+1})$};
\path (9.2,2.4) node[text=black,anchor=base west] {$X$};
\path (2.3,2.4) node[text=black,anchor=base west] {$X$};
\path (6.3,2.4) node[text=black,anchor=base west] {$\Grass(d-1,\cV_d)$};

\draw[thick,black] (7.6,2.6) rectangle (7.2,3.4);
\draw[thick,black] (5.2,2.6) rectangle (4.8,3.8);
\path[draw=black,thick,fill=black!50] (7.6,4.2) rectangle (7.2,5);
\path[draw=black,thick,fill=black!50] (5.2,4) rectangle (4.8,5.2);
\path (5.9,5.8) node[text=black,anchor=base west] {$\barres$};
\path (8.3,5.8) node[text=black,anchor=base west] {$\barbord$};
\path (3.5,5.8) node[text=black,anchor=base west] {$\barpush$};
\end{tikzpicture}
\caption{Image of the diagrams by the morphisms $\barpush$,
$\barres$ and $\barbord$ when $d\geq 2$ and $e=1$. No arrow means
the corresponding generator is mapped to zero.}
\label{specialCases1_fig}%
\end{center}
\end{figure}

\noindent\textit{Case $d=1$, $e\geq 2$.} In that case,
$\Grass_X(d-1,\cV^1)=\Grass_X(0,\cV^1)=X$, so we can no longer
choose $P$ to be the whole relative Picard group modulo $2$ because
injectivity of $\upsilon^*$ would not hold (see
Remark~\ref{Pic-can_rem}). We therefore first choose
$P=P_1:=\{\Det_d^{d+1}\}=\{1\}$ and take (see $P_1$-part of Figure~\ref{specialCases2_fig})
$$\begin{array}{ll}
\SetZ=\left\{\rect{1}{(e-1)}\right\}, & v_{\rect{1}{(e-1)}}=(\iota_{\!Z})_*(\gen_{1,e-1}(\rect{1}{(e-1)})), \vspace{1ex}\\
 & w'_{\rect{1}{(e-1)}}=\gen_{1,e}(\rect{1}{e})\weq \ext(v_{\rect{1}{(e-1)}}), \vspace{1ex}\\
\SetX=\left\{\vide\right\}, & w_\vide=\gen_{1,e}(\vide), \vspace{1ex}\\
& u'_\vide=\alpha^*(1_X)\weq \upsilon^*(w_\vide), \vspace{1ex}\\
\SetU=\emptyset. \\
\end{array}$$
As in previous cases, by part \eqref{iotapush_item} of
Proposition~\ref{push-pull-bord_prop}, by equation
\eqref{upsilonspecial_eq} and by induction, using
Theorem~\ref{thm:localization}, we get that
$\gen_{1,e}(\rect{1}{e})$ and $\gen_{1,e}(\vide)$ form a basis of
the $\{1\}$-part of the Witt groups of $\Grass_X(1,\cV)$.

We then choose $P=P_2:=\{\Det_d\}$, and take (see $P_2$-part of Figure~\ref{specialCases2_fig})
$$\begin{array}{lll}
\SetZ=\SetX=\emptyset, & & \vspace{1ex}\\
\SetU=\left\{\star\right\} & u_{\star}=\alpha^*(1_X), &
v'_{\star}=(\iota_{\!Z})_*(\gen_{1,e}(\vide))\weq \bord(u_{\star}).
\end{array}$$
By Theorem~\ref{thm:localization} using \eqref{upsilonspecial_eq},
we obtain that the empty set is a basis of the $\{\Det_d\}$-part of
the Witt groups of $\Grass_X(1,\cV)$.

Putting together both bases (one of which is empty), we obtain a
basis of the Witt groups of $\Grass_X(d,\cV)$
by~\cite[Lemma~6.10]{Balmer11}, since $P_1 \sqcup
P_2=\Pic_X(\Grass_X(1,\cV))$. This basis contains exactly the only
two generators corresponding to even diagrams of size $(1,e)$, \ie
$\vide$ and $\rect{1}{e}$.  See Figure~\ref{specialCases2_fig}.
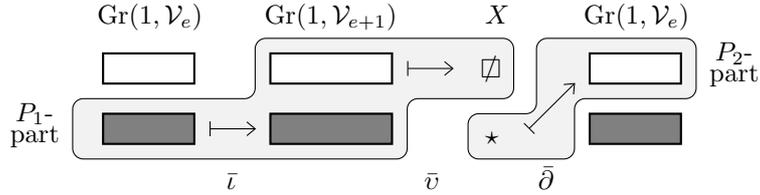
\begin{figure}[!ht]
\begin{center}
\begin{tikzpicture}[y=-1cm]
\pgfarrowsdeclare{biggertip}{biggertip}{
  \setlength{\arrowsize}{.4pt}
  \addtolength{\arrowsize}{.5\pgflinewidth}
  \pgfarrowsrightextend{0}
  \pgfarrowsleftextend{-5\arrowsize}
}{
  \setlength{\arrowsize}{0.4pt}
  \addtolength{\arrowsize}{.3\pgflinewidth}
  \pgfpathmoveto{\pgfpoint{-6\arrowsize}{5\arrowsize}}
  \pgfpathlineto{\pgfpointorigin}
  \pgfpathlineto{\pgfpoint{-6\arrowsize}{-5\arrowsize}}
  \pgfusepathqstroke
}
\path[draw=black,fill=black!0,arrows=|-biggertip] (3.4,1) -- (4,1);
\path[draw=black,fill=black!0,arrows=|-biggertip] (6,0.2) -- (6.6,0.2);
\path[draw=black,fill=black!0,arrows=|-biggertip] (7.6,1.0) -- (8.2,0.4);
\path (8.2,-0.4) node[text=black,anchor=base west] {$\Grass(1,\cV_e)$};
\path (1.8,-0.4) node[text=black,anchor=base west] {$\Grass(1,\cV_e)$};
\path (4,-0.4) node[text=black,anchor=base west] {$\Grass(1,\cV_{e+1})$};
\path (6.9,-0.4) node[text=black,anchor=base west] {$X$};
\path (6.9,1.2) node[text=black,anchor=base west] {$\star$};
\draw (7.1,0.2) node[text=black, anchor=center] {$\!\vide\!$};

\draw (1.1,0.8) node (p1) [text=black] {$P_1$-};
\draw (p1)+(0,-2ex) node (p1) [text=black] {part};
\draw (10.3,0) node (p2) [text=black] {$P_2$-};
\draw (p2)+(0,-2ex) node (p2) [text=black] {part};
\path[draw=black,thick,fill=black!50] (2,0.8) rectangle (3.2,1.2);
\path[draw=black,thick,fill=black!50] (8.4,0.8) rectangle (9.6,1.2);
\path[draw=black,thick,fill=black!50] (4.2,0.8) rectangle (5.8,1.2);
\filldraw[fill=white,thick,draw=black] (2,0) rectangle (3.2,0.4);
\filldraw[fill=white,thick,draw=black] (8.4,0) rectangle (9.6,0.4);
\filldraw[fill=white,thick,draw=black] (4.2,0) rectangle (5.8,0.4);
\path (6.1,1.8) node[text=black,anchor=base west] {$\barres$};
\path (7.6,1.8) node[text=black,anchor=base west] {$\barbord$};
\path (3.5,1.8) node[text=black,anchor=base west] {$\barpush$};
\begin{pgfonlayer}{background}
\draw[fill=black!10,fill opacity=0.5,rounded corners] (1.6,0.6) -- (4,0.6) -- (4,-0.2) -- (7.4,-0.2) -- (7.4,0.6) -- (6,0.6) -- (6,1.4) -- (1.6,1.4) -- cycle;
\draw[fill=black!10,fill opacity=0.5,rounded corners] (6.8,0.8) -- (7.7,0.8) -- (7.7,-0.2) -- (9.8,-0.2) -- (9.8,0.6) -- (8.2,0.6) -- (8.2,1.4) -- (6.8,1.4) -- cycle;
\end{pgfonlayer}
\end{tikzpicture}
\caption{Image of the diagrams by the morphisms $\barpush$,
$\barres$ and $\barbord$ when $d=1$ and $e\geq 2$. No arrow means
the corresponding generator is mapped to zero.}
\label{specialCases2_fig}%
\end{center}
\end{figure}

\noindent\textit{Case $d=e=1$.} It works as in the case $d=1$,
$e\geq 2$ except that $(\iota_{\!Z})_*(\gen_{1,e-1}(\vide))$ is
replaced by $\genpt_1$ and that
$(\iota_{\!Z})_*(\gen_{1,e}(\rect{1}{1}))$ is replaced by
$\genpt_0$. See Figure~\ref{specialCases3_fig}. Note that, as announced, this case does not use any
induction assumption. This
completes the induction and therefore the proof of Main
Theorem~\ref{main_thm}.\qedhere
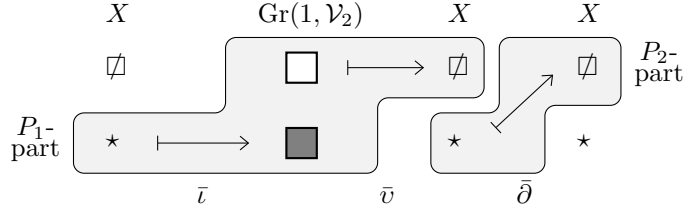
\begin{figure}[!ht]
\begin{center}
\begin{tikzpicture}[y=-1cm]
\pgfarrowsdeclare{biggertip}{biggertip}{
  \setlength{\arrowsize}{.4pt}
  \addtolength{\arrowsize}{.5\pgflinewidth}
  \pgfarrowsrightextend{0}
  \pgfarrowsleftextend{-5\arrowsize}
}{
  \setlength{\arrowsize}{0.4pt}
  \addtolength{\arrowsize}{.3\pgflinewidth}
  \pgfpathmoveto{\pgfpoint{-6\arrowsize}{5\arrowsize}}
  \pgfpathlineto{\pgfpointorigin}
  \pgfpathlineto{\pgfpoint{-6\arrowsize}{-5\arrowsize}}
  \pgfusepathqstroke
}

\path[draw=black,fill=black!0,arrows=|-biggertip] (5.6,6.8) -- (6.6,6.8);
\path[draw=black,fill=black!0,arrows=|-biggertip] (3.1,7.8) -- (4.3,7.8);
\path[draw=black,fill=black!0,arrows=|-biggertip] (7.55,7.6) -- (8.3,6.9);

\path (2.3,6.9) node[text=black,anchor=base west] {$\vide$};
\path (2.3,7.85) node[text=black,anchor=base west] {$\star$};
\path (8.5,6.9) node[text=black,anchor=base west] {$\vide$};
\path (8.5,7.85) node[text=black,anchor=base west] {$\star$};
\path (6.8,6.9) node[text=black,anchor=base west] {$\vide$};
\path (6.8,7.85) node[text=black,anchor=base west] {$\star$};

\path (4.3,6.2) node[text=black,anchor=base west] {$\Grass(1,\cV_2)$};
\path (2.3,6.2) node[text=black,anchor=base west] {$X$};
\path (8.5,6.2) node[text=black,anchor=base west] {$X$};
\path (6.8,6.2) node[text=black,anchor=base west] {$X$};

\draw (1.5,7.6) node (p1) [text=black] {$P_1$-};
\draw (p1)+(0,-2ex) node (p1) [text=black] {part};
\draw (9.7,6.6) node (p2) [text=black] {$P_2$-};
\draw (p2)+(0,-2ex) node (p2) [text=black] {part};

\path (5.9,8.6) node[text=black,anchor=base west] {$\barres$};
\path (7.7,8.6) node[text=black,anchor=base west] {$\barbord$};
\path (3.5,8.6) node[text=black,anchor=base west] {$\barpush$};

\filldraw[fill=white,thick,draw=black] (5.2,6.6) rectangle (4.8,7);
\path[draw=black,thick,fill=black!50] (5.2,7.6) rectangle (4.8,8);

\begin{pgfonlayer}{background}
\draw[fill=black!10,fill opacity=0.5,rounded corners] (2,7.4) -- (4,7.4) -- (4,6.4) -- (7.4,6.4) -- (7.4,7.2) -- (6,7.2) -- (6,8.2) -- (2,8.2) -- cycle;
\draw[fill=black!10,fill opacity=0.5,rounded corners] (6.7,7.4) -- (7.6,7.4) -- (7.6,6.4) -- (9.2,6.4) -- (9.2,7.3) -- (8.2,7.3) -- (8.2,8.2) -- (6.7,8.2) -- cycle;
\end{pgfonlayer}
\end{tikzpicture}
\caption{Image of the diagrams by the morphisms $\barpush$,
$\barres$ and $\barbord$ when $d=e=1$. No arrow means the
corresponding generator is mapped to zero.}
\label{specialCases3_fig}%
\end{center}
\end{figure}
\end{proof}


\section{Corollaries and examples}
\label{final_sec}%


Here is how to deduce a more classical result on Witt groups that
does not involve ``total'' concepts. Recall that $\pi$ is the
structural morphism of $\Grass_X(d,\cV)$ and that $\Taut_d$ is its
tautological bundle. Also recall that the generator
$\gen_{d,e}(\Lambda)$ lives in
$\Wcoh^{|\Lambda|}\big(\Grass_X(d,\cV),L_\Lambda\big)$ where
$|\Lambda|$ is the area of~$\Lambda$. Any line bundle on
$\Grass_X(d,\cV)$ is, up to isomorphism, of the form $\pi^*K \otimes
\det (\Taut_d)^{\otimes l}$, for some line bundle $K$ on $X$ and
some $l\in\bbZ$. A diagram $\Lambda$ is such that
$[L_\Lambda]=[\pi^*K \otimes \det(\Taut_d)^{\otimes l}] \in
\Pic_X(\Grass_X(d,\cV))/2$ if and only if $\twist(\Lambda)=l \in
\bbZ/2$, where $\twist(\Lambda)$ is the half-perimeter
of~$\Lambda$.

\begin{coro}
Let $K$ be a line bundle on $X$ and let $l$ be an integer. For each
even $(d,e)$-diagram $\Lambda$ such that $l=\twist(\Lambda) \in
\bbZ/2$, choose a line bundle $N_\Lambda$ over $\Grass_X(d,\cV)$ and
an isomorphism $N_\Lambda^{\otimes 2}\otimes \pi^*(K \otimes
\det(\cV)^{\otimes -\row(\Lambda)}) \otimes L_\Lambda \isoto
\pi^*(K) \otimes \det(\Taut_d)^l$. Then the morphism of
$\Wcoh^0(X,\cO_X)$-modules
$$
\bigoplus_{\ourfrac{\Lambda\text{ even s.\,t.}}{\twist(\Lambda)=l
\in \bbZ/2}} \hspace{-2ex} \Wcoh^{k-|\Lambda|}\big(X,K\otimes
\det(\cV)^{\otimes -\row(\Lambda)}\big) \isoto
\Wcoh^k\big(\Grass_X(d,\cV),\pi^*(K) \otimes \det(\Taut_d)^{\otimes
l}\big)
$$
sending $(x_\Lambda)$ to $\sum x_\Lambda \cdot \gen_{d,e}(\Lambda)$
(notation of Definition~\ref{lincomb}) is an isomorphism.
\end{coro}

\begin{proof}
Apply Theorem~\ref{thm:classic-basis} to the total basis formed by
the $\gen_{d,e}(\Lambda)$.
\end{proof}

\begin{coro}
\label{ranks_cor}%
Let $d'$ (resp.\ $e'$) be the integral part of $d/2$ (resp.\ $e/2$)
and consider the binomial coefficient
$\binom{a+b}{a}=\frac{(a+b)!}{a!b!}$. The cardinal of a total basis
of the Witt groups of $\Grass_X(d,\cV)$ over $X$ is $2
\binom{d'+e'}{e'}$. If we assume moreover that $\Wcoh^i(X,K)=0$ for
$i\not\equiv 0 \mod 4$ or $[K]\neq [\cO_X] \in \Pic(X)/2$, for
instance for $X$ local (\eg\ a field), then, as a module over
$\Wcoh^0(X,\cO_X)$,
\begin{itemize}
\item the classical Witt group of symmetric forms $\Wcoh^0(\Grass_X(d,\cV),\cO)$ has rank $\binom{d'+e'}{e'}$,
\item the classical Witt group of antisymmetric forms $\Wcoh^2(\Grass_X(d,\cV),\cO)$ is zero.
\item the Witt groups $\Wcoh^1(\Grass_X(d,\cV),L)$ and $\Wcoh^3(\Grass_X(d,\cV),L)$ are zero when $[L]=\Det_d \in \Pic_X(\Grass_X(d,\cV))/2$.
\end{itemize}
\end{coro}

\begin{proof}
Let ``$4$-block'' mean ``$2\times 2$ square''. Every even Young
diagram $\Lambda$ is either
\begin{enumerate}[{\hspace{2ex}}(a)]
\item \label{blocks}%
a union of $4$-blocks and $\gen(\Lambda)$ counts in
$\Wcoh^0(\Grass_X(d,\cV),\cO)$,
\item \label{row}%
a single row plus $4$-blocks ($e$ even) and $\gen(\Lambda)$ counts
in $\Wcoh^e(\Grass_X(d,\cV),L_\Lambda)$,
\item \label{col}%
a single column plus $4$-blocks ($d$ even) and $\gen(\Lambda)$
counts in $\Wcoh^d(\Grass_X(d,\cV),L_\Lambda)$,
\item \label{col-row}%
a single row and a single column plus $4$-blocks ($d$ and $e$ odd)
and $\gen(\Lambda)$ counts in $\Wcoh^{d+e-1}(\Grass_X(d,\cV),\cO)$.
\end{enumerate}
All possibilities (a)-(d) are exclusive and can be enumerated easily
by counting the diagrams of $4$-blocks, which amounts to counting
the usual Young diagrams in $(d'\!\times\!e')$-frame. We get
$\binom{d'+e'}{e'}$ elements in case~(a), and the other results
depend on the parities of $d$ and $e$ but are also very easy to
figure out in each case.
\end{proof}

\begin{coro} \label{bord-non-zero_coro}%
The connecting homomorphism $\bord$ is zero for all line bundles $L$
(and thus the long exact sequence~\eqref{LES_eq} decomposes as split
short exact sequences as for Chow groups) if and only if both $d$
and $e$ are even.
\end{coro}

\begin{proof}
Looking back at the proof of the main theorem, and
at~\eqref{bar-bord_eq} (or Figure~\ref{bord_fig}), we see that
$\bord$ is zero if and only if there is no even $(d-1,e)$-diagram
$\Lambda''$ such that $\Lambda''_{d-1}$ is odd. This implies that
$e$ is even (otherwise $\Lambda''=\rect{(d-1)}{e}$ would be such an even
diagram) and that $d-1$ is odd (otherwise $\Lambda''=(1,\ldots,1)$
would be such an even diagram). Conversely, assume $e$ even and the
existence of an even $(d-1,e)$-diagram $\Lambda''$ such that
$\Lambda''_{d-1}$ is odd. Then $e_k$ is odd (since
$e=\Lambda''_{d-1}+e_k$ is even), hence all $e_i$ are odd since
$\Lambda''$ is an even diagram. In particular, $e_1$ is odd, hence
$e_1>0$ and therefore $d_1$ is even. This implies that
$d-1=d_k=(d_k-d_{k-1})+\ldots+(d_2-d_1)+d_1$ is even, \ie $d$ is
odd, as was to be shown.
\end{proof}

\begin{nota}
\label{Gal_not}%
For $d,e\geqslant1$, we write $\Gal_X(d,e)$ for the split
Grassmannian
$$
\Gal_X(d,e)=\Grass_X(d,\cO_X^{d+e})\,.
$$
\end{nota}

\begin{exam}
Figure~\ref{G33_fig} shows how the different generators are mapped
to each other, up to lax-similitude, by $\iota_*$, $\upsilon^*$ and
$\bord$ in the long exact sequence~\eqref{LES_eq} for~$\Gal(3,3)$.
We use d\'evissage (resp.\ homotopy invariance) to identify the
generators of the Witt groups of $\Gal(3,2)$ (resp.\ of $\Gal(2,3)$)
with generators of the Witt groups of $\Gal(3,3)$ with support on
$\Gal(3,2)$ (resp.\ of the open complement of $\Gal(3,2)$).
\end{exam}

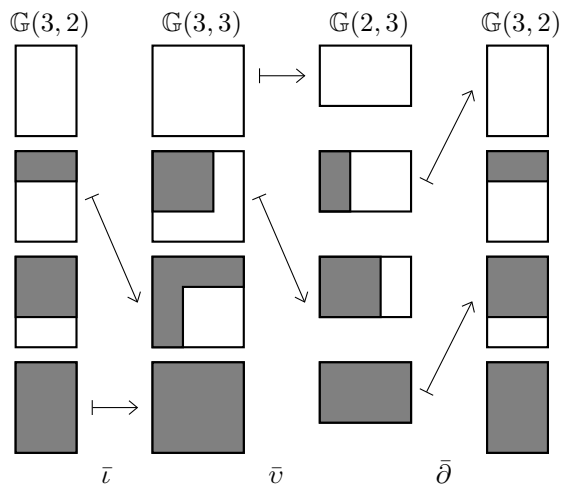
\begin{figure}[!ht]
\begin{center}
\begin{tikzpicture}[y=-1cm]
\pgfarrowsdeclare{biggertip}{biggertip}{
  \setlength{\arrowsize}{.4pt}
  \addtolength{\arrowsize}{.5\pgflinewidth}
  \pgfarrowsrightextend{0}
  \pgfarrowsleftextend{-5\arrowsize}
}{
  \setlength{\arrowsize}{0.4pt}
  \addtolength{\arrowsize}{.3\pgflinewidth}
  \pgfpathmoveto{\pgfpoint{-6\arrowsize}{5\arrowsize}}
  \pgfpathlineto{\pgfpointorigin}
  \pgfpathlineto{\pgfpoint{-6\arrowsize}{-5\arrowsize}}
  \pgfusepathqstroke
}

\path[draw=black,fill=black!0,arrows=|-biggertip] (3,2.4) -- (3.6,3.8);
\path[draw=black,fill=black!0,arrows=|-biggertip] (3,5.2) -- (3.6,5.2);
\path[draw=black,fill=black!0,arrows=|-biggertip] (5.2,2.4) -- (5.8,3.8);
\path[draw=black,fill=black!0,arrows=|-biggertip] (5.2,0.8) -- (5.8,0.8);
\path[draw=black,fill=black!0,arrows=|-biggertip] (7.4,5) -- (8,3.8);
\path[draw=black,fill=black!0,arrows=|-biggertip] (7.4,2.2) -- (8,1);
\path (5.2,6.2) node[text=black,anchor=base west] {$\barres$};
\path (3,6.2) node[text=black,anchor=base west] {$\barpush$};
\path (7.4,6.2) node[text=black,anchor=base west] {$\barbord$};
\path (1.8,0.2) node[text=black,anchor=base west] {$\Gal(3,2)$};
\path (3.8,0.2) node[text=black,anchor=base west] {$\Gal(3,3)$};
\path (6,0.2) node[text=black,anchor=base west] {$\Gal(2,3)$};
\path (8,0.2) node[text=black,anchor=base west] {$\Gal(3,2)$};

\draw[thick,black] (3.8,0.4) rectangle (5,1.6);
\draw[thick,black] (3.8,3.2) rectangle (5,4.4);
\draw[thick,black] (3.8,4.6) rectangle (5,5.8);
\draw[thick,black] (6,0.4) rectangle (7.2,1.2);
\draw[thick,black] (6,4.6) rectangle (7.2,5.4);
\path[draw=black,thick,fill=black!50] (6,1.8) -- (6,2.6) -- (6.4,2.6) -- (6.4,1.8) -- cycle;
\path[draw=black,thick,fill=black!50] (6,3.2) -- (6,4) -- (6.8,4) -- (6.8,3.2) -- cycle;
\path[draw=black,thick,fill=black!50] (6,4.6) -- (6,5.4) -- (7.2,5.4) -- (7.2,4.6) -- cycle;
\path[draw=black,thick,fill=black!50] (3.8,4.6) -- (3.8,5.8) -- (5,5.8) -- (5,4.6) -- cycle;
\path[draw=black,thick,fill=black!50] (3.8,1.8) -- (3.8,2.6) -- (4.6,2.6) -- (4.6,1.8) -- cycle;
\draw[thick,black] (3.8,1.8) rectangle (5,3);
\draw[thick,black] (6,1.8) rectangle (7.2,2.6);
\draw[thick,black] (6,3.2) rectangle (7.2,4);
\draw[thick,black] (2,0.4) rectangle (2.8,1.6);
\draw[thick,black] (2,1.8) rectangle (2.8,3);
\draw[thick,black] (2,3.2) rectangle (2.8,4.4);
\draw[thick,black] (2,4.6) rectangle (2.8,5.8);
\path[draw=black,thick,fill=black!50] (2,1.8) -- (2,2.2) -- (2.8,2.2) -- (2.8,1.8) -- cycle;
\path[draw=black,thick,fill=black!50] (2,3.2) -- (2,4) -- (2.8,4) -- (2.8,3.2) -- cycle;
\path[draw=black,thick,fill=black!50] (2,4.6) -- (2,5.8) -- (2.8,5.8) -- (2.8,4.6) -- cycle;
\draw[thick,black] (8.2,0.4) rectangle (9,1.6);
\draw[thick,black] (8.2,1.8) rectangle (9,3);
\draw[thick,black] (8.2,3.2) rectangle (9,4.4);
\draw[thick,black] (8.2,4.6) rectangle (9,5.8);
\path[draw=black,thick,fill=black!50] (8.2,1.8) -- (8.2,2.2) -- (9,2.2) -- (9,1.8) -- cycle;
\path[draw=black,thick,fill=black!50] (8.2,3.2) -- (8.2,4) -- (9,4) -- (9,3.2) -- cycle;
\path[draw=black,thick,fill=black!50] (8.2,4.6) -- (8.2,5.8) -- (9,5.8) -- (9,4.6) -- cycle;
\path[draw=black,thick,fill=black!50] (3.8,3.2) -- (3.8,4.4) -- (4.2,4.4) -- (4.2,3.6) -- (5,3.6) -- (5,3.2) -- cycle;

\end{tikzpicture}%
\end{center}
\caption{Maps $\barpush$, $\barres$ and $\barbord$ on diagrams,
corresponding to the maps $\iota_*$, $\upsilon^*$ and $\bord$ on
generators (no arrow means mapped to zero).}
\label{G33_fig}%
\end{figure}

\begin{exam}
Figures~\ref{G44_fig},~\ref{G45_fig} and~\ref{G55_fig} give the even
Young diagrams in \deframe\ and the corresponding shifts in $\bbZ/4$
and twists $\Det_d$ or $\cO$ for the Grassmannians $\Gal(4,4)$,
$\Gal(4,5)$ and $\Gal(5,5)$.
\end{exam}

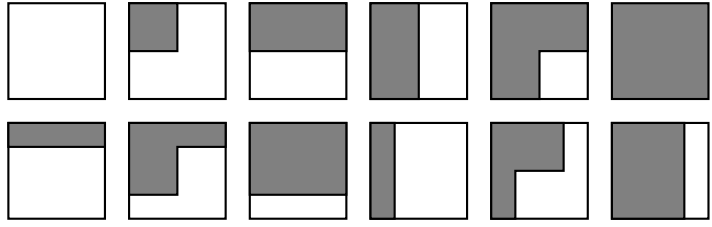
\begin{figure}[!ht]
\begin{center}
\begin{tikzpicture}[y=-1cm]

\draw[thick,black] (0.635,0.635) rectangle (1.905,1.905);
\draw[thick,black] (2.2225,0.635) rectangle (3.4925,1.905);
\draw[thick,black] (3.81,0.635) rectangle (5.08,1.905);
\draw[thick,black] (5.3975,0.635) rectangle (6.6675,1.905);
\draw[thick,black] (6.985,0.635) rectangle (8.255,1.905);
\path[draw=black,thick,fill=black!50] (8.5725,0.635) rectangle (9.8425,1.905);
\draw[thick,black] (0.635,2.2225) rectangle (1.905,3.4925);
\draw[thick,black] (2.2225,2.2225) rectangle (3.4925,3.4925);
\draw[thick,black] (3.81,2.2225) rectangle (5.08,3.4925);
\draw[thick,black] (5.3975,2.2225) rectangle (6.6675,3.4925);
\draw[thick,black] (6.985,2.2225) rectangle (8.255,3.4925);
\draw[thick,black] (8.5725,2.2225) rectangle (9.8425,3.4925);
\path[draw=black,thick,fill=black!50] (2.2225,0.635) -- (2.8575,0.635) -- (2.8575,1.27) -- (2.2225,1.27) -- cycle;
\path[draw=black,thick,fill=black!50] (3.81,1.27) -- (5.08,1.27) -- (5.08,0.635) -- (3.81,0.635) -- cycle;
\path[draw=black,thick,fill=black!50] (5.3975,0.635) -- (5.3975,1.905) -- (6.0325,1.905) -- (6.0325,0.635) -- cycle;
\path[draw=black,thick,fill=black!50] (6.985,0.635) -- (6.985,1.905) -- (7.62,1.905) -- (7.62,1.27) -- (8.255,1.27) -- (8.255,0.635) -- cycle;
\path[draw=black,thick,fill=black!50] (0.635,2.54) -- (1.905,2.54) -- (1.905,2.2225) -- (0.635,2.2225) -- cycle;
\path[draw=black,thick,fill=black!50] (2.2225,2.2225) -- (2.2225,3.175) -- (2.8575,3.175) -- (2.8575,2.54) -- (3.4925,2.54) -- (3.4925,2.2225) -- cycle;
\path[draw=black,thick,fill=black!50] (3.81,2.2225) -- (3.81,3.175) -- (5.08,3.175) -- (5.08,2.2225) -- cycle;
\path[draw=black,thick,fill=black!50] (5.3975,2.2225) -- (5.3975,3.4925) -- (5.715,3.4925) -- (5.715,2.2225) -- cycle;
\path[draw=black,thick,fill=black!50] (6.985,2.2225) -- (6.985,3.4925) -- (7.3025,3.4925) -- (7.3025,2.8575) -- (7.9375,2.8575) -- (7.9375,2.2225) -- cycle;
\path[draw=black,thick,fill=black!50] (8.5725,2.2225) -- (8.5725,3.4925) -- (9.525,3.4925) -- (9.525,2.2225) -- cycle;
\end{tikzpicture}
\end{center}
\caption{Diagrams of generators for $\Gal(4,4)$: first row in shift
$0$ and twist $\cO$, second row in shift $0$ and twist~$\Det_d$.}
\label{G44_fig}%
\end{figure}

\begin{figure}[!ht]
\begin{center}
\begin{tikzpicture}[y=-1cm]

\path[draw=black,thick,fill=black!50] (2.2225,0.3175) -- (2.8575,0.3175) -- (2.8575,0.9525) -- (2.2225,0.9525) -- cycle;
\path[draw=black,thick,fill=black!50] (4.1275,0.9525) -- (5.3975,0.9525) -- (5.3975,0.3175) -- (4.1275,0.3175) -- cycle;
\path[draw=black,thick,fill=black!50] (6.0325,0.3175) -- (6.0325,1.5875) -- (6.6675,1.5875) -- (6.6675,0.3175) -- cycle;
\path[draw=black,thick,fill=black!50] (7.9375,0.3175) -- (7.9375,1.5875) -- (8.5725,1.5875) -- (8.5725,0.9525) -- (9.2075,0.9525) -- (9.2075,0.3175) -- cycle;
\path[draw=black,thick,fill=black!50] (7.9375,1.905) -- (7.9375,3.175) -- (8.89,3.175) -- (8.89,2.54) -- (9.525,2.54) -- (9.525,1.905) -- cycle;
\draw[thick,black] (0.3175,0.3175) rectangle (1.905,1.5875);
\draw[thick,black] (2.2225,0.3175) rectangle (3.81,1.5875);
\draw[thick,black] (4.1275,0.3175) rectangle (5.715,1.5875);
\draw[thick,black] (6.0325,0.3175) rectangle (7.62,1.5875);
\draw[thick,black] (7.9375,0.3175) rectangle (9.525,1.5875);
\path[draw=black,thick,fill=black!50] (0.635,3.175) -- (0.635,1.905) -- (0.3175,1.905) -- (0.3175,3.175) -- cycle;
\draw[thick,black] (0.3175,1.905) rectangle (1.905,3.175);
\path[draw=black,thick,fill=black!50] (2.2225,1.905) -- (3.175,1.905) -- (3.175,2.54) -- (2.54,2.54) -- (2.54,3.175) -- (2.2225,3.175) -- cycle;
\draw[thick,black] (2.2225,1.905) rectangle (3.81,3.175);
\draw[thick,black] (7.9375,1.905) rectangle (9.525,3.175);
\path[draw=black,thick,fill=black!50] (9.8425,0.3175) -- (9.8425,1.5875) -- (11.1125,1.5875) -- (11.1125,0.3175) -- cycle;
\draw[thick,black] (9.8425,0.3175) rectangle (11.43,1.5875);
\path[draw=black,thick,fill=black!50] (9.8425,1.905) rectangle (11.43,3.175);
\path[draw=black,thick,fill=black!50] (4.1275,1.905) -- (5.715,1.905) -- (5.715,2.54) -- (4.445,2.54) -- (4.445,3.175) -- (4.1275,3.175) -- cycle;
\draw[thick,black] (4.1275,1.905) rectangle (5.715,3.175);
\draw[thick,black] (6.0325,1.905) rectangle (7.62,3.175);
\path[draw=black,thick,fill=black!50] (6.0325,3.175) -- (6.985,3.175) -- (6.985,1.905) -- (6.0325,1.905) -- cycle;
\end{tikzpicture}
\end{center}
\caption{Diagrams of generators for $\Gal(4,5)$: first row in shift
$0$ and twist $\cO$, second row in shift $0$ and twist~$\Det_d$.}
\label{G45_fig}%
\end{figure}
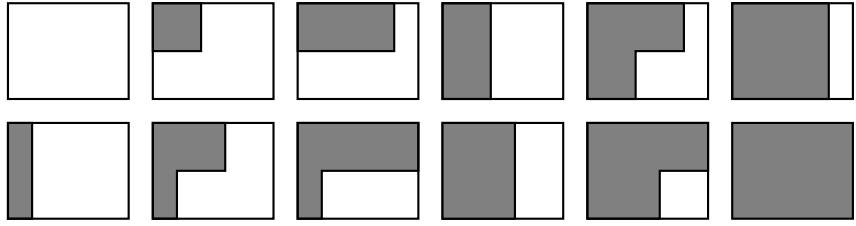

\begin{figure}[!ht]
\begin{center}
\begin{tikzpicture}[y=-1cm]

\path[draw=black,thick,fill=black!50] (2.2225,0.3175) -- (2.8575,0.3175) -- (2.8575,0.9525) -- (2.2225,0.9525) -- cycle;
\path[draw=black,thick,fill=black!50] (4.1275,0.9525) -- (5.3975,0.9525) -- (5.3975,0.3175) -- (4.1275,0.3175) -- cycle;
\path[draw=black,thick,fill=black!50] (6.0325,0.3175) -- (6.0325,1.5875) -- (6.6675,1.5875) -- (6.6675,0.3175) -- cycle;
\path[draw=black,thick,fill=black!50] (7.9375,0.3175) -- (7.9375,1.5875) -- (8.5725,1.5875) -- (8.5725,0.9525) -- (9.2075,0.9525) -- (9.2075,0.3175) -- cycle;
\path[draw=black,thick,fill=black!50] (9.8425,0.3175) -- (9.8425,1.5875) -- (11.1125,1.5875) -- (11.1125,0.3175) -- cycle;
\draw[thick,black] (0.3175,0.3175) rectangle (1.905,1.905);
\draw[thick,black] (2.2225,0.3175) rectangle (3.81,1.905);
\draw[thick,black] (4.1275,0.3175) rectangle (5.715,1.905);
\draw[thick,black] (6.0325,0.3175) rectangle (7.62,1.905);
\draw[thick,black] (7.9375,0.3175) rectangle (9.525,1.905);
\draw[thick,black] (9.8425,0.3175) rectangle (11.43,1.905);
\draw[thick,black] (2.2225,2.2225) rectangle (3.81,3.81);
\draw[thick,black] (4.1275,2.2225) rectangle (5.715,3.81);
\draw[thick,black] (6.0325,2.2225) rectangle (7.62,3.81);
\draw[thick,black] (7.9375,2.2225) rectangle (9.525,3.81);
\path[draw=black,thick,fill=black!50] (9.8425,2.2225) rectangle (11.43,3.81);
\path[draw=black,thick,fill=black!50] (7.9375,2.2225) -- (7.9375,3.81) -- (8.89,3.81) -- (8.89,3.175) -- (9.525,3.175) -- (9.525,2.2225) -- cycle;
\path[draw=black,thick,fill=black!50] (4.1275,2.2225) -- (5.715,2.2225) -- (5.715,3.175) -- (4.445,3.175) -- (4.445,3.81) -- (4.1275,3.81) -- cycle;
\path[draw=black,thick,fill=black!50] (0.635,3.81) -- (0.635,2.54) -- (1.905,2.54) -- (1.905,2.2225) -- (0.3175,2.2225) -- (0.3175,3.81) -- cycle;
\draw[thick,black] (0.3175,2.2225) rectangle (1.905,3.81);
\path[draw=black,thick,fill=black!50] (2.2225,2.2225) -- (3.81,2.2225) -- (3.81,2.54) -- (3.175,2.54) -- (3.175,3.175) -- (2.54,3.175) -- (2.54,3.81) -- (2.2225,3.81) -- cycle;
\path[draw=black,thick,fill=black!50] (6.0325,3.81) -- (6.985,3.81) -- (6.985,2.54) -- (7.62,2.54) -- (7.62,2.2225) -- (6.0325,2.2225) -- cycle;
\end{tikzpicture}
\end{center}
\caption{Diagrams of generators for $\Gal(5,5)$: first row in shift
$0$ and twist $\cO$, second row in shift $1$ and twist~$\Det_d$.}
\label{G55_fig}%
\end{figure}
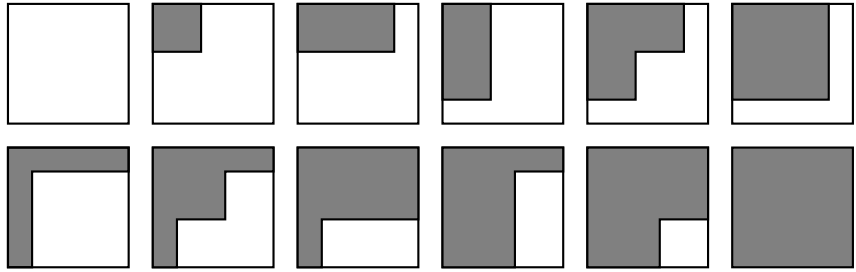

We conclude with a few comments.

\begin{rema}
As mentioned in Remark~\ref{ind_rem}, we could have considered a
larger set of elements $\gen_{d,e}(\Lambda)$ using
Remark~\ref{cond-push_rem} instead of assuming $\Lambda$ even. This
larger set is also stable by applying $\iota_*$, $\upsilon^*$ and
$\bord$. Some of these extra elements are then easily seen to be
zero from the exact sequence, but some of them are nonzero.
\end{rema}

\begin{rema}
Part of the multiplicative structure on Witt groups can be computed
at each induction step using the projection formula. Unfortunately,
this is not enough for the whole computation. Despite the results
for the Grothendieck and the Chow rings using basis of Schubert
cells, it is unclear to the authors what kind of
Littlewood-Richardson type rule one should expect.

Note however that all generators $\gen_{d,e}(\Lambda)$ are nilpotent
except for $\Lambda=\vide$. This can be checked using homotopy
invariance and a discussion on the supports or simply the fact that
these Witt classes are generically trivial.
\end{rema}

\begin{rema}
Although we do not need it here, it is possible to show that, for
$\cV=\cO_X^{d+e}$, the isomorphism $\Gal(d,e)=\Grass(d,\cV)\simeq
\Grass(e,\cV^\vee)=\Gal(e,d)$ sends $\gen_{d,e}(\Lambda)$ to
$\gen_{e,d}(\Lambda^\vee)$, up to \lax-similitude of course, where
$\Lambda^\vee$ is the dual partition.
\end{rema}

\bibliography{WittGrassmann}

\end{document}